\documentclass[11pt,twoside]{article}
\usepackage{amsmath,amssymb,amsthm}
\usepackage{amsmath,amssymb,amsthm}
\usepackage{graphicx,color,subfigure}
\usepackage{epstopdf}
\usepackage{txfonts}
\usepackage{cite}
\usepackage{caption}
\allowdisplaybreaks[4]

\newtheorem{lm}{Lemma}[section]
\newtheorem{thm}{Theorem}[section]

\newcounter{saveeqn}%

\makeatletter
\evensidemargin
\oddsidemargin
\makeatother

\makeatletter
\@addtoreset{equation}{section}
\makeatother

\title{\Large\bf Crossing-sliding bifurcations in planar $\mathbb{Z}_2$-symmetric Filippov systems \thanks{
This work is financially supported by the National Key R \& D Program of China (No. 29302022YFA1005900). The first author is supported by
NSFC (No. 12271378) and Sichuan Science and Technology Program (No. 2024NSFJQ0008). The third author is supported by
NSFC (No. 12201509).
}
}

\author{Xingwu Chen$^1$, ~~Jiahao Li$^1$,~~Tao Li$^2$~\!\!
\footnote{Corresponding author: Tao Li (litao@swufe.edu.cn)}
\\
{\small 1. School of Mathematics, Sichuan University, Chengdu, Sichuan 610064, P. R. China}\\
{\small 2. School of Mathematics, Southwestern University of Finance and Economics,}\\
  {\small Chengdu, Sichuan 611130, P. R. China}
  }
\date{}

\begin{document}
\maketitle

%%%%%%%%%%%%%%%%%%%%%%%%%%%%%%%%%%%
\begin{abstract}
In this paper we investigate the crossing-sliding bifurcations of planar Filippov systems with $\mathbb{Z}_2$-symmetry. Such bifurcations are triggered by the perturbations of a critical crossing cycle and constitute an important class of discontinuity-induced bifurcations. By constructing transition maps and developing a decomposition theorem of functions to overcome the difficulty of describing bifurcation boundaries in multi-parameter settings, we systematically characterize the codimension-one and codimension-two bifurcation scenarios through the explicit statement of non-degenerate conditions and the presentation of the corresponding bifurcation diagrams. The asymptotic properties of all bifurcation curves are also derived.

\end{abstract}

\baselineskip 15pt
\parskip 10pt
\thispagestyle{empty}
\setcounter{page}{1}

%%%%%%%%%%%%%%%%%%%%%%%%%%%%%%%%%%%%%%%%%%%%%%%%%%%%%%%%%%%%
\section{Introduction}
\setcounter{equation}{0}
\setcounter{lm}{0}
\setcounter{thm}{0}
\setcounter{rmk}{0}
\setcounter{df}{0}
\setcounter{cor}{0}

The emergence of non-smooth dynamical phenomena in real world challenges many fundamental assumptions in the classical theory of dynamical systems. This has motivated the extension of bifurcation theory from smooth systems to non-smooth ones, particularly discontinuous piecewise-smooth differential systems of the form
\begin{equation}\label{generalsystem}
(\dot x, \dot y)=\left\{
\begin{aligned}
&(f^+(x,y;\alpha),g^+(x,y;\alpha)),\qquad (x,y)\in\Sigma^+,\\
&(f^-(x,y;\alpha),g^-(x,y;\alpha)),\qquad (x,y)\in\Sigma^-,
\end{aligned}
\right.
\end{equation}
where $(x,y)\in\mathbb{R}^2$, $\alpha=(\alpha_1,\alpha_2,\cdots,\alpha_m)\in\mathbb{R}^m$, $f^\pm, g^\pm:\mathbb{R}^2\times\mathbb{R}^m\rightarrow\mathbb{R}^2$ are smooth with respect to variables and parameters,
$$\Sigma^+:=\left\{(x, y)\in\mathbb{R}^2: h(x,y;\alpha)>0\right\},\qquad \Sigma^-:=\left\{(x, y)\in\mathbb{R}^2: h(x,y;\alpha)<0\right\},$$
$h: \mathbb{R}^2\times\mathbb{R}^m\rightarrow\mathbb{R}$ is a smooth function having $0$ as a regular value, i.e. the gradient $(h_x,h_y)$ does not vanish on $\Sigma:=\left\{(x, y)\in\mathbb{R}^2: h(x,y;\alpha)=0\right\}$. Usually, $\Sigma$ is named {\it discontinuity boundary} (see e.g.\cite{MD}). In the context of Filippov's solution \cite{AFF}, system (\ref{generalsystem}) is also known as a {\it Filippov system}. Induced by discontinuity, system (\ref{generalsystem}) exhibits a rich spectrum of novel bifurcation phenomena (see e.g. \cite{YAK, MG1, MD}), which absent in smooth differential systems because they typically involve the interactions of an invariant set with the discontinuity boundary.

Among discontinuity-induced bifurcations, {\it sliding bifurcations} have attracted considerable attention as one of the most extensively studied categories. These bifurcations describe the dynamical transitions caused by the perturbations of {\it tangential periodic orbits}, i.e. periodic orbits with a point where the orbit of a subsystem is tangent to the discontinuity boundary. Such point is called a {\it tangent point}. Significant progress has been made in characterizing and understanding sliding bifurcations, thanks to contributions from many researchers. For instance, studies in \cite{YAK,MG1} have addressed the classification of codimension-one sliding bifurcations of system (\ref{generalsystem}), obtaining four distinct types: {\it grazing-sliding, crossing-sliding, multi-sliding}, and {\it switching-sliding bifurcations}; \cite{FLMH} has given the stability of some critical configurations in sliding bifurcations; works in \cite{AC,FLMHXZ,LTXC,FLMHsfsf,CFL} and \cite{AGN,EFEPFT,NTZ,WHW,FC,FC2} have investigated the grazing-sliding bifurcations and crossing-sliding bifurcations of system (\ref{generalsystem}), respectively, focusing on aspects such as the establishment of bifurcation diagrams and the identification of the maximum number of limit cycles emerging from the bifurcations; sliding bifurcations in high-dimensional Filippov systems are studied in \cite{AC,MD,MBPK,JH,ATTRACTOR1,ATTRACTOR2,ABN,ATTRACTOR3,RSHO} and references therein.

The main interest of this paper focuses on the crossing-sliding bifurcations of system (\ref{generalsystem}), also called critical crossing cycle bifurcations \cite{EFEPFT}. These occur when a {\it critical crossing cycle} (see Section 2 for a detail definition) is perturbed. In Section 4.1.4 of \cite{YAK}, it is claimed that there are two possible codimension-one crossing-sliding bifurcations in system (\ref{generalsystem}) with $\alpha\in\mathbb{R}$ if there is a unique tangent point in the critical crossing cycle and it is a {\it regular-fold}, namely at which one vector field exhibits quadratic tangency to $\Sigma$ while the other maintains transversal intersection. The first case is the one depicted in Figure~\ref{generalcodim-1bifurdia}, where the critical crossing cycle existing for $\alpha=0$ bifurcates into a sliding cycle or a crossing cycle as $\alpha$ varies. In the second case, either a sliding cycle and a standard cycle simultaneously bifurcate from the critical crossing cycle, or the critical crossing cycle disappears without giving rise to any new cycles.
However, subsequent work by \cite{EFEPFT} proves that the second case cannot occur.
Building on \cite{YAK, EFEPFT}, further researches have considered the degeneracies in which the intersections between the critical crossing cycle and $\Sigma$ in the codimension-one crossing-sliding bifurcation become tangent points of high multiplicity (see \cite{FC2} for the definition), obtaining several codimension-two bifurcation diagrams of system (\ref{generalsystem}). For instance, \cite{NTZ} (resp. \cite{AGN}) analyzes the case where the tangent point $T_0$ becomes a {\it fold-fold} (resp. {\it regular-cusp}), while \cite{WHW, AGN} studies the case where the crossing point $P$ becomes a regular-fold. Here a fold-fold and a regular-cusp represent two specific types of tangent points of high multiplicity, whose detail definitions are provided in Section 2. Besides, \cite{FC} studies the perturbations of a critical crossing cycle composed of two critical crossing cycles connecting the same fold-fold, giving the sum of the maximum numbers of bifurcating crossing and sliding cycles; \cite{FLMH} derives the stability of some kinds of critical crossing cycles. In general, as the multiplicity of tangent points increases, so does the codimension of the bifurcation, raising the theoretical research difficulty and resulting in relatively scarce analytical results. In \cite{AGN}, it is proved that the maximum number of crossing cycles bifurcating from a critical crossing cycle with one or multiple tangent points of any multiplicity
  is $1$ if the tangent points are regular ones for one of subsystems. Recently, the investigation in \cite{FC2} further examines the perturbations of a critical crossing cycle with one tangent point by allowing the tangent point is of any multiplicity for each subsystem, determining how this multiplicity affects the maximum numbers of bifurcating crossing and sliding cycles.
\begin{figure}
  \begin{minipage}[t]{1.0\linewidth}
  \centering
  \includegraphics[width=6.3in]{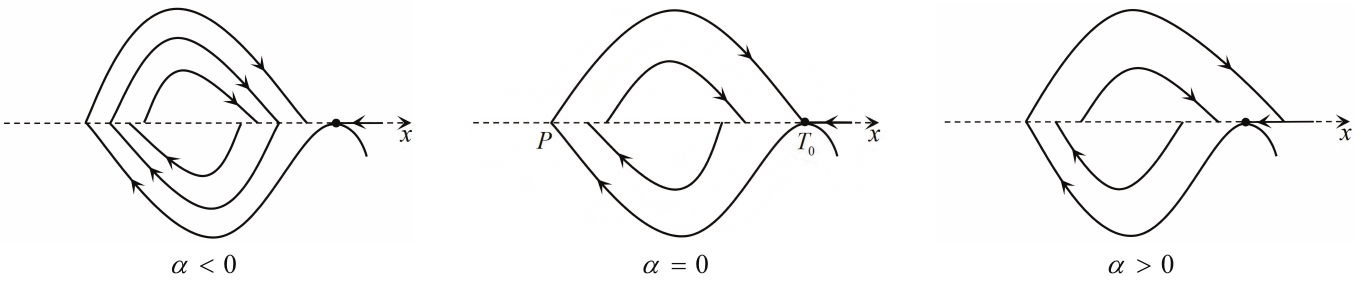}
  \end{minipage}
\caption{{\small Codim-1 crossing-sliding bifurcation of system (\ref{generalsystem}).}}
\label{generalcodim-1bifurdia}
\end{figure}

In this paper, based on previous studies, we extend the investigation of crossing-sliding bifurcations to system (\ref{generalsystem}) possessing $\mathbb{Z}_2$-symmetry with respect to the origin $O$, i.e.
\begin{equation}\label{Z2system}
(\dot x, \dot y)=\left\{
\begin{aligned}
  &(f^+(x,y;\alpha),g^+(x,y;\alpha)),
~~~~&&(x, y)\in\Sigma^+,\\
  &(-f^+(-x,-y;\alpha),-g^+(-x,-y;\alpha)),
~~~~&&(x, y)\in\Sigma^-
\end{aligned}
\right.
\end{equation}
with the discontinuity boundary $\Sigma$ satisfying $h(-x,-y;\alpha)=-h(x,y;\alpha)$. System (\ref{Z2system}) is widely used in scientific and engineering disciplines. Special cases of (\ref{Z2system}) serve as mathematical models for mechanical systems with Coulomb friction \cite{AAK}, discontinuous systems of SD oscillators \cite{CWPGT, TCY}, and control systems with relay mechanism \cite{L, BJV, GP}. Significantly, the results of \cite{GP} provide an evidence that critical crossing cycles can emerge in mathematical models of practical problems which take some special form of system (\ref{Z2system}). Motivated by these reasons, in this paper we conduct a theoretical investigation of the crossing-sliding bifurcations in generic system (\ref{Z2system}). Our goal is to symmetrically characterize the codimension-one and codimension-two bifurcation scenarios through the explicit statement of non-degenerate conditions and the presentation of the corresponding bifurcation diagrams.

As we will see, the crossing-sliding bifurcations of system (\ref{Z2system}) involve crossing cycles, critical crossing cycles, sliding cycles, tangent-equilibrium connections and tangent-tangent connections, see Section 2 for their definitions. Generally speaking, it is difficult to use a unified map to identify different types of cycles and connections, especially for codimension-2 bifurcations. This leads to a challenge: after using different maps to obtain different types of cycles and connections, how can we determine their coexistence? To address this, we develop a function decomposition theorem to describe the bifurcation boundaries, thereby precisely determining the regions where various cycles and connections exist, and ultimately establishing complete bifurcation diagrams for the codimension-one and codimension-two bifurcation scenarios in the sense of $\Sigma$-equivalence \cite{MG1}. Moreover, the asymptotic properties of all bifurcation curves are also derived. Our results not only provide a theoretical tool for interpreting and predicting crossing-sliding bifurcation phenomena in practical problems described by system (\ref{Z2system}), but also extend the investigation of sliding bifurcations in system (\ref{Z2system}) beyond the work of \cite{CFL}, which establishes the bifurcation diagram of a grazing-sliding bifurcation in system (\ref{Z2system}).
Furthermore, since certain configurations of critical crossing cycles considered in this paper can arise from high-codimensional
collision bifurcations of tangent points in system (\ref{Z2system}), our results also contribute to understanding the complex dynamical behaviors associated with such bifurcations.

This paper is organized as follows. In Section 2 we shortly review some basic notions on Filippov systems involved in this paper. In Section 3 we set up the problem of this paper and then state main theorems. After introducing several preliminary lemmas in Section 4, we give the proofs of the main theorems in Sections 5, 6 and 7.

%%%%%%%%%%%%%%%%%%%%%%%%%%%%%%%%
\section{Basic notions}

In this section we give a short review on some basic notions related to the Filippov system (\ref{generalsystem}).
Due to the discontinuity on $\Sigma$, we must adopt new method to define the solutions that reach $\Sigma$ at some time. A used widely rule is the Filippov's convention \cite{AFF}. According to this convention, $\Sigma$ is separated into the {\it crossing set}
$\Sigma^c:=\left\{(x, y)\in\Sigma: Z^+h(x,y;\alpha)\cdot Z^-h(x,y;\alpha)>0\right\}$
and the {\it sliding set}
$\Sigma^s:=\left\{(x, y)\in\Sigma: Z^+h(x,y;\alpha)\cdot Z^-h(x,y;\alpha)\le0\right\}$
as in \cite{YAK}, where $Z^\pm=(f^\pm,g^\pm)$, $Z^\pm h=\langle Z^\pm,\nabla h\rangle$ and $\langle \cdot,\cdot\rangle$ denotes inner product.
Moreover, a sliding segment in the interior of $\Sigma^s$ is said to be {\it stable} if $Z^+h(x,y;\alpha)<0<Z^-h(x,y;\alpha)$ and {\it unstable} if $Z^+h(x,y;\alpha)>0>Z^-h(x,y;\alpha)$.

Since both vector fields $Z^+$ and $Z^-$ are transversal to $\Sigma$ with the same-sign normal components on $\Sigma^c$, the solution reaching $\Sigma$ at a point in $\Sigma^c$ will cross $\Sigma$. On $\Sigma^s$, either both $Z^+$ and $Z^-$ are transversal to $\Sigma$ with the opposite-sign normal components or at least one of normal components is zero. In this case, by the Filippov convex method \cite{AFF}, the solution reaching $\Sigma$ at a point in $\Sigma^s$ is allowed to slide along $\Sigma^s$, and the sliding dynamics obeys the differential system
\begin{equation}\label{generalslidd}
(\dot x,\dot y)=\mu Z^-(x,y;\alpha)+(1-\mu)Z^+(x,y;\alpha),\qquad (x,y)\in\Sigma^s,
\end{equation}
where $\mu\in[0,1]$ is selected to ensure that the vector field of (\ref{generalslidd}) is tangent to $\Sigma$, i.e.
$$\mu Z^-h(x,y;\alpha)+(1-\mu)Z^+h(x,y;\alpha)=0,\qquad (x,y)\in\Sigma^s.$$
In particular, if $Z^-h(x,y;\alpha)-Z^+h(x,y;\alpha)\ne0$, then $\mu={Z^+h(x,y;\alpha)}/({Z^+h(x,y;\alpha)-Z^-h(x,y;\alpha)})$. An equilibrium of (\ref{generalslidd}) in the interior of $\Sigma^s$ is called a {\it pseudo-equilibrium} of (\ref{generalsystem}) as in \cite{YAK}. In these settings, the solutions of (\ref{generalsystem}) that do interact with $\Sigma$ can be constructed by concatenating the standard solutions in $\Sigma^\pm$ and the sliding solutions in $\Sigma$, see \cite{YAK} for more details.

As in \cite{MG1,YAK}, a point $p$ in
$\partial\Sigma^s:=\left\{(x, y)\in\Sigma: Z^+h(x,y;\alpha)\cdot Z^-h(x,y;\alpha)=0\right\}$
can be classified into a {\it boundary equilibrium} of the subsystem in $\Sigma^\pm$ if $Z^\pm(p;\alpha)=0$ and a {\it tangent point} of the subsystem in $\Sigma^\pm$ if $Z^\pm(p;\alpha)\ne0$ and $Z^\pm h(p;\alpha)=0$. A tangent point $p$ of the subsystem in $\Sigma^\pm$ is called a {\it fold} (resp. {\it cusp}) if $(Z^\pm)^2h(p;\alpha)\ne0$ (resp. $(Z^\pm)^2h(p;\alpha)=0$ and $(Z^\pm)^3h(p;\alpha)\ne0$), where $(Z^\pm)^2h$ and $(Z^\pm)^3h$ denote the second-order and the third-order Lie derivatives respectively. In addition, a fold $p$ of the subsystem in $\Sigma^+$ (resp. $\Sigma^-$) is said to be {\it visible} if $(Z^+)^2h(p;\alpha)>0$ (resp. $(Z^-)^2h(p;\alpha)<0$). Reversing the inequality, we can get the definition of an {\it invisible} fold.
Finally, $p\in\partial\Sigma^s$ is called a {\it regular-fold} (resp. {\it regular-cusp}) of (\ref{generalsystem}) if it is a regular point of one subsystem and a fold (resp. cusp) of the other. A {\it fold-fold} of (\ref{generalsystem}) means that it is a fold of both subsystems simultaneously.

In addition to the {\it standard periodic orbits} that entirely lie in $\Sigma^+$ or $\Sigma^-$, system (\ref{generalsystem}) possesses some novel periodic orbits that do not exist in smooth systems (cf. \cite{YAK, MG1, CFL, AGN}), including {\it crossing periodic orbits} and
{\it tangential periodic orbits}. Here a crossing periodic orbit is formed by concatenating the regular orbits of two subsystems only at some points of $\Sigma^c$, and a tangential periodic orbit contains at least one tangent point. Moreover, tangential periodic orbits are classified into
the following three kinds.
\vspace{-12pt}
\begin{itemize}
\setlength{\itemsep}{0mm}
\item[(i)] {\it sliding periodic orbit}, which contains a sliding segment with non-zero length. A sliding periodic orbit is said to be {\it stable} (resp. {\it unstable}) if its sliding segment is stable (resp. unstable).
\item[(ii)] {\it critical crossing periodic orbit}, which occupies $\Sigma^+$ and $\Sigma^-$ and intersects $\Sigma$ only at some points of $\Sigma^c\cup\partial\Sigma^s$.
    \item[(iii)] {\it grazing periodic orbit}, which lies totally in $\Sigma^+\cup\Sigma$ or $\Sigma^-\cup\Sigma$ and intersects $\Sigma$ only at tangent points. Clearly,
 a grazing periodic orbit must be a periodic orbit of a subsystem.
\end{itemize}
\vspace{-12pt}
An isolated standard (resp. crossing, sliding, critical crossing, and grazing) periodic orbit of system (\ref{generalsystem}) in the set of all periodic orbits is called a {\it standard} (resp. {\it crossing, sliding, critical crossing}, and {\it grazing}) cycle.

As in smooth systems, we still refer to an orbit of system (\ref{generalsystem}) connecting one equilibrium (resp. two equilibria) as a {\it homoclinic} (resp. {\it heteroclinic}) {\it connection} or {\it homoclinic} (resp. {\it heteroclinic}) {\it orbit}, although here we allow the equilibrium to be a pseudo-equilibrium or a boundary equilibrium, and the orbit to be formed by concatenating the regular orbits of two subsystems and system (\ref{generalslidd}).
In particular, if a homoclinic orbit contains a sliding segment, it is called a {\it sliding homoclinic orbit}.
Besides, we refer to an orbit of system (\ref{generalsystem}) connecting a tangent point and an equilibrium (resp. a tangent point) as a {\it tangent-equilibrium connection} (resp. {\it tangent-tangent connection}), which plays an important role in the bifurcation analysis of system (\ref{generalsystem}).

%%%%%%%%%%%%%%%%%%%%%%%%%%%%%%%%%%%%%%%%%%%%%%%%%%%%%%%%%%%%
\section{Main results}
\setcounter{equation}{0}
\setcounter{lm}{0}
\setcounter{thm}{0}
\setcounter{rmk}{0}
\setcounter{df}{0}
\setcounter{cor}{0}

As introduced in Section 1, system (\ref{Z2system}) is widely used in scientific and engineering disciplines (see e.g. \cite{AAK,CWPGT, TCY,L, BJV}) and the results of \cite{GP} provide an evidence that critical crossing cycles can emerge in mathematical models of practical problems which take some special form of system (\ref{Z2system}). Hence, motivated by these reasons, a meaningful research objective is to establish theoretical results on the crossing-sliding bifurcations in generic system (\ref{Z2system}), thereby revealing how the dynamics near a critical crossing cycle evolve with parameters. To this end, we next formulate the problem precisely.

As will be seen, it is only near two $\mathbb{Z}_2$-symmetric points in the discontinuity boundary $\Sigma$ that $\Sigma$ is relevant to our problem. Therefore, without loss of generality, we always take from now on that $\Sigma$ is the $x$-axis, i.e. $h(x,y;\alpha)=y$. In such setting, we call the subsystem in $\Sigma^+$ (resp. $\Sigma^-$) {\it the upper subsystem} (resp. {\it lower subsystem}) of system (\ref{Z2system}). Besides, we always require that each subsystem is $C^k$ $(k\ge3)$.
Consider the following basic assumption:
\begin{description}
\setlength{\itemsep}{0mm}
\item[(H0)] There exist constants $a>0$ and $\tau_0>0$ such that $\gamma_0^+(\tau_0)=(a,0)$ and $\{\gamma_0^+(t):0<t<\tau_0\}\subset\Sigma^+$, where $\gamma_0^+(t)$ is the solution of the upper subsystem of $(\ref{Z2system})$ with $\alpha=0$ satisfying $\gamma_0^+(0)=(-a,0)$.
\setlength{\itemsep}{0mm}
\end{description}
Under this assumption, the unperturbed system, i.e. system (\ref{Z2system}) with $\alpha=0$, has a periodic orbit $\Gamma_0$ formed by concatenating $\gamma^+_0:=\{\gamma_0^+(t):0\le t\le\tau_0\}$ and its $\mathbb{Z}_2$-symmetric counterpart at $(-a,0)$ and $(a,0)$. Thus our problem can be precisely formulated as {\it investigating how the dynamics of system $(\ref{Z2system})$ near $\Gamma_0$ evolve with $\alpha$ when $\Gamma_0$ is a critical crossing cycle}, namely $\Gamma_0$ is an isolated periodic orbit and at least one of $(-a,0)$ and $(a,0)$ is a tangent point of the upper subsystem with $\alpha=0$.
It is worth noting that we can do the change $(x,y,t)\rightarrow(-x,y,-t)$ if $(a,0)$ is a tangent point of the upper subsystem with $\alpha=0$. Thus, without loss of generality, it is enough to consider the case where
$(-a,0)$ is a tangent point of the upper subsystem with $\alpha=0$, i.e. $g^+(-a,0;0)=0$ and $f^+(-a,0;0)\ne0$.

In Subsection 3.1 (resp. 3.2) we state the non-degenerate conditions that define codimension-one (resp. codimension-two) crossing-sliding bifurcations of system (\ref{Z2system}) and then give the corresponding bifurcation diagrams. Here the bifurcation codimension refers to the minimum number of parameters required to unfold all bifurcation phenomena. To do this, we introduce the following notations:
\begin{equation}\label{notationsdefine}
\begin{aligned}
\lambda(t)&:=\exp\left(\int^{\tau_0}_t(f^+_x+g^+_y)(\gamma^+_0(s);0)ds\right),\qquad t\in[0,\tau_0],\\
\kappa_i&:=\int^{\tau_0}_0\lambda(t)(f^+g^+_{\alpha_i}-g^+f^+_{\alpha_i})(\gamma^+_0(t);0)dt,\qquad i=1,2,\cdots,m,
\end{aligned}
\end{equation}
where the subscripts $x, y$ and $\alpha_i$ denote the corresponding partial derivatives.

%%%%%%%%%%%%%%%%%%%%%%%%%%%%%%%%%%%%%%%%%%%%%%%%%%%%%%%%%%%%
\subsection{Codimension-one crossing-sliding bifurcations}\label{codimenonesubsection}
Under the assumption {\bf (H0)}, we further require that the following non-degenerate conditions hold.
\begin{description}
\setlength{\itemsep}{0mm}
\item[(H1)]  $(-a,0)$ is a fold of the upper subsystem of $(\ref{Z2system})$ with $\alpha=0$ satisfying
    $$
    g^+(-a,0;0)=0,\qquad f^+(-a,0;0)g^+_{x}(-a,0;0)>0.
    $$
\item[(H2)] $(a,0)$ is a regular point of the upper subsystem of $(\ref{Z2system})$ with $\alpha=0$ satisfying
    $
    g^+(a,0;0)<0.
    $
\setlength{\itemsep}{0mm}
\end{description}
Due to the $\mathbb{Z}_2$-symmetry, these conditions imply that $(-a,0)$ and $(a,0)$ are regular-folds of system (\ref{Z2system}) with $\alpha=0$, and that $\Gamma_0$ is a critical crossing cycle intersecting $\Sigma$ exactly at the two points. The following theorem describes the bifurcation phenomena resulting from $\Gamma_0$.
\begin{figure}
  \begin{minipage}[t]{1.0\linewidth}
  \centering
  \includegraphics[width=6.3in]{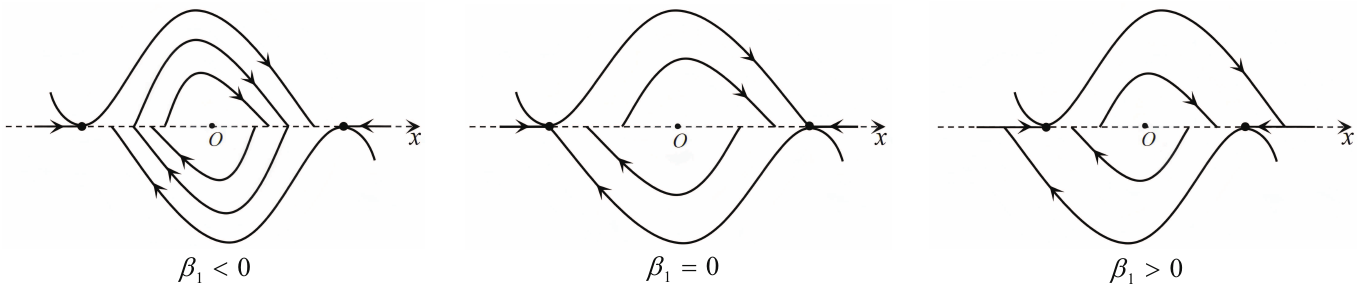}
  \end{minipage}
\caption{{\small Codim-1 crossing-sliding bifurcation of system (\ref{Z2system}) when $f^+(-a,0;0)>0$.}}
\label{codim-1bifurdia1}
\end{figure}
\begin{thm}\label{codim-1-bifur}
Assume that system $(\ref{Z2system})$ with $\alpha=0$ has a critical crossing cycle $\Gamma_0$ characterized by {\bf(H0)}, {\bf(H1)} and {\bf(H2)}.
If $(\theta_1,\theta_2,\cdots,\theta_m)\ne0$, where
\begin{equation}\label{3892fcsfvdfg}
\theta_i=-\frac{\kappa_i}{g^+(a,0;0)}-\frac{g^+_{\alpha_i}(-a,0;0)}{g^+_x(-a,0;0)},
\end{equation}
then for any sufficiently small annulus $\mathcal{A}$ of $\Gamma_0$, there exists a neighborhood $U$ of $\alpha=0$ and a locally diffeomorphism $\beta=(\beta_1,\beta_2,\cdots,\beta_m)=(\rho_1(\alpha),\rho_2(\alpha),\cdots,\rho_m(\alpha))$ with $\rho_i(0)=0$ from $U$ to its range $V$ such that for any $(\beta_2^*,\beta_3^*,\cdots,\beta_m^*)\in V^*$ the bifurcation diagram of system $\left.(\ref{Z2system})\right|_{\mathcal{A}}$ on the hyperplane $(\beta_2,\beta_3,\cdots,\beta_m)=(\beta_2^*,\beta_3^*,\cdots,\beta_m^*)$ is the one shown in Figure~\ref{codim-1bifurdia1} $($resp. Figure~\ref{codim-1bifurdia2}$)$ when $f^+(-a,0;0)>0$ $($resp. $<0$$)$, where $V^*$ is the set satisfying that $(0,\beta_2,\cdots,\beta_m)\in V$ for $(\beta_2,\cdots,\beta_m)\in V^*$.
\end{thm}

\begin{figure}[htp]
  \begin{minipage}[t]{1.0\linewidth}
  \centering
  \includegraphics[width=6.3in]{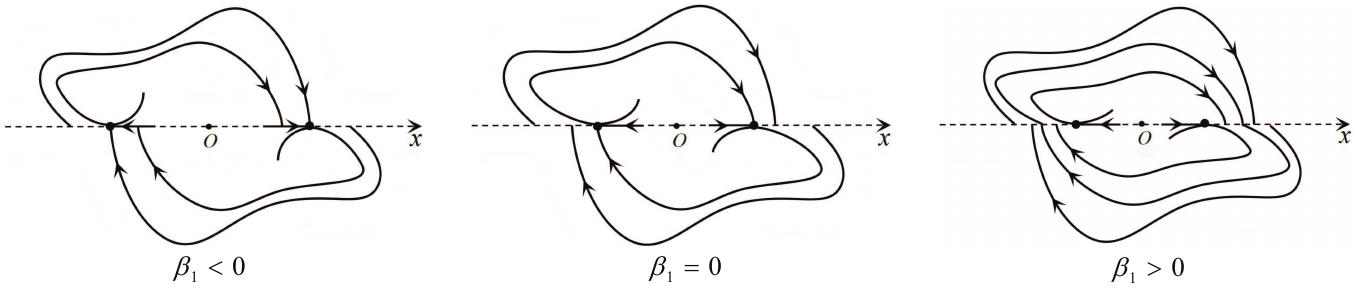}
  \end{minipage}
\caption{{\small Codim-1 crossing-sliding bifurcation of system (\ref{Z2system}) when $f^+(-a,0;0)<0$.}}
\label{codim-1bifurdia2}
\end{figure}

%%%%%%%%%%%%%%%%%%%%%%%%%%%%%%%%%%%%%%%%%%%%%%%%%%%%%%%%%%%%
\subsection{Codimension-two crossing-sliding bifurcations}

According to the non-degenerate conditions that define the codimension-one crossing-sliding bifurcations in Subsection~\ref{codimenonesubsection}, a codimension-two crossing-sliding bifurcation of system $(\ref{Z2system})$ occurs only if $(-a,0)$ degenerates to be a cusp of the upper subsystem with $\alpha=0$, or $(a,0)$ degenerates to be a fold of the upper subsystem with $\alpha=0$. In this paper, our attention will be focused on the scenarios where a first return map is well-defined (at least for the one side of $\Gamma_0$). Under this requirement, we can obtain an preliminary classification for codimension-two crossing-sliding bifurcations of system $(\ref{Z2system})$ as follows:
\begin{description}
\item[Case 1:] {\bf(H2)} is remained, but {\bf(H1)} is replaced by
\begin{description}
\vspace{-5pt}
\item[~~~~~(H1)$^\prime$] $(-a,0)$ is a cusp of the upper subsystem of $(\ref{Z2system})$ with $\alpha=0$ satisfying
$$f^+(-a,0;0)\ne0,\qquad g^+(-a,0;0)=0, \qquad g^+_{x}(-a,0;0)=0,\qquad g^+_{xx}(-a,0;0)>0.$$
\end{description}
\item[Case 2:] {\bf(H1)} is remained, but {\bf(H2)} is replaced by
\begin{description}
\vspace{-5pt}
\item[~~~~~(H2)$^\prime$] $(a,0)$ is a fold of the upper subsystem of $(\ref{Z2system})$ with $\alpha=0$ satisfying
    $$g^+(a,0;0)=0,\qquad f^+(a,0;0)g^+_x(a,0;0)>0,\qquad f^+(-a,0;0)f^+(a,0;0)>0.$$
\end{description}
\end{description}
For each case, $\Gamma_0$ is still a critical crossing cycle of system (\ref{Z2system}) with $\alpha=0$.
However, the intersections of it and $\Sigma$ are now regular-cusps (resp. fold-folds) instead of regular-folds for Case 1 (resp. Case 2).

Letting
\begin{equation}\label{ew45fsd}
\begin{aligned}
\zeta_i:=-\frac{2g^+_{\alpha_i}(-a,0;0)}{g^+_{xx}(-a,0;0)},\qquad \eta_i:=-\frac{\kappa_i}{g^+(a,0;0)}-\frac{g^+_{x\alpha_i}(-a,0;0)}{g^+_{xx}(-a,0;0)},
\end{aligned}
\end{equation}
$i=1,2,\cdots,m$, we have the following bifurcation result for Case 1.

\begin{thm}\label{codim-2-cusp-bifur}
Assume that system $(\ref{Z2system})$ with $\alpha=0$ has a critical crossing cycle $\Gamma_0$ characterized by {\bf(H0)}, {\bf(H1)$^\prime$} and {\bf(H2)}. If $(\zeta_1,\zeta_2,\cdots,\zeta_m)$ and $(\eta_1,\eta_2,\cdots,\eta_m)$ is linearly independent, then for any sufficiently small annulus $\mathcal{A}$ of $\Gamma_0$, there exists a neighborhood $U$ of $\alpha=0$ and a locally diffeomorphism $\beta=(\beta_1,\beta_2,\cdots,\beta_m)=(\phi_1(\alpha),\phi_2(\alpha),\cdots,\phi_m(\alpha))$ with $\phi_i(0)=0$ from $U$ to its range $V$ such that for any $(\beta_3^*,\beta_4^*,\cdots,\beta_m^*)\in V^*$ the bifurcation diagram of system $\left.(\ref{Z2system})\right|_{\mathcal{A}}$ on the hyperplane $(\beta_3,\beta_4,\cdots,\beta_m)=(\beta_3^*,\beta_4^*,\cdots,\beta_m^*)$
is the one shown in Figure~\ref{codim-2bifurdia2wca} $($resp. Figure~\ref{codim-2bifurdia2wca345}$)$ when $f^+(-a,0;0)>0$ $($resp. $<0$$)$, where $V^*$ is the set satisfying that $(0,0,\beta_3,\cdots,\beta_m)\in V$ for $(\beta_3,\cdots,\beta_m)\in V^*$, and all bifurcation curves except the axes are quadratically tangent to the $\beta_2$-axis at $(\beta_1,\beta_2)=(0,0)$.
\begin{figure}[h]
  \begin{minipage}[t]{1.0\linewidth}
  \centering
  \includegraphics[width=5.3in]{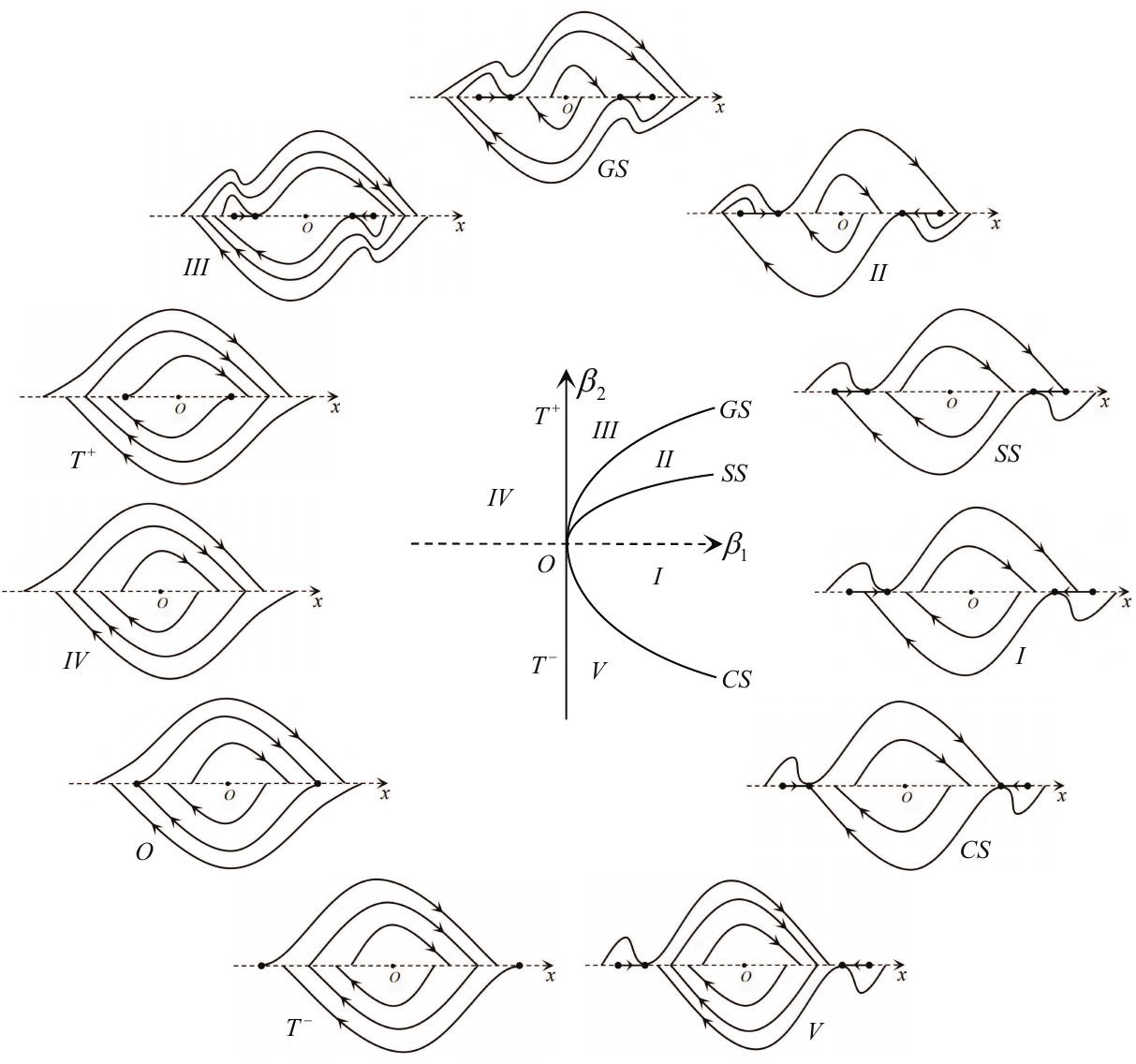}
  \end{minipage}
\caption{{\small Codim-2 crossing-sliding bifurcation of regular-cusp type of system (\ref{Z2system}) when $f^+(-a,0;0)>0$.}}
\label{codim-2bifurdia2wca}
\end{figure}
\begin{figure}[h]
  \begin{minipage}[t]{1.0\linewidth}
  \centering
  \includegraphics[width=5.3in]{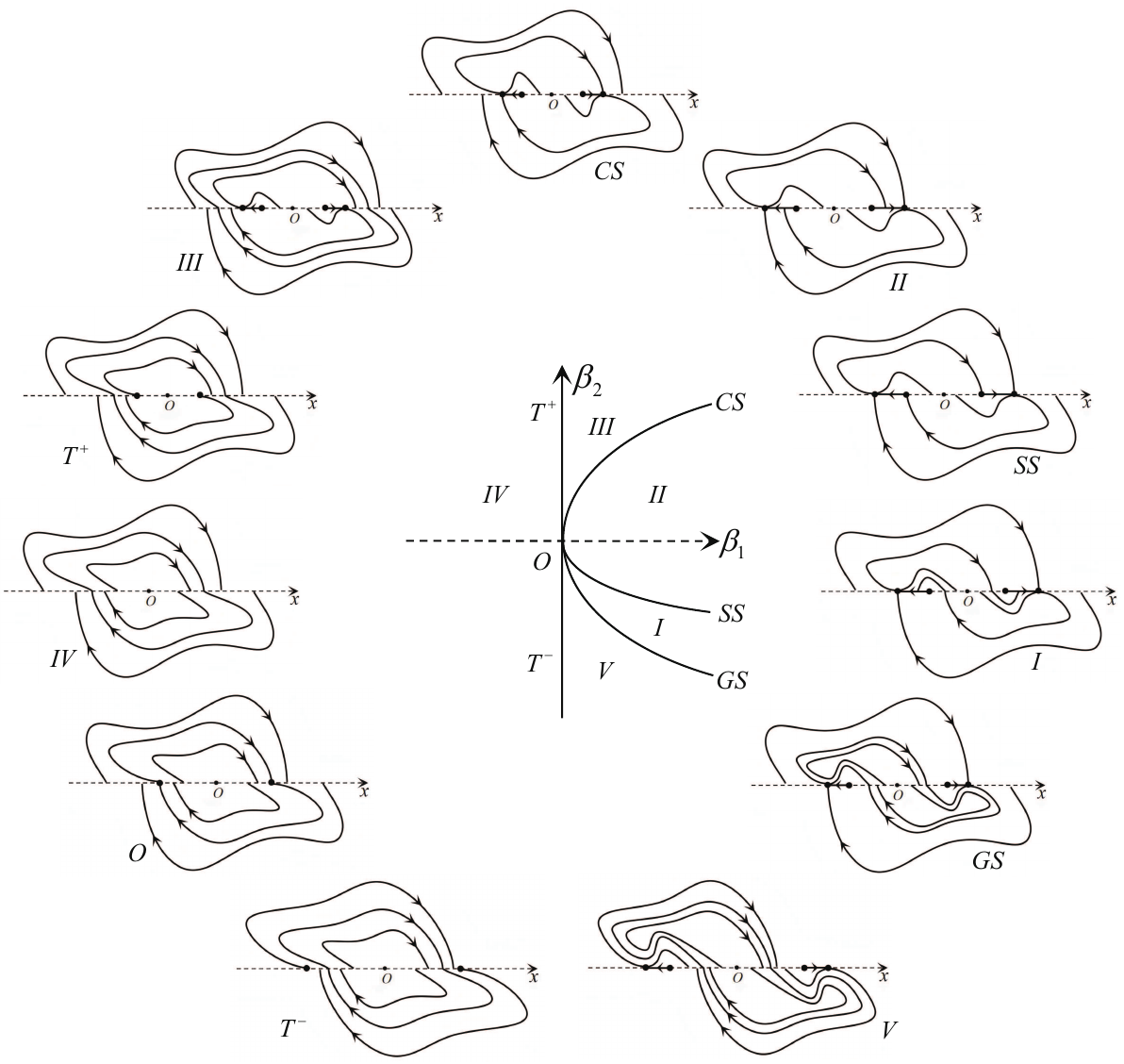}
  \end{minipage}
\caption{{\small Codim-2 crossing-sliding bifurcation of regular-cusp type of system (\ref{Z2system}) when $f^+(-a,0;0)<0$.}}
\label{codim-2bifurdia2wca345}
\end{figure}
\end{thm}

Letting
\begin{equation}\label{4scsrve4}
\mu_i:=-\frac{g^+_{\alpha_i}(-a,0;0)}{g^+_x(-a,0;0)}-
\frac{g^+_{\alpha_i}(a,0;0)}{g^+_x(a,0;0)},
\end{equation}
$i=1,2,\cdots,m$,
under an additional
 non-degenerate condition to trigger a codimension-two bifurcation:
 \begin{description}
\setlength{\itemsep}{0mm}
\item[(H3)] $\Delta:=g^+_x(-a,0;0)\lambda(0)-g^+_x(a,0;0)>0$, where $\lambda(t)$ is defined in (\ref{notationsdefine}),
\end{description}
we have the following bifurcation result for Case 2

\begin{thm}\label{codim-2-fold-bifur1}
Assume that system $(\ref{Z2system})$ with $\alpha=0$ has a critical crossing cycle $\Gamma_0$ characterized by {\bf(H0)}, {\bf(H1)}, {\bf(H2)$^\prime$} and {\bf(H3)}. If $(\mu_1,\mu_2,\cdots,\mu_m)$ and $(\kappa_1,\kappa_2,\cdots,\kappa_m)$ are linearly independent, then for any sufficiently small annulus $\mathcal{A}$ of $\Gamma_0$, there exists a neighborhood $U$ of $\alpha=0$ and a locally diffeomorphism $\beta=(\beta_1,\beta_2,\cdots,\beta_m)=(\varphi_1(\alpha),\varphi_2(\alpha),\cdots,\varphi_m(\alpha))$ with $\varphi_i(0)=0$ from $U$ to its range $V$ such that for any $(\beta_3^*,\beta_4^*,\cdots,\beta_m^*)\in V^*$ the bifurcation diagram of system $\left.(\ref{Z2system})\right|_{\mathcal{A}}$ on the hyperplane $(\beta_3,\beta_4,\cdots,\beta_m)=(\beta_3^*,\beta_4^*,\cdots,\beta_m^*)$ is the one shown in Figure~\ref{codim-2bifurdia2folddia} $($resp. Figure~\ref{codim-2bifurdia2folddia2334}$)$ when $f^+(-a,0;0)>0$ $($resp. $<0)$, where $V^*$ is the set satisfying that $(0,0,\beta_3,\cdots,\beta_m)\in V$ for $(\beta_3,\cdots,\beta_m)\in V^*$, and all bifurcation curves except the axes are quadratically tangent to the $\beta_1$-axis at $(\beta_1,\beta_2)=(0,0)$.
\begin{figure}[h]
  \begin{minipage}[t]{1.0\linewidth}
  \centering
  \includegraphics[width=5.5in]{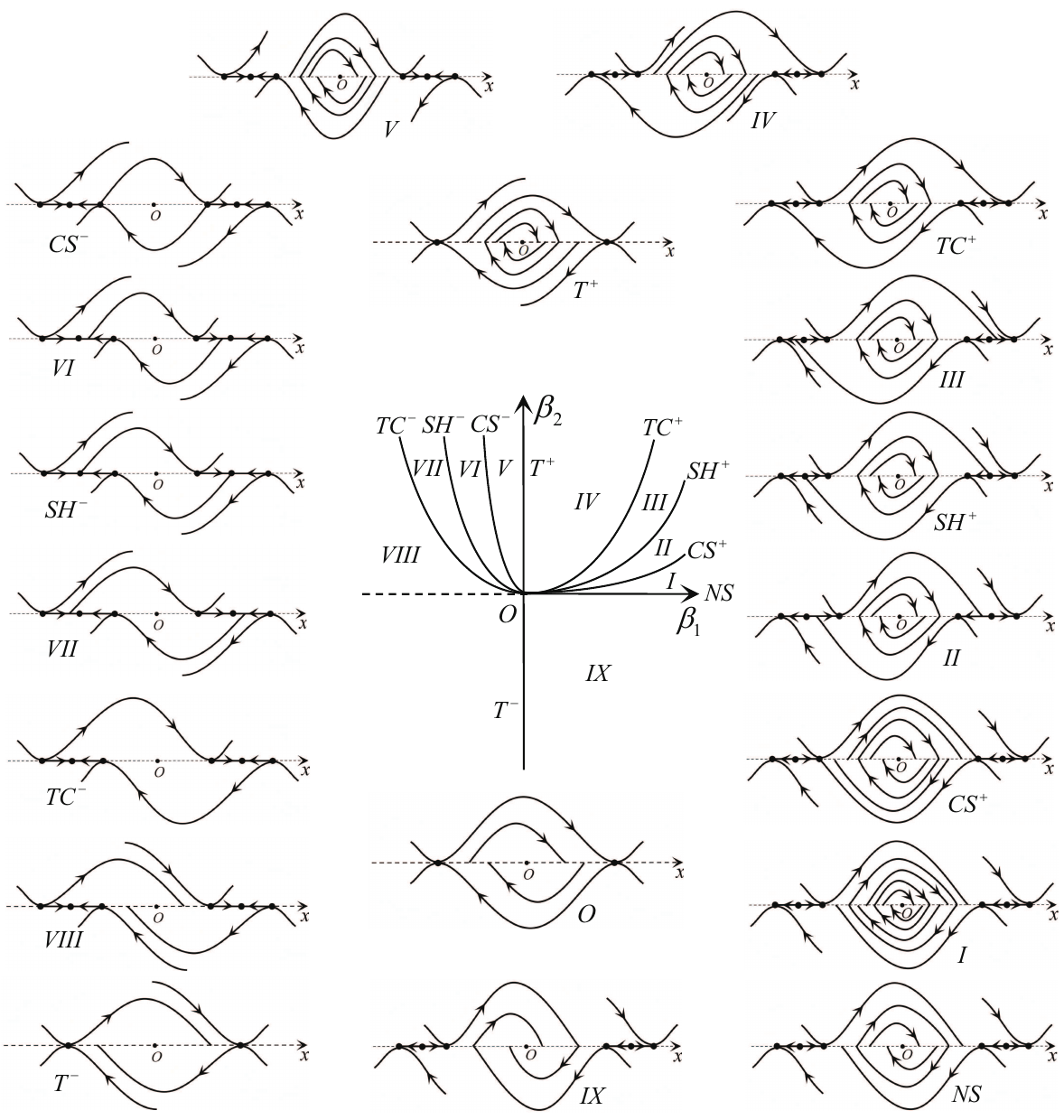}
  \end{minipage}
\caption{{\small Codim-2 crossing-sliding bifurcation of fold-fold type of system (\ref{Z2system}) when $f^+(-a,0;0)>0$.}}
\label{codim-2bifurdia2folddia}
\end{figure}
\begin{figure}[h]
  \begin{minipage}[t]{1.0\linewidth}
  \centering
  \includegraphics[width=5.5in]{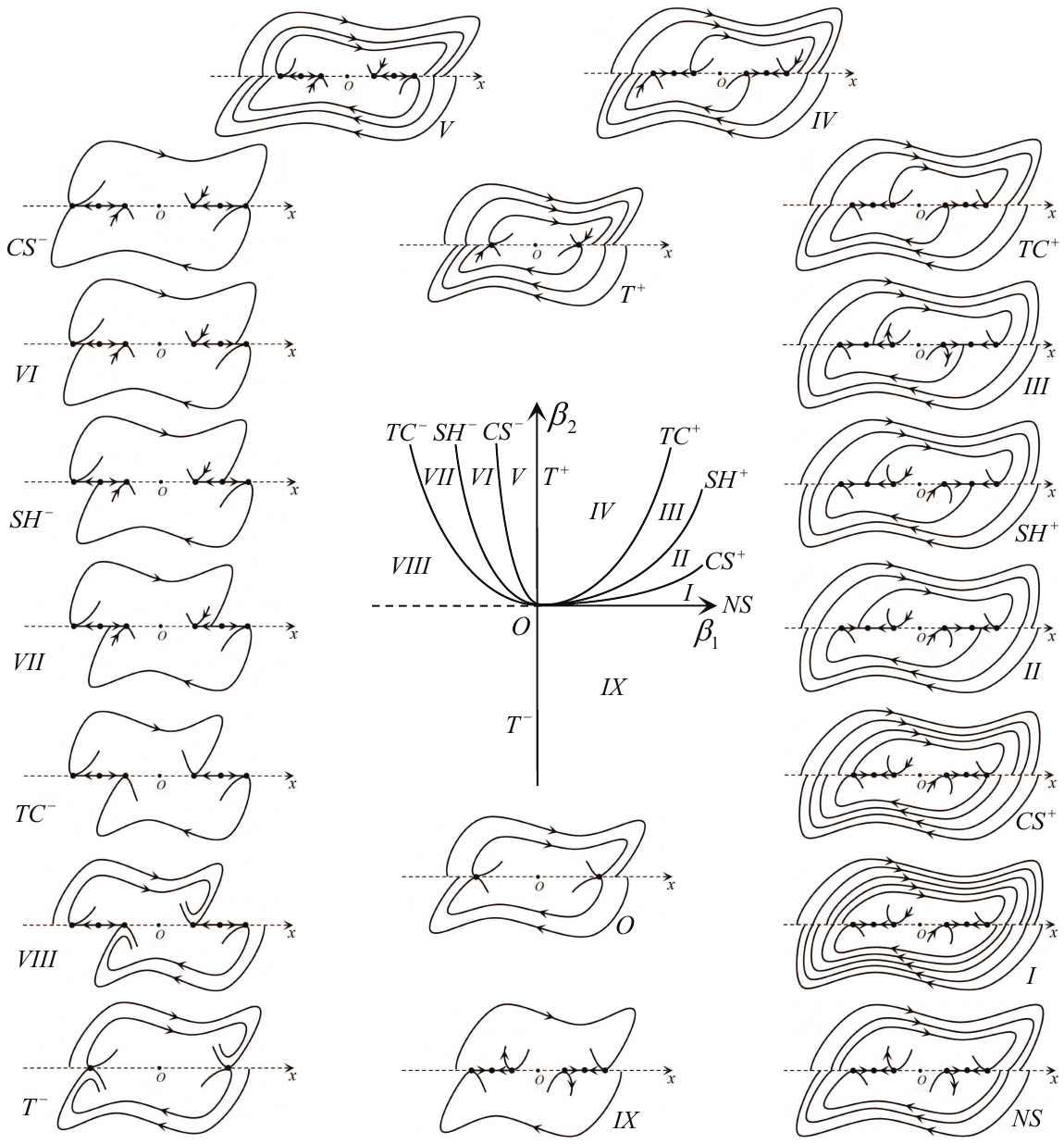}
  \end{minipage}
\caption{{\small Codim-2 crossing-sliding bifurcation of fold-fold type of system (\ref{Z2system}) when $f^+(-a,0;0)<0$.}}
\label{codim-2bifurdia2folddia2334}
\end{figure}
\end{thm}

In Theorem~\ref{codim-2-fold-bifur1}, an additional condition {\bf (H3)} is required to trigger a codimension-two bifurcation. Strictly speaking, it suffices to assume that $\Delta\ne0$. The stronger condition {\bf (H3)} is imposed, since the bifurcation result for $\Delta<0$ can be derived from the result presented herein via the transformation $(x,y,t)\rightarrow(x,-y,-t)$. In fact, when $f^+(-a,0;0)>0$ (resp. $<0$), $\Delta>0$ means that $\Gamma_0$ is internally unstable (resp. externally stable) and $\Delta<0$ means that $\Gamma_0$ is internally stable (resp. externally unstable) as we will see.

%%%%%%%%%%%%%%%%%%%%%%%%%%%%%%%%%%%%%%%%%%%%%%%%%%%%%%%%%%%%
\section{Preliminary lemmas}
\setcounter{equation}{0}
\setcounter{lm}{0}
\setcounter{thm}{0}
\setcounter{rmk}{0}
\setcounter{df}{0}
\setcounter{cor}{0}

To prove the main theorems, we give some preliminary results in this section.

\begin{lm}\label{cnkrcdwerc}
Assume that the upper subsystem of $(\ref{Z2system})$ with $\alpha=0$ satisfies {\bf(H0)}. Then there is a point $(b,c)\in\gamma_0^+\cap\Sigma^+$ and a constant $\delta>0$ such that $g^+(b,c;0)<0$, and $(b,c)$ is the unique intersection of $\Pi:=\{(x,c): |x-b|<\delta\}$ and $\gamma_0^+$. Moreover, there exists a constant $\varepsilon_1>0$, a neighborhood $U_1$ of $\alpha=0$ and $C^k$ functions $\tau^+(x,y;\alpha)$ and $\sigma^+(x,y;\alpha)$ $(resp.~\tau^-(x,y;\alpha)$ and $\sigma^-(x,y;\alpha))$ defined for $\|(x+a,y)\|<\varepsilon_1~(resp.~\|(x-a,y)\|<\varepsilon_1)$ and $\alpha\in U_1$ such that $\tau^+(-a,0;0)=\tau^+_0$ $(resp.~\tau^-(a,0;0)=\tau^-_0)$, $\sigma^+(-a,0;0)=b$ $(resp.~\sigma^-(a,0;0)=b)$ and the forward $(resp.~backward)$ orbits of the upper subsystem of $(\ref{Z2system})$ originating from $\{(x,y):\|(x+a,y)\|<\varepsilon_1\}$ $(resp.~ \{(x,y):\|(x-a,y)\|<\varepsilon_1\})$ can reach $\Pi$ at $(\sigma^+(x,y;\alpha),c)$ $(resp.~ (\sigma^-(x,y;\alpha),c))$ after a finite time $t=\tau^+(x,y;\alpha)$ $(resp.~t=\tau^-(x,y;\alpha))$ for $\alpha\in U_1$, where
$\tau^+_0>0$ $(resp.~\tau^-_0<0)$ is the travelling time of $\gamma_0^+$ from $(-a,0)$ $(resp.~(a,0))$ to $(b,c)$.
\end{lm}
\begin{proof}
The existence of $(b, c)$ and $\delta$ is obvious.  Let $(x^+(t,x,y;\alpha),y^+(t,x,y;\alpha))$ be the solution of the upper subsystem of $(\ref{Z2system})$ with the initial value $(x,y)$ at $t=0$. Consider equations
$$
\begin{aligned}
&F(t, x, y; \alpha, z):=x^+(t, x, y; \alpha)-z=0,\\
&G(t, x, y; \alpha, z):=y^+(t, x, y; \alpha)-c=0.
\end{aligned}
$$
Under the assumption {\bf(H0)}, we have $F(\tau^\pm_0, \mp a, 0; 0, b)=G(\tau^\pm_0, \mp a, 0; 0, b)=0$. Moreover, due to $g^+(b,c;0)<0$, the Jacobian matrix
$$\frac{\partial(F, G)}{\partial(t, z)}(\tau^\pm_0, \mp a, 0; 0, b)=
\left(
\begin{array}{cc}
f^+(b,c; 0)&-1\\
g^+(b,c;0)&0
\end{array}
\right)$$
is nonsingular. Thus, by the Implicit Function Theorem, there exists a constant $\varepsilon_1>0$, a neighborhood $U_1$ of $\alpha=0$ and $C^k$ functions $\tau^\pm(x,y;\alpha)$ and $\sigma^\pm(x,y;\alpha)$ defined for $\|(x\pm a,y)\|<\varepsilon_1$ and $\alpha\in U_1$ such that $\tau^\pm(\mp a,0;0)=\tau^\pm_0$, $\sigma^\pm(\mp a,0;0)=b$ and
\begin{equation}\label{awwrafmceifj}
\begin{aligned}
F(\tau^\pm(x,y;\alpha),x,y;\alpha,\sigma^\pm(x,y;\alpha))&=0,\\
G(\tau^\pm(x,y;\alpha),x,y;\alpha,\sigma^\pm(x,y;\alpha))&=0.
\end{aligned}
\end{equation}
In other words, the forward (resp. backward) orbits of the upper subsystem of $(\ref{Z2system})$ originating from $\{(x,y):\|(x+a,y)\|<\varepsilon_1\}$ (resp. $\{(x,y):\|(x-a,y)\|<\varepsilon_1\}$) can reach $\Pi$ at $(\sigma^+(x,y;\alpha),c)$ (resp. $(\sigma^-(x,y;\alpha),c)$) after a finite time $t=\tau^+(x,y;\alpha)$ (resp. $t=\tau^-(x,y;\alpha)$) for $\alpha\in U_1$. This completes the proof of Lemma~\ref{cnkrcdwerc}.
\end{proof}

Following the idea of \cite{EFEPFT}, we compute the partial derivatives of $\sigma^\pm(x,y;\alpha)$ listed in the next lemma.

\begin{lm}\label{Transipero}
Let
$$
\begin{aligned}
\lambda^+(t):=&~\exp\left(\int^{\tau^+_0}_t(f^+_x+g^+_y)(\gamma^+_0(s);0)ds\right),\qquad
&&\kappa^+_i:=\int^{\tau^+_0}_0\lambda^+(t)(f^+g^+_{\alpha_i}-g^+f^+_{\alpha_i})(\gamma^+_0(t);0)dt,\\
\lambda^-(t):=&~\exp\left(\int^{\tau^-_0}_t(f^+_x+g^+_y)(\hat\gamma^+_0(s);0)ds\right),\qquad
&&\kappa^-_i:=\int^{\tau^-_0}_0\lambda^-(t)(f^+g^+_{\alpha_i}-g^+f^+_{\alpha_i})(\hat\gamma^+_0(t);0)dt,\\
\end{aligned}
$$
where $i=1,2,\cdots,m$ and $\hat\gamma^+_0(t)$ denotes the solution of the upper subsystem of $(\ref{Z2system})$ with $\alpha=0$ satisfying $\hat\gamma^+_0(0)=(a,0)$. Then the following properties hold.
\begin{itemize}
\item[{\rm(1)}] $\sigma^+_{\alpha_i}(-a,0;0)=-\kappa_i^+/g^+(b,c;0)$ and $\sigma^-_{\alpha_i}(a,0;0)=-\kappa_i^-/g^+(b,c;0)$.
\item[{\rm(2)}]
If $g^+(-a,0;0)=0$, then
$$
\sigma^+_x(-a,0;0)=0,~~~ \sigma^+_{xx}(-a,0;0)=\frac{g^+_x(-a,0;0)\lambda^+(0)}{g^+(b,c;0)},~~~
\sigma^+_{x\alpha_i}(-a,0;0)=\frac{g^+_{\alpha_i}(-a,0;0)\lambda^+(0)}{g^+(b,c;0)};
$$
if $g^+(a,0;0)=0$, then
$$
\sigma^-_x(a,0;0)=0,~~~~~\sigma^-_{xx}(a,0;0)=\frac{g^+_x(a,0;0)\lambda^-(0)}{g^+(b,c;0)},~~~~~
\sigma^-_{x\alpha_i}(a,0;0)=\frac{g^+_{\alpha_i}(a,0;0)\lambda^-(0)}{g^+(b,c;0)}.
$$
\item[{\rm(3)}] $\lambda^-(0)=\lambda^+(0)/\lambda(0)$ and $\kappa_i^-=\kappa_i^+-\kappa_i\lambda^-(0)$, where $\lambda(t)$ and $\kappa_i$ are defined in $(\ref{notationsdefine})$.
 \end{itemize}
\end{lm}

\begin{proof}
Taking the derivative with respect to $\alpha_i$ in (\ref{awwrafmceifj}), we arrive at
\begin{equation}\label{qpofifj}
\left(
\begin{array}{c}
f^+(x_1, y_1; \alpha)\\
g^+(x_1, y_1; \alpha)
\end{array}
\right)\tau^\pm_{\alpha_i}(x,y;\alpha)+
w(\tau^\pm(x,y;\alpha), x, y; \alpha)=
\left(
\begin{array}{c}
\sigma^\pm_{\alpha_i}(x,y,\alpha)\\
0
\end{array}
\right),
\end{equation}
where $x_1=x^+(\tau^\pm(x, y; \alpha), x, y;\alpha), y_1=y^+(\tau^\pm(x, y; \alpha), x, y;\alpha)$ and
\begin{equation}\label{wfmckacui}
w(t, x, y; \alpha)=(
x^+_{\alpha_i}(t, x, y; \alpha),y^+_{\alpha_i}(t, x, y; \alpha)
)^\top.
\end{equation}
Left-multiplying (\ref{qpofifj}) by the vector $(-g^+(x_1, y_1; \alpha), f^+(x_1, y_1; \alpha))$ we get
\begin{equation}\label{eijvjippp}
(-g^+(x_1, y_1; \alpha), f^+(x_1, y_1; \alpha))w(\tau^\pm(x,y;\alpha), x, y; \alpha)=-g^+(x_1, y_1; \alpha)\sigma^\pm_{\alpha_i}(x, y; \alpha).
\end{equation}

Taking the derivative with respect to $t$ in (\ref{wfmckacui}) yields
$$
\begin{aligned}
w_t(t,x,y;\alpha)&=\frac{\partial}{\partial\alpha_i}\left(
\begin{array}{c}
f^+(x^+(t, x, y; \alpha),y^+(t, x, y; \alpha);\alpha)\\
g^+(x^+(t, x, y; \alpha),y^+(t, x, y; \alpha);\alpha)
\end{array}
\right)\\
&=A(t,x,y;\alpha)w(t, x, y;\alpha)+b(t, x, y;\alpha),
\end{aligned}
$$
where $w(0, x, y; \alpha)=(0,0)^\top$,
\begin{equation}\label{38279jcdscsd}
\begin{aligned}
A(t,x,y;\alpha)&=\frac{\partial(f^+, g^+)}{\partial(x, y)}(x^+(t, x, y; \alpha),y^+(t, x, y; \alpha);\alpha),\\
b(t, x, y;\alpha)&=\left(
\begin{array}{c}
f^+_{\alpha_i}(x^+(t, x, y; \alpha),y^+(t, x, y; \alpha);\alpha)\\
g^+_{\alpha_i}(x^+(t, x, y; \alpha),y^+(t, x, y; \alpha);\alpha)
\end{array}
\right).
\end{aligned}
\end{equation}
According to \cite[Proposition 7]{EFEPFT},
\begin{equation}\label{9kcsdf}
\left(
\begin{array}{c}
u(t,x,y;\alpha)\\
v(t,x,y;\alpha)
\end{array}
\right)=\bar\lambda(t,x,y;\alpha)
\left(
\begin{array}{c}
-g^+(x^+(t, x, y; \alpha),y^+(t, x, y; \alpha);\alpha)\\
f^+(x^+(t, x, y; \alpha),y^+(t, x, y; \alpha);\alpha)
\end{array}
\right)
\end{equation}
is the solution of the adjoint variational equation
$(\dot u, \dot v)^\top=-A(t,x,y;\alpha)^\top(u, v)^\top$ around the solution $(x^+(t, x, y; \alpha),y^+(t, x, y; \alpha))$,
where
$$\bar\lambda(t,x,y;\alpha)=\exp\left(-\int^t_0(f^+_x+g^+_y)(x^+(s, x, y; \alpha),y^+(s, x, y; \alpha))ds\right).$$
Then it follows from the statement (c) of \cite[Proposition 6]{EFEPFT} that
\begin{equation}\label{7y378nv}
(u(t,x,y;\alpha), v(t,x,y;\alpha))w(t,x,y;\alpha)=\int^t_0(u(s,x,y;\alpha), v(s,x,y;\alpha))b(s,x,y;\alpha)ds.
\end{equation}

Taking $t=\tau^\pm(x,y;\alpha)$ we have
$$
\begin{aligned}
&\bar\lambda(\tau^\pm(x,y;\alpha),x,y;\alpha)(-g^+(x_1, y_1; \alpha), f^+(x_1, y_1; \alpha))w(\tau^\pm(x,y;\alpha),x,y;\alpha)\\
=&\int^{\tau^\pm(x,y;\alpha)}_0\bar\lambda(s,x,y;\alpha)(f^+g^+_{\alpha_i}-g^+f^+_{\alpha_i})(x^+(s, x, y; \alpha),y^+(s, x, y; \alpha);\alpha)ds
\end{aligned}
$$
from (\ref{38279jcdscsd}), (\ref{9kcsdf}) and (\ref{7y378nv}). Substituting (\ref{eijvjippp}) into the above equality we obtain
$$
\begin{aligned}
\sigma^\pm_{\alpha_i}(x, y; \alpha)&=-\frac{\int^{\tau^\pm(x,y;\alpha)}_0\bar\lambda(s,x,y;\alpha)(f^+g^+_{\alpha_i}-g^+f^+_{\alpha_i})(x^+(s, x, y; \alpha),y^+(s, x, y; \alpha);\alpha)ds}{g^+(x_1, y_1; \alpha)\bar\lambda(\tau^\pm(x,y;\alpha),x,y;\alpha)}.
\end{aligned}
$$
Finally, statement (1) holds by taking $(x,y;\alpha)=(-a,0;0)$ and $(x,y;\alpha)=(a,0;0)$ respectively.

Taking the derivative with respect to $x$ and $y$ in (\ref{awwrafmceifj}) yields
\begin{equation}\label{qeqwrqpoojwv}
\left(\!\!
\begin{array}{c}
f^+(x_1, y_1; \alpha)\\
g^+(x_1, y_1; \alpha)
\end{array}
\!\!\right)
(\tau^\pm_x(x, y; \alpha),
\tau^\pm_y(x, y; \alpha)
)+\Phi(\tau^\pm(x,y;\alpha), x, y; \alpha)=
\left(\!\!
\begin{array}{cc}
\sigma^\pm_x(x, y; \alpha)&\sigma^\pm_y(x, y; \alpha)\\
0&0\\
\end{array}
\!\!\right),
\end{equation}
where
$$\Phi(t, x, y; \alpha)=
\left(
\begin{array}{cc}
x^+_x(t, x, y;\alpha)&x^+_y(t, x, y;\alpha)\\
y^+_x(t, x, y;\alpha)&y^+_y(t, x, y;\alpha)
\end{array}
\right)$$
is the principal fundamental matrix of the variational equation of the upper subsystem of $(\ref{Z2system})$ along the solution $(x^+(t, x, y;\alpha), y^+(t, x, y;\alpha))$.
Right-multiplying (\ref{qeqwrqpoojwv}) by $(1, 0)^\top$ we have
$$
\left(
\begin{array}{c}
f^+(x_1, y_1; \alpha)\\
g^+(x_1, y_1; \alpha)
\end{array}
\right)\tau^\pm_x(x, y; \alpha)+\Phi(\tau^\pm(x,y;\alpha), x, y; \alpha)
\left(
\begin{array}{c}
1\\
0
\end{array}
\right)
=
\left(
\begin{array}{c}
\sigma^\pm_x(x, y; \alpha)\\
0
\end{array}
\right),
$$
and then
\begin{equation}\label{ewbfmjcnjkf}
\Phi(\tau^\pm(x,y;\alpha), x, y; \alpha)
\left(
\begin{array}{c}
-1\\
0
\end{array}
\right)=\left(
\begin{array}{c}
f^+(x_1, y_1; \alpha)\\
g^+(x_1, y_1; \alpha)
\end{array}
\right)\tau^\pm_x(x, y; \alpha)-\left(
\begin{array}{c}
\sigma^\pm_x(x, y; \alpha)\\
0
\end{array}
\right).
\end{equation}

Besides, notice that
$$
\Phi(\tau^\pm(x,y;\alpha), x, y; \alpha)\left(
\begin{array}{c}
f^+(x, y; \alpha)\\
g^+(x, y; \alpha)
\end{array}
\right)=\left(
\begin{array}{c}
f^+(x_1, y_1; \alpha)\\
g^+(x_1, y_1; \alpha)
\end{array}
\right).
$$
This, together with (\ref{ewbfmjcnjkf}), implies that
$$
\Phi(\tau^\pm(x,y;\alpha), x, y; \alpha)
\left(
\begin{array}{cc}
f^+(x, y; \alpha)&-1\\
g^+(x, y; \alpha)&0
\end{array}
\right)=\left(
\begin{array}{cc}
f^+(x_1, y_1; \alpha)&-1\\
g^+(x_1, y_1; \alpha)&0
\end{array}
\right)
\left(
\begin{array}{cc}
1&\tau^\pm_x(x, y; \alpha)\\
0&\sigma^\pm_x(x, y; \alpha)
\end{array}
\right).$$
Then, taking determinants and using the Liouville's formula in the computation of the determinant of $\Phi(\tau^\pm(x,y;\alpha), x, y; \alpha)$, we get
\begin{equation}\label{urvnjfhp}
\begin{aligned}
&\sigma^\pm_x(x, y; \alpha)
=\frac{g^+(x, y; \alpha)}{g^+(x_1, y_1; \alpha)}\exp\left(\int^{\tau^\pm(x, y; \alpha)}_0(f^+_x+g^+_y)(x^+(t, x,y;\alpha),y^+(t, x,y;\alpha);\alpha)dt\right).
\end{aligned}
\end{equation}

Furthermore, taking the derivative with respect to $x$ and $\alpha_i$ in (\ref{urvnjfhp}) severally yields
\begin{equation}\label{ajcr3}
\begin{aligned}
\sigma^\pm_{xx}(x, y; \alpha)=&~\frac{g^+_x(x,y;\alpha)g^+(x_1,y_1;\alpha)-g^+(x,y;\alpha)g^+_x(x_1,y_1;\alpha)\sigma^\pm_x(x,y;\alpha)}{(g^+(x_1,y_1;\alpha))^2}\\
&\times \exp\left(\int^{\tau^\pm(x, y; \alpha)}_0(f^+_x+g^+_y)(x^+(t, x,y;\alpha),y^+(t, x,y;\alpha);\alpha)dt\right)\\
&+\sigma^\pm_x(x, y; \alpha)\Bigg\{-\frac{f^+_x(x_1,y_1;\alpha)+g^+_y(x_1,y_1;\alpha)}{g^+(x_1,y_1;\alpha)}y^+_x(\tau^\pm(x,y;\alpha),x,y;\alpha)\\
&+\int^{\tau^\pm(x, y; \alpha)}_0\nabla(f^+_x+g^+_y)(x^+(t, x,y;\alpha),y^+(t, x,y;\alpha);\alpha)\left(
\begin{array}{c}
x^+_x(t, x,y;\alpha)\\
y^+_x(t, x,y;\alpha)
\end{array}
\right)dt\Bigg\}
\end{aligned}
\end{equation}
and
\begin{equation}\label{cnkjnj23}
\begin{aligned}
\sigma^\pm_{x\alpha_i}(x, y; \alpha)=&~\frac{g^+_{\alpha_i}(x,y;\alpha)g^+(x_1,y_1;\alpha)-g^+(x,y;\alpha)\left(g^+_x(x_1,y_1;\alpha)
\sigma^\pm_{\alpha_i}(x,y;\alpha)+g^+_{\alpha_i}(x_1,y_1;\alpha)\right)}{(g^+(x_1,y_1;\alpha))^2}\\
&\times\exp\left(\int^{\tau^\pm(x, y; \alpha)}_0(f^+_x+g^+_y)(x^+(t, x,y;\alpha),y^+(t, x,y;\alpha);\alpha)dt\right)\\
&+\sigma^\pm_x(x, y; \alpha)\Bigg\{-\frac{f^+_x(x_1,y_1;\alpha)+g^+_y(x_1,y_1;\alpha)}{g^+(x_1,y_1;\alpha)}y^+_{\alpha_i}(\tau^\pm(x,y;\alpha),x,y;\alpha)\\
&+\int^{\tau^\pm(x, y; \alpha)}_0\nabla(f^+_x+g^+_y)(x^+(t, x,y;\alpha),y^+(t, x,y;\alpha);\alpha)\left(
\begin{array}{c}
x^+_{\alpha_i}(t, x,y;\alpha)\\
y^+_{\alpha_i}(t, x,y;\alpha)
\end{array}
\right)dt\\
&+\int^{\tau^\pm(x, y; \alpha)}_0(f^+_{x\alpha_i}+g^+_{y\alpha_i})(x^+(t, x,y;\alpha),y^+(t, x,y;\alpha);\alpha)
dt\Bigg\}.
\end{aligned}
\end{equation}
Finally, due to $g^+(-a,0;0)=0$ and $g^+(a,0;0)$, we can take $(x,y;\alpha)=(-a,0;0)$ and $(x,y;\alpha)=(a,0;0)$ respectively in (\ref{urvnjfhp}), (\ref{ajcr3}) and (\ref{cnkjnj23}) to obtain statement (2).

Note that $\tau_0^-+\tau_0=\tau_0^+$ and $\gamma_0^+(t+\tau_0)=\hat\gamma_0^+(t)$ for all $t\in\mathbb{R}$. Then
$$
\begin{aligned}
\lambda^-(0)&=\exp\left(\int^{\tau^-_0}_0(f^+_x+g^+_y)(\gamma^+_0(s+\tau_0);0)ds\right)\\
&=\exp\left(\int^{\tau^+_0}_0(f^+_x+g^+_y)(\gamma^+_0(t);0)dt-\int^{\tau_0}_0(f^+_x+g^+_y)(\gamma^+_0(t);0)dt\right)\\
&=\lambda^+(0)/\lambda(0),\\
\end{aligned}$$
and
$$
\begin{aligned}
\kappa^-_i&=\int^{\tau^-_0}_0\exp\left(\int^{\tau^-_0}_t(f^+_x+g^+_y)(\gamma^+_0(s+\tau_0);0)ds\right)(f^+g^+_{\alpha_i}-g^+f^+_{\alpha_i})(\gamma^+_0(t+\tau_0);0)dt\\
&=\int^{\tau^+_0}_{\tau_0}\exp\left(\int^{\tau^-_0}_{u-\tau_0}(f^+_x+g^+_y)(\gamma^+_0(s+\tau_0);0)ds\right)(f^+g^+_{\alpha_i}-g^+f^+_{\alpha_i})(\gamma^+_0(u);0)du\\
&=\int^{\tau^+_0}_{\tau_0}\exp\left(\int^{\tau^+_0}_{u}(f^+_x+g^+_y)(\gamma^+_0(v);0)dv\right)(f^+g^+_{\alpha_i}-g^+f^+_{\alpha_i})(\gamma^+_0(u);0)du\\
&=\kappa_i^+-\kappa_i\lambda^-(0),
\end{aligned}$$
i.e. statement (3) holds.
\end{proof}

Similar to Lemmas~\ref{cnkrcdwerc} and \ref{Transipero}, we can prove the following result and omit its proof here.
\begin{lm}\label{dafsferwerc}
Assume that the upper subsystem of $(\ref{Z2system})$ with $\alpha=0$ satisfies {\bf(H0)}. If $g^+(a,0;0)<0$, there exists a constant $\varepsilon_2>0$, a neighborhood $U_2$ of $\alpha=0$ and $C^k$ functions $\tau(x,y;\alpha)$ and $\sigma(x,y;\alpha)$ defined for $\|(x+a,y)\|<\varepsilon_2$ and $\alpha\in U_2$ such that $\tau(-a,0;0)=\tau_0$, $\sigma(-a,0;0)=a$ and the forward orbits of the upper subsystem of $(\ref{Z2system})$ originating from $\{(x,y):\|(x+a,y)\|<\varepsilon_2\}$ can reach $\Sigma$ at $(\sigma(x,y;\alpha),0)$ after a finite time $t=\tau(x,y;\alpha)$ for $\alpha\in U_2$. Moreover,
$$\sigma_{\alpha_i}(-a,0;0)=-\frac{\kappa_i}{g^+(a,0;0)},$$
and, if $g^+(-a,0;0)=0$, then
$$
\sigma_x(-a,0;0)=0,~~~~ \sigma_{xx}(-a,0;0)=\frac{g^+_x(-a,0;0)\lambda(0)}{g^+(a,0;0)},~~~~
\sigma_{x\alpha_i}(-a,0;0)=\frac{g^+_{\alpha_i}(-a,0;0)\lambda(0)}{g^+(a,0;0)}.
$$
\end{lm}

The following lemma provides a result on the decomposition of functions.

\begin{lm}\label{eiowu45cd}
Let $F(x):\mathbb{R}^m\rightarrow\mathbb{R}$ and $c_i(x):\mathbb{R}^m\rightarrow\mathbb{R}$ be $C^k$ $(k\ge1)$ functions that satisfy $F(0)=0$ and $c_i(0)=0$ respectively, where
$i=1,2\cdots,l\le m$. Assume that $F(x)\equiv0$ in $\{x\in N:c_1(x)=c_2(x)=\cdots=c_l(x)=0\}$ for a neighborhood $N$ of $x=0$.
Then there exists a neighborhood $N^*\subset N$ of $x=0$ and $C^{k-1}$ functions $F_i(x)$ such that
$$F(x)=\sum_{i=1}^lc_i(x)F_i(x)$$
for $x\in N^*$ if the Jacobian matrix of $(c_1(x),c_2(x),\cdots,c_l(x))$ at $x=0$ has the full rank. In addition, if $\nabla F(0)=0$, i.e. the gradient of $F(x)$ at $x=0$ vanishes, then $F_1(0)=F_2(0)=\cdots=F_l(0)=0$.
\end{lm}

\begin{proof}
If the Jacobian matrix of $(c_1(x),c_2(x),\cdots,c_l(x))$ at $x=0$ has the full rank, there are $j_i\in\{1,2,\cdots,m\}$, $i=1,2,\cdots,l$, such that
$$
\det\left.\frac{\partial (c_1(x),c_2(x),\cdots,c_l(x))}{\partial(x_{j_1},x_{j_2},\cdots,x_{j_l})}\right|_{x=0}\ne0.
$$
Let $c_i(x):=x_{j_i}$
%,\quad i=l+1,\cdots,m$$
such that $x_{j_i}\ne x_{j_1}, x_{j_2},\cdots,x_{j_l}$ for all $i=l+1,\cdots,m$ and $x_{j_p}\ne x_{j_q}$ for $p\ne q$.
Then there is a neighborhood $N_x\subset N$ such that $y=c(x):=(c_1(x),c_2(x),\cdots,c_m(x))$ is a $C^k$ diffeomorphism from $N_x$ to $N_y:=\{c(x): x\in N_x\}$. In this case, letting $G(y):=F(c^{-1}(y))$ for $y\in N_y$, we get $G(0,0,\cdots,0,\bar y)=0$ for $(0,0,\cdots,0,\bar y)\in N_y$ from the assumption of lemma, where $\bar y=(c_{l+1}(x),\cdots,c_m(x))$.
Hence, by the Taylor formula with the integral form of the remainder \cite[Theorem 3.18]{JJCA},
$$G(y)=G(0,0,\cdots,0,\bar y)+\sum_{i=1}^ly_i\int^1_0\frac{\partial G(ty_1,ty_2,\cdots,ty_l,\bar y)}{\partial y_i}dt=
\sum_{i=1}^ly_iG_i(y)$$
for $y\in N_y$, where
$$
G_i(y)=\int^1_0\frac{\partial G(ty_1,ty_2,\cdots,ty_l,\bar y)}{\partial y_i}dt.
$$
Since $F(x)$ and the diffeomorphism $y=c(x)$ are $C^k$, $G(y)$ is also $C^k$, which implies that $G_i(y)$, $i=1,2,\cdots,l$, are $C^{k-1}$.
Finally, by taking $F_i(x):=G_i(c(x))$ and $N^*:=N_x$, we get
$$F(x)=\sum_{i=1}^lc_i(x)F_i(x)$$
for $x\in N^*$. In addition,
if $\nabla F(0)=0$, then
$$\frac{\partial G(0)}{\partial y_{i}}=\sum_{j=1}^m\frac{\partial F(0)}{\partial x_{j}}\frac{\partial x_j(0)}{\partial y_i}=0$$ for $i=1,\cdots,l$, which implies $F_i(0)=G_i(0)=0$ by definitions of $F_i(x), G_i(y)$. This lemma is proved.
\end{proof}

%%%%%%%%%%%%%%%%%%%%%%%%%%%%%%%%%%%%%%%%%%%%%%%%%%%%%%%%%%%%
\section{Proof of Theorem~\ref{codim-1-bifur}}
\setcounter{equation}{0}
\setcounter{lm}{0}
\setcounter{thm}{0}
\setcounter{rmk}{0}
\setcounter{df}{0}
\setcounter{cor}{0}

To prove Theorem~\ref{codim-1-bifur}, we start by studying the dynamics on $\Sigma$.
\begin{lm}\label{codimenoneslidingdy}
Under the assumptions {\bf(H1)} and {\bf(H2)}, for any sufficiently small $\varepsilon_3>0$ there exists a neighborhood $U_3$ of $\alpha=0$ and a smooth function
\begin{equation}\label{2324sff}
\varrho_1(\alpha)=a+\sum_{i=1}^m\frac{g^+_{\alpha_i}(-a,0;0)}{g^+_x(-a,0;0)}\alpha_i+\mathcal{O}(\|\alpha\|^2)
\end{equation}
defined in $U_3$ such that the following properties hold for system $(\ref{Z2system})$ with $\alpha\in U_3$.
\begin{itemize}
\item[{\rm(1)}] $I^-_1:=\{(x,0):|x+a|<\varepsilon_3\}$ is split into an upward crossing segment
and a stable sliding segment
$$S^-_1:=
\left\{\begin{aligned}
&\{(x,0):-a-\varepsilon_3<x<-\varrho_1(\alpha)\}\qquad when~f^+(-a,0;0)>0,\\
&\{(x,0):-\varrho_1(\alpha)<x<-a+\varepsilon_3\}\qquad when~f^+(-a,0;0)<0
\end{aligned}
\right.
$$
by $(-\varrho_1(\alpha),0)$, which is a visible fold of the upper subsystem and a regular point of the lower subsystem. Moreover, the sliding orbit on $S^-_{1}$ is rightward $($resp. leftward$)$ when $f^+(-a,0;0)>0$ $($resp. $<0$$)$.
\item[{\rm(2)}] The dynamics exhibited on $I^+_1:=\{(x,0):|x-a|<\varepsilon_3\}$ and $I^-_1$ are $\mathbb{Z}_2$-symmetric with respect to $O$, namely $I^+_1$ is split into a downward crossing segment and a stable sliding segment
$$S^+_1:=
\left\{\begin{aligned}
&\{(x,0):\varrho_1(\alpha)<x<a+\varepsilon_3\}\qquad when~f^+(-a,0;0)>0,\\
&\{(x,0):a-\varepsilon_3<x<\varrho_1(\alpha)\}\qquad when~f^+(-a,0;0)<0
\end{aligned}
\right.
$$
by $(\varrho_1(\alpha),0)$, which is a visible fold of the lower subsystem and a regular point of the upper subsystem. Moreover, the sliding orbit on $S^+_{1}$ is leftward $($resp. rightward$)$ when $f^+(-a,0;0)>0$ $($resp. $<0$$)$.
\end{itemize}
\end{lm}

\begin{proof}
For any sufficiently small $\varepsilon_3>0$,
\begin{equation}\label{4827dscsf}
f^+(x,0;0)g^+_x(x,0;0)>0,\qquad g^+(-x,0;0)<0,
\end{equation}
and $x=-a$ is the unique zero of $g^+(x,0;0)$ for $(x,0)\in I^-_1$ by {\bf(H1)}, {\bf(H2)} and the sign-preserving property of continuous functions.
Considering the equation $g^+(x,0;\alpha)=0$, we have
$g^+(-a,0;0)=0$ and $g^+_x(-a,0;0)\ne0$ from {\bf(H1)}. Then, by the Implicit Function Theorem, there is a neighborhood $U_3$ of $\alpha=0$, a unique and smooth function $\varrho_1(\alpha):U_3\rightarrow\mathbb{R}$ such that $\varrho_1(0)=a$, $g^+(-\varrho_1(\alpha),0;\alpha)=0$ and $(-\varrho_1(\alpha),0)\in I^-_1$ for $\alpha\in U_3$. Moreover, $\varrho_1(\alpha)$ has the expansion given in (\ref{2324sff}).

On the other hand, due to (\ref{4827dscsf}), $U_3$ can be chosen such that
$f^+(x,0;\alpha)g^+_x(x,0;\alpha)>0$, $g^+(-x,0;\alpha)<0$
for $(x,0)\in I^-_1$ and $\alpha\in U_3$ by the sign-preserving property of continuous functions again. Hence $(-\varrho_1(\alpha),0)\in I^-_1$ is a visible fold of the upper subsystem with $f^+(-\varrho_1(\alpha),0;\alpha)g^+_x(-\varrho_1(\alpha),0;\alpha)>0$ and a regular point of the lower subsystem with $g^+(\varrho_1(\alpha),0;\alpha)<0$ for $\alpha\in U_3$. In this case, it is easy to verify that $I^-_1$ is split into an upward crossing segment and a stable sliding segment $S^-_1$, on which the sliding orbit is rightward (resp. leftward) when $f^+(-a,0;0)>0$ (resp. $<0$), by a standard analysis with the related notions. That is, statement (1) holds. Using the $\mathbb{Z}_2$-symmetry of system (\ref{Z2system}), we obtain statement (2) directly from statement (1).
\end{proof}

The next lemma states the transition of cycles as $\alpha$ varies.

\begin{lm}\label{codimenonecycexste1}
Under the assumption of Theorem~\ref{codim-1-bifur}, for any sufficiently small annulus $\mathcal{A}$ of $\Gamma_0$ there exists a neighborhood $U$ of $\alpha=0$ and a smooth function
\begin{equation}\label{cjrwecdsf}
\rho_1(\alpha)=\sum_{i=1}^m\theta_i\alpha_i+\mathcal{O}(\|\alpha\|^2)
\end{equation}
defined in $U$ such that the following statements hold for $\alpha\in U$.
\begin{itemize}
\item[{\rm(1)}] If $f^+(-a,0;0)\rho_1(\alpha)<0$, system $(\ref{Z2system})$ has a unique crossing cycle in $\mathcal{A}$, which is hyperbolically stable and does not $($resp. does$)$ enclose the regular-folds $(-\varrho_1(\alpha),0)$ and $(\varrho_1(\alpha),0)$ when $f^+(-a,0;0)>0$ $($resp. $<0)$. If $f^+(-a,0;0)\rho_1(\alpha)\ge0$, system $(\ref{Z2system})$ has no crossing cycles in $\mathcal{A}$.
\item[{\rm(2)}] If $\rho_1(\alpha)=0$, system $(\ref{Z2system})$ has a unique critical crossing cycle in $\mathcal{A}$, which crosses $\Sigma$ at the regular-folds $(-\varrho_1(\alpha),0)$ and $(\varrho_1(\alpha),0)$ and is internally $($resp. externally$)$ stable when $f^+(-a,0;0)>0$ $($resp. $<0)$. If $\rho_1(\alpha)\ne0$, system $(\ref{Z2system})$ has no critical crossing cycles in $\mathcal{A}$.
\item[{\rm(3)}] If $f^+(-a,0;0)\rho_1(\alpha)>0$, system $(\ref{Z2system})$ has a unique sliding cycle in $\mathcal{A}$, which is stable and slides on $S^-_1$ and $S^+_1$ simultaneously. If $f^+(-a,0;0)\rho_1(\alpha)\le0$, system $(\ref{Z2system})$ has no sliding cycles in $\mathcal{A}$.
\end{itemize}
\end{lm}

\begin{proof}
For any sufficiently small annulus $\mathcal{A}$ of $\Gamma_0$, letting $\varepsilon>0$ be the constant such that $\mathcal{A}\cap\Sigma=\{(x,0):|x+a|<\varepsilon\}\cup\{(x,0):|x-a|<\varepsilon\}$, we can ensure $\varepsilon<\min\{\varepsilon_2,\varepsilon_3\}$, where $\varepsilon_2$ and $\varepsilon_3$ are given in Lemma~\ref{dafsferwerc} and Lemma~\ref{codimenoneslidingdy} respectively. Then there exists a neighborhood $U\subset U_2\cap U_3$ of $\alpha=0$ such that
\begin{equation}\label{T1map}
\mathcal{T}_1(x;\alpha):=\sigma(x,0;\alpha)+x\qquad{\rm for}~~|x+a|<\varepsilon,~ \alpha\in U,
\end{equation}
is defined well, where $U_2$ and $\sigma(x,0;\alpha)$ are given in Lemma~\ref{dafsferwerc}, $U_3$ is given in Lemma~\ref{codimenoneslidingdy}.

When $f^+(-a,0;0)>0$ (resp. $<0$), based on Lemma~\ref{codimenoneslidingdy} and the $\mathbb{Z}_2$-symmetry of (\ref{Z2system}), the crossing cycles of system (\ref{Z2system}) in $\mathcal{A}$ are in one-to-one correspondence with the zeros of $\mathcal{T}_1(x;\alpha)$ in $(-\varrho_1(\alpha),-a+\varepsilon)$ (resp. $(-a-\varepsilon,-\varrho_1(\alpha))$) for $\alpha\in U$. Moreover, if it exists, the crossing cycle does not (resp. does) enclose the regular-folds $(-\varrho_1(\alpha),0)$ and $(\varrho_1(\alpha),0)$ when $f^+(-a,0;0)>0$ (resp. $<0$). On the other hand, by Lemma~\ref{dafsferwerc} and the Implicit Function Theorem, $U$ can be chosen to ensure that for $\alpha\in U$ the map $\mathcal{T}_1(x;\alpha)$ has a unique zero $x=\varrho_2(\alpha)$ in $(-a-\varepsilon,-a+\varepsilon)$, where $\varrho_2(\alpha)$ is smooth and
\begin{equation}\label{48jsvsd}
\varrho_2(\alpha)=-a+\sum_{i=1}^m\frac{\kappa_i}{g^+(a,0;0)}\alpha_i+\mathcal{O}(\|\alpha\|^2).
\end{equation}
Take
$\rho_1(\alpha):=-\varrho_1(\alpha)-\varrho_2(\alpha)$
for $\alpha\in U$. Then system (\ref{Z2system}) for $\alpha\in U$ has a unique crossing cycle in $\mathcal{A}$ if $f^+(-a,0;0)\rho_1(\alpha)<0$, while if $f^+(-a,0;0)\rho_1(\alpha)\ge0$, it has no crossing cycles in $\mathcal{A}$. The expansion (\ref{cjrwecdsf}) comes from (\ref{3892fcsfvdfg}), (\ref{2324sff}) and (\ref{48jsvsd}) directly. Besides, the crossing cycle is hyperbolic and stable due to $\partial\mathcal{T}_1(\varrho_2(\alpha),\alpha)/\partial x=1+\mathcal{O}(\|\alpha\|)>0$. Thus statement (1) holds.

Based on Lemma~\ref{codimenoneslidingdy} and the $\mathbb{Z}_2$-symmetry of (\ref{Z2system}), a critical crossing cycle of system (\ref{Z2system}) in $\mathcal{A}$ must cross $\Sigma$ at the regular-folds $(-\varrho_1(\alpha),0)$ and $(\varrho_1(\alpha),0)$. Hence, system (\ref{Z2system}) for $\alpha\in U$ has at most one critical crossing cycle in $\mathcal{A}$ and there is a critical crossing cycle if and only if $x=-\varrho_1(\alpha)$ is a zero of $\mathcal{T}_1(x;\alpha)$. From the analysis on the zeros of $\mathcal{T}_1(x;\alpha)$ in the last paragraph, $x=-\varrho_1(\alpha)$ is a zero of $\mathcal{T}_1(x;\alpha)$ if and only if $\rho_1(\alpha)=0$. Due to $\partial\mathcal{T}(\varrho_2(\alpha),\alpha)/\partial x=1+\mathcal{O}(\|\alpha\|)>0$, the critical crossing cycle is internally (resp. externally) stable when $f^+(-a,0;0)>0$ (resp. $<0$). This concludes statement (2).

Finally, from the sliding dynamics obtained in Lemma~\ref{codimenoneslidingdy} and the $\mathbb{Z}_2$-symmetry of (\ref{Z2system}), a sliding cycle of system (\ref{Z2system}) in $\mathcal{A}$ will slide on $S^-_1$ and enter into $\Sigma^+$ at $(-\varrho_1(\alpha),0)$ until it reaches $\Sigma$ again in $S^+_1$, and subsequently, the sliding cycle will slide on $S^+_1$ and enter into $\Sigma^-$ at $(\varrho_1(\alpha),0)$ until it returns to $S^-_1$. Hence, system (\ref{Z2system}) for $\alpha\in U$ has at most one sliding cycle in $\mathcal{A}$ and there is a sliding cycle if and only if $f^+(-a,0;0)\mathcal{T}_1(-\varrho_1(\alpha);\alpha)>0$, which is equivalent to $f^+(-a,0;0)\rho_1(\alpha)>0$. The stability of a sliding cycle is due to the stability of its sliding segments. Thus statement (3) holds.
\end{proof}

\begin{proof}[{\bf Proof of Theorem~\ref{codim-1-bifur}}]
For any sufficiently small annulus $\mathcal{A}$ of $\Gamma_0$, we take $U$ as the neighborhood of $\alpha=0$ given in Lemma~\ref{codimenonecycexste1}, and $\rho_1(\alpha)$ for $\alpha\in U$ as the function given in (\ref{cjrwecdsf}).
If $(\theta_1,\theta_2,\cdots,\theta_m)\ne0$, there is $j_1\in\{1,2,\cdots,m\}$ such that
\begin{equation}\label{38vsvd343few}
\theta_{j_1}\ne0.
\end{equation}
Let
\begin{equation}\label{sfaacdsn34afr}
\rho_i(\alpha):=\alpha_{j_i},\quad i=2,3,\cdots,m
\end{equation}
such that $\alpha_{j_i}\ne\alpha_{j_1}$ for all $i=2,3,\cdots,m$ and $\alpha_{j_k}\ne\alpha_{j_l}$ for $k\ne l$. Then the Jacobian matrix of
\begin{equation}\label{cdshgjgjn3465}
\beta=(\beta_1,\beta_2,\cdots,\beta_m)=(\rho_1(\alpha),\rho_2(\alpha),\cdots,\rho_m(\alpha))
\end{equation}
at $\alpha=0$ is nonsingular from (\ref{cjrwecdsf}), (\ref{38vsvd343few}) and (\ref{sfaacdsn34afr}). Thus (\ref{cdshgjgjn3465}) is a diffeomorphism from $U$ to its range $V$, where $U$ can be reduced if necessary.
Finally, combining the dynamics on $\Sigma$ and the information of various cycles given in Lemma~\ref{codimenoneslidingdy} and Lemma~\ref{codimenonecycexste1} respectively,
we get that for any $(\beta_2^*,\beta_3^*,\cdots,\beta_m^*)\in V^*$ the bifurcation diagram of system $\left.(\ref{Z2system})\right|_{\mathcal{A}}$ on the hyperplane $(\beta_2,\beta_3,\cdots,\beta_m)=(\beta_2^*,\beta_3^*,\cdots,\beta_m^*)$ is the one shown in Figure~\ref{codim-1bifurdia1} $($resp. Figure~\ref{codim-1bifurdia2}$)$ when $f^+(-a,0;0)>0$ $($resp. $<0$$)$, i.e. Theorem~\ref{codim-1-bifur} holds.
\end{proof}

%%%%%%%%%%%%%%%%%%%%%%%%%%%%%%%%%%%%%%%%%%%%%%%%%%%%%%%%%%%%
\section{Proof of Theorem~\ref{codim-2-cusp-bifur}}
\setcounter{equation}{0}
\setcounter{lm}{0}
\setcounter{thm}{0}
\setcounter{rmk}{0}
\setcounter{df}{0}
\setcounter{cor}{0}

We only prove Theorem~\ref{codim-2-cusp-bifur} for $f^+(-a,0;0)>0$, as an analogous proof for $f^+(-a,0;0)<0$ can be obtained by employing the methodology of this section. First of all, we state the dynamics on $\Sigma$ in the following lemma.

\begin{lm}\label{sldynamcierwe3}
Under the assumptions {\bf(H1)$^\prime$}, {\bf(H2)} and $f^+(-a,0;0)>0$, for any sufficiently small $\varepsilon_4>0$, there exists a neighborhood $U_4$ of $\alpha=0$ and two smooth functions
\begin{equation}\label{489jvdcsd}
\phi_1(\alpha)=\sum_{i=1}^m\zeta_i\alpha_i+\mathcal{O}(\|\alpha\|^2),\qquad \xi_1(\alpha)=a+\sum_{i=1}^m\frac{g^+_{x\alpha_i}(-a,0;0)}{g^+_{xx}(-a,0;0)}\alpha_i+\mathcal{O}(\|\alpha\|^2)
\end{equation}
defined in $U_4$ such that the following statements hold for system $(\ref{Z2system})$ with $\alpha\in U_4$.
\begin{itemize}
\item[{\rm(1)}] If $\phi_1(\alpha)=0$, $I^-_2:=\{(x,0):|x+a|<\varepsilon_4\}$ is split into two upward crossing segments by $T^-_c:=(-\xi_1(\alpha),0)$, which is a cusp of the upper subsystem satisfying $f^+(-\xi_1(\alpha),0;\alpha)>0$ and $g^+_{xx}(-\xi_1(\alpha),0;\alpha)>0$ and a regular point of the lower subsystem. If $\phi_1(\alpha)<0$, $I^-_2$ is an upward crossing segment. If $\phi_1(\alpha)>0$, $I^-_2$ is split into two upward crossing segments and a stable sliding segment $$S^-_2:=\left\{(x,0):-\sqrt{\phi_1(\alpha)}-\xi_1(\alpha)<x<\sqrt{\phi_1(\alpha)}-\xi_1(\alpha)\right\}$$
    by $T^-_{iv}$ and $T^-_v$, where $T^-_{iv}:=(-\sqrt{\phi_1(\alpha)}-\xi_1(\alpha),0)$ is an invisible fold of the upper subsystem and a regular point of the lower subsystem, $T^-_{v}:=(\sqrt{\phi_1(\alpha)}-\xi_1(\alpha),0)$ is a visible fold of the upper subsystem and a regular point of the lower subsystem. Moreover, the sliding orbit on $S^-_{2}$ is rightward.
\item[{\rm(2)}] The dynamics exhibited on $I^+_2:=\{(x,0):|x-a|<\varepsilon_4\}$ and $I^-_2$ are $\mathbb{Z}_2$-symmetric with respect to $O$. That is, if $\phi_1(\alpha)=0$, $I^+_2$ is split into two downward crossing segments by $T^+_c:=(\xi_1(\alpha),0)$, which is a cusp of the lower subsystem satisfying $f^+(-\xi_1(\alpha),0;\alpha)>0$ and $g^+_{xx}(-\xi_1(\alpha),0;\alpha)>0$ and a regular point of the upper subsystem. If $\phi_1(\alpha)<0$, $I^+_2$ is a downward crossing segment. If  $\phi_1(\alpha)>0$, $I^+_2$ is split into two downward crossing segments and a stable sliding segment
    $$S^+_2:=\left\{(x,0):-\sqrt{\phi_1(\alpha)}+\xi_1(\alpha)<x<\sqrt{\phi_1(\alpha)}+\xi_1(\alpha)\right\}$$
    by $T^+_v$ and $T^+_{iv}$, where $T^+_{v}:=(-\sqrt{\phi_1(\alpha)}+\xi_1(\alpha),0)$ is a visible fold of the lower subsystem and a regular point of the upper subsystem, $T^+_{iv}:=(\sqrt{\phi_1(\alpha)}+\xi_1(\alpha),0)$ is an invisible fold of the lower subsystem and a regular point of the upper subsystem. Moreover, the sliding orbit on $S^+_{2}$ is leftward.
\end{itemize}
\end{lm}

\begin{proof}
Since $g^+(-a,0;0)=0$ in {\bf(H1)$^\prime$},
we can write $g^+(x,0;\alpha)$ around $x=-a$ and $\alpha=0$ in the form
$$
\begin{aligned}
g^+(x,0;\alpha)=&\sum_{i=1}^mg^+_{\alpha_i}(-a,0;0)\alpha_i+\mathcal{O}(\|\alpha\|^2)+\left(\sum_{i=1}^mg^+_{x\alpha_i}(-a,0;0)\alpha_i+\mathcal{O}(\|\alpha\|^2)\right)(x+a)\\
&+\frac{1}{2}\left(g^+_{xx}(-a,0;0)+\mathcal{O}(\|\alpha\|)\right)(x+a)^2+\mathcal{O}((x+a)^3).
\end{aligned}
$$
Perform a linear coordinate shift by introducing a new variable $\tilde x=I(x;\alpha):=x+a+\tilde\xi_1,$
where $\tilde\xi_1=\tilde\xi_1(\alpha)$ is a priori unknown function that will be determined later. Thus
\begin{equation}\label{ekto5y53}
\begin{aligned}
g^+\circ I^{-1}(\tilde x;\alpha)=&\sum_{i=1}^mg^+_{\alpha_i}(-a,0;0)\alpha_i+\mathcal{O}(\|\alpha\|^2)+\left(\sum_{i=1}^mg^+_{x\alpha_i}(-a,0;0)\alpha_i+\mathcal{O}(\|\alpha\|^2)\right)(\tilde x-\tilde\xi_1)\\
&+\frac{1}{2}\left(g^+_{xx}(-a,0;0)+\mathcal{O}(\|\alpha\|)\right)(\tilde x-\tilde\xi_1)^2+\mathcal{O}((\tilde x-\tilde\xi_1)^3)\\
=&\sum_{i=1}^mg^+_{\alpha_i}(-a,0;0)\alpha_i+\mathcal{O}(\|\alpha\|^2)-\left(\sum_{i=1}^mg^+_{x\alpha_i}(-a,0;0)\alpha_i+\mathcal{O}(\|\alpha\|^2)\right)\tilde\xi_1+\mathcal{O}(\tilde\xi_1^2)\\
&+\left(\sum_{i=1}^mg^+_{x\alpha_i}(-a,0;0)\alpha_i+\mathcal{O}(\|\alpha\|^2)-\left(g^+_{xx}(-a,0;0)+\mathcal{O}(\|\alpha\|)\right)\tilde\xi_1+\mathcal{O}(\tilde\xi_1^2)\right)\tilde x\\
&+\frac{1}{2}\left(g^+_{xx}(-a,0;0)+\mathcal{O}(\|\alpha\|)+\mathcal{O}(\tilde\xi_1)\right)\tilde x^2+\mathcal{O}(\tilde x^3).
\end{aligned}
\end{equation}
Let $L(\alpha,\tilde\xi_1)$ be the coefficient of $\tilde x$ in (\ref{ekto5y53}). Then
$L(0,0)=0, {\partial L(0,0)}/{\partial\tilde\xi_1}=-g^+_{xx}(-a,0;0)$, ${\partial L(0,0)}/{\partial\alpha_i}=g^+_{x\alpha_i}(-a,0;0)$.
Due to $g^+_{xx}(-a,0;0)>0$ in {\bf(H1)$^\prime$}, by the Implicit Function Theorem there is a neighborhood $\widetilde U_4$ of $\alpha=0$, a unique and smooth function $\tilde\xi_1=\tilde\xi_1(\alpha): \widetilde U_4\rightarrow\mathbb{R}$ satisfying $\tilde\xi_1(0)=0$ and $L(\alpha,\tilde\xi_1(\alpha))=0$. In addition,
\begin{equation}\label{90dsf}
\tilde\xi_1(\alpha)=\sum_{i=1}^m\frac{g^+_{x\alpha_i}(-a,0;0)}{g^+_{xx}(-a,0;0)}\alpha_i+\mathcal{O}(\|\alpha\|^2).
\end{equation}

Substituting $\tilde\xi_1(\alpha)$ into (\ref{ekto5y53}), we obtain
$$
g^+\circ I^{-1}(\tilde x;\alpha)=\sum_{i=1}^mg^+_{\alpha_i}(-a,0;0)\alpha_i+\mathcal{O}(\|\alpha\|^2)+\frac{1}{2}\left(g^+_{xx}(-a,0;0)+\mathcal{O}(\|\alpha\|)\right)\tilde x^2+\mathcal{O}(\tilde x^3).
$$
Due to $g^+_{xx}(-a,0;0)>0$, regarding $\tilde x^2$ as a new variable and using the Implicit Function Theorem, we can reduce $\widetilde U_4$ to ensure that for $\alpha\in\widetilde U_4$ there is a unique and smooth function
$$\phi_1(\alpha)=-\sum_{i=1}^m\frac{2g^+_{\alpha_i}(-a,0;0)}{g^+_{xx}(-a,0;0)}\alpha_i+\mathcal{O}(\|\alpha\|^2)=\sum_{i=1}^m\zeta_i\alpha_i+\mathcal{O}(\|\alpha\|^2)$$
such that $g^+\circ I^{-1}(\tilde x;\alpha)$ near $\tilde x=0$ has a unique zero $\tilde x=\tilde x^*_0:=0$ if $\phi_1(\alpha)=0$, no zeros if $\phi_1(\alpha)<0$, and exactly two zeros
$\tilde x=\tilde x^*_\pm:=\pm\sqrt{\phi_1(\alpha)}$ if $\phi_1(\alpha)>0$.
Moreover,
$$
\begin{aligned}
\frac{\partial g^+\circ I^{-1}(\tilde x^*_0;\alpha)}{\partial \tilde x}=&~0,\qquad \frac{\partial^2 g^+\circ I^{-1}(\tilde x^*_0;\alpha)}{\partial \tilde x^2}=g^+_{xx}(-a,0;0)+\mathcal{O}\left(\|\alpha\|\right),\\
\frac{\partial g^+\circ I^{-1}(\tilde x^*_\pm;\alpha)}{\partial \tilde x}=&~\pm\left( g^+_{xx}(-a,0;0)+\mathcal{O}\left(\|\alpha\|\right)+\mathcal{O}(\sqrt{\phi_1(\alpha)})\right)\sqrt{\phi_1(\alpha)}.
\end{aligned}
$$
As a result, for any sufficiently small $\varepsilon_4>0$ there is a neighborhood $U_4\subset \widetilde U_4$ of $\alpha=0$ such that the following conclusions hold for $\alpha\in U_4$:
\begin{description}
\setlength{\itemsep}{0mm}
\item[{\rm(i)}] If $\phi_1(\alpha)=0$, $g^+(x,0;\alpha)$ has a unique zero $x=x^*_0:=\tilde x^*_0-\xi_1(\alpha)=-\xi_1(\alpha)$ in $(-a-\varepsilon_4,-a+\varepsilon_4)$, which is of multiplicity two with $g^+_{xx}(x^*_0,0;\alpha)=g^+_{xx}(-a,0;0)+\mathcal{O}\left(\|\alpha\|\right)>0$.
\item[{\rm(ii)}] If $\phi_1(\alpha)<0$, $g^+(x,0;\alpha)$ has no zeros in $(-a-\varepsilon_4,-a+\varepsilon_4)$ and, more precisely, $g^+(x,0;\alpha)>0$.
\item[{\rm(iii)}] If $\phi_1(\alpha)>0$, $g^+(x,0;\alpha)$ has exactly two zeros $x=x^*_\pm:=\tilde x^*_\pm-\xi_1(\alpha)=\pm\sqrt{\phi_1(\alpha)}-\xi_1(\alpha)$ in $(-a-\varepsilon_4,-a+\varepsilon_4)$, which are simple with $g^+_x(x^*_\pm,0;\alpha)=\pm(g^+_{xx}(-a,0;0)+\mathcal{O}\left(\|\alpha\|\right)+\mathcal{O}(\sqrt{\phi_1(\alpha)}))\sqrt{\phi_1(\alpha)}\gtrless0$.
\end{description}
Here $\xi_1(\alpha):=a+\tilde\xi_1(\alpha)$ and thus (\ref{489jvdcsd}) is obtained from (\ref{90dsf}).

On the other hand, by {\bf(H1)$^\prime$}, {\bf(H2)}, $f^+(-a,0;0)>0$ and the sign-preserving property of continuous functions, $U_4$ can be chosen such that
$$f^+(x,0;\alpha)>0,\qquad g^+_{xx}(x,0;\alpha)>0,\qquad g^+(-x,0;\alpha)<0$$
for $(x,0)\in I^-_2$ and $\alpha\in U_4$. Hence for $\alpha\in U_4$ we obtain that $T^-_c\in I^-_2$ is a cusp of the upper subsystem satisfying $f^+(-\xi_1(\alpha),0;\alpha)>0$ and $g^+_{xx}(-\xi_1(\alpha),0;\alpha)>0$, $T^-_{iv}\in I^-_2$ is an invisible fold of the upper subsystem satisfying $f^+(-\sqrt{\phi_1(\alpha)}-\xi_1(\alpha),0;\alpha)>0$, $T^-_{v}\in I^-_2$ is a visible fold of the upper subsystem satisfying $f^+(\sqrt{\phi_1(\alpha)}-\xi_1(\alpha),0;\alpha)>0$. Moreover, they are regular points of the lower subsystem satisfying $g^+(\xi_1(\alpha),0;\alpha)<0$, $g^+(\sqrt{\phi_1(\alpha)}+\xi_1(\alpha),0;\alpha)<0$ and $g^+(-\sqrt{\phi_1(\alpha)}+\xi_1(\alpha),0;\alpha)<0$ respectively. In this case, it is easy to verify that $I^-_2$ is split into two upward crossing segments by $T^-_c$ if $\phi_1(\alpha)=0$, and into two upward crossing segments and a stable sliding segment $S^-_2$ by $T^-_{iv}$ and $T^-_v$ if $\phi_1(\alpha)>0$, while if $\phi_1(\alpha)<0$, $I^-_2$ is an upward crossing segment. Moreover, a direct computation for the sliding vector field yields that the sliding orbit on $S^-_{2}$ is rightward. That is, statement (1) holds. Using the $\mathbb{Z}_2$-symmetry of system (\ref{Z2system}), we obtain statement (2) directly from statement (1).
\end{proof}

Notice that the transition of dynamics near $(-a,0)$ or $(a,0)$ corresponds to the bifurcation induced by the collision of a visible regular-fold and an invisible regular-fold. The bifurcation diagram of such bifurcation has been drawn in \cite{YAK}. Lemma~\ref{sldynamcierwe3} provides a description version associated to the bifurcation diagram
within the context of this paper.

To prove Theorem~\ref{codim-2-cusp-bifur}, we next determine the information on the cycles bifurcating from $\Gamma_0$.

\begin{lm}\label{eiwu343cd}
Under the assumption of Theorem~\ref{codim-2-cusp-bifur} and $f^+(-a,0;0)>0$, for any sufficiently small annulus $\mathcal{A}$ of $\Gamma_0$ there exists a neighborhood $U\subset U_4$ of $\alpha=0$ and smooth functions
\begin{equation}\label{w234dsvsf}
\phi_2(\alpha)=\sum_{i=1}^m\eta_i\alpha_i+\mathcal{O}(\|\alpha\|^2),
\end{equation}
$\xi_2(z,\alpha)$ and $\xi_3(z,\alpha)$ with $\xi_2(0,0)=\xi_2(0,0)=0$ defined for $z$ closed to $0$ and $\alpha\in U$
such that the following statements hold for $\alpha\in U$.
\begin{itemize}
\item[{\rm(1)}] System $(\ref{Z2system})$ has at most one crossing cycle in $\mathcal{A}$, which is hyperbolic and stable if it exists. Moreover, there exits a crossing cycle if and only if one of the following conditions holds:
\begin{itemize}
\item[{\rm(1a)}] $\phi_1(\alpha)<0$;
\item[{\rm(1b)}] $\phi_1(\alpha)=0$ and $\phi_2(\alpha)\ne0$, for which the crossing cycle encloses the regular-cusps $T^-_c$ and $T^+_c$ if $\phi_2(\alpha)>0$, while if $\phi_2(\alpha)<0$, it does not enclose them;
\item[{\rm(1c)}] $\phi_1(\alpha)>0$ and $\phi_2(\alpha)<(-1-\xi_2(\sqrt{\phi_1(\alpha)},\alpha))\sqrt{\phi_1(\alpha)}$, for which the crossing cycle does not enclose any of the regular-folds $T^\pm_{iv}$ and $T^\pm_v$;
\item[{\rm(1d)}] $\phi_1(\alpha)>0$ and $\phi_2(\alpha)>(3+\xi_3(\sqrt{\phi_1(\alpha)},\alpha))\sqrt{\phi_1(\alpha)}$, for which the crossing cycle encloses all of the regular-folds  $T^\pm_{iv}$ and $T^\pm_v$.
\end{itemize}
\item[{\rm(2)}] System $(\ref{Z2system})$ has at most one critical crossing cycle in $\mathcal{A}$. Moreover, there is a critical crossing cycle if and only if one of the following conditions holds:
\begin{itemize}
\item[{\rm(2a)}] $\phi_1(\alpha)=0$ and $\phi_2(\alpha)=0$, for which the critical crossing cycle crosses $\Sigma$ at the regular-cusps $T^-_c$ and $T^+_c$, and it is stable.
\item[{\rm(2b)}] $\phi_1(\alpha)>0$ and $\phi_2(\alpha)=(-1-\xi_2(\sqrt{\phi_1(\alpha)},\alpha))\sqrt{\phi_1(\alpha)}$, for which the critical crossing cycle crosses $\Sigma$ at the regular-folds $T^-_v$ and $T^+_v$, and it is internally stable.
\item[{\rm(2c)}] $\phi_1(\alpha)>0$ and $\phi_2(\alpha)=(3+\xi_3(\sqrt{\phi_1(\alpha)},\alpha))\sqrt{\phi_1(\alpha)}$, for which the critical crossing cycle crosses $\Sigma$ at the points where the backward orbits of the upper subsystem of $(\ref{Z2system})$ starting at the regular-folds $T^-_v$ and $T^+_v$ intersect with $\Sigma$ respectively, and it is externally stable.
\end{itemize}
\item[{\rm(3)}] System $(\ref{Z2system})$ has at most one sliding cycle in $\mathcal{A}$, which is stable if it exists. Moreover, there is a sliding cycle if and only if one of the following conditions holds:
\begin{itemize}
\item[{\rm(3a)}] $\phi_1(\alpha)>0$ and $(-1-\xi_2(\sqrt{\phi_1(\alpha)},\alpha))\sqrt{\phi_1(\alpha)}<\phi_2(\alpha)<(1-\xi_2(\sqrt{\phi_1(\alpha)},\alpha))\sqrt{\phi_1(\alpha)}$, for which the sliding cycle arrives at a point of $S^+_2$ from $\Sigma^+$.
\item[{\rm(3b)}] $\phi_1(\alpha)>0$ and $\phi_2(\alpha)=(1-\xi_2(\sqrt{\phi_1(\alpha)},\alpha))\sqrt{\phi_1(\alpha)}$, for which the sliding cycle arrives at $T^+_{iv}$ from $\Sigma^+$.
\item[{\rm(3c)}] $\phi_1(\alpha)>0$ and $(1-\xi_2(\sqrt{\phi_1(\alpha)},\alpha))\sqrt{\phi_1(\alpha)}<\phi_2(\alpha)<(3+\xi_3(\sqrt{\phi_1(\alpha)},\alpha))\sqrt{\phi_1(\alpha)}$, for which the sliding cycle arrives at a point of $S^+_2$ from $\Sigma^-$.
\end{itemize}
\end{itemize}
\end{lm}

\begin{proof}
For any sufficiently small annulus $\mathcal{A}$ of $\Gamma_0$, letting $\varepsilon>0$ be the constant such that $\mathcal{A}\cap\Sigma=\{(x,0):|x+a|<\varepsilon\}\cup\{(x,0):|x-a|<\varepsilon\}$, we can ensure $\varepsilon<\min\{\varepsilon_2,\varepsilon_4\}$, where $\varepsilon_2$ and $\varepsilon_4$ are given in Lemma~\ref{dafsferwerc} and Lemma~\ref{sldynamcierwe3} respectively. Then there exists a neighborhood $U\subset U_2\cap U_4$ of $\alpha=0$ such that
we still can define
\begin{equation}\label{cjt43}
\mathcal{T}_1(x;\alpha)=\sigma(x,0;\alpha)+x\qquad{\rm for}~~|x+a|<\varepsilon,~ \alpha\in U
\end{equation}
as in (\ref{T1map}), where $U_2$ and $U_4$ are given in Lemma~\ref{dafsferwerc} and Lemma~\ref{sldynamcierwe3} respectively. In the following, we always work in $U$ and allow it to be reduced if necessary.

Due to $f^+(-a,0;0)>0$, it follows from Lemma~\ref{sldynamcierwe3} and the $\mathbb{Z}_2$-symmetry of (\ref{Z2system}) that for $\alpha\in U$ the crossing cycles of system (\ref{Z2system}) in $\mathcal{A}$ are in one-to-one correspondence with the zeros of $\mathcal{T}_1(x;\alpha)$ in
$$
\mathcal{I}=\left\{
\begin{aligned}
&(-a-\varepsilon,-a+\varepsilon)\qquad &&{\rm if}~~\phi_1(\alpha)<0,\\
&(-a-\varepsilon,-a+\varepsilon)\setminus\{-\xi_1(\alpha)\}\qquad &&{\rm if}~~\phi_1(\alpha)=0,\\
&(-a-\varepsilon,-a+\varepsilon)\setminus[-\varsigma(\alpha),\sqrt{\phi_1(\alpha)}-\xi_1(\alpha)]\qquad &&{\rm if}~~\phi_1(\alpha)>0,
\end{aligned}
\right.
$$
where $\varsigma(\alpha)$ is the value such that the backward orbit of the upper subsystem starting at $T^-_v$ intersects with $\Sigma$ at $(-\varsigma(\alpha),0)$ and it will be determined later. On the other hand, by Lemma~\ref{dafsferwerc} and the Implicit Function Theorem $U$ can be chosen to ensure that for $\alpha\in U$ the map $\mathcal{T}_1$ has a unique zero $x=\varrho_2(\alpha)$ in $(-a-\varepsilon,-a+\varepsilon)$ as in the proof of Lemma~\ref{codimenonecycexste1}.
Hence, system (\ref{Z2system}) for $\alpha\in U$ has at most one crossing cycle in $\mathcal{A}$. Moreover, if there exists a crossing cycle, it is hyperbolic and stable due to $\partial\mathcal{T}_1(\varrho_2(\alpha);\alpha)/\partial x=1+\mathcal{O}(\|\alpha\|)>0$. That is, the first part of statement (1) holds.

To prove the rest part of statement (1), it is enough to determine the values of $\alpha\in U$ for which $\varrho_2(\alpha)\in\mathcal{I}$. Clearly, $\varrho_2(\alpha)\in\mathcal{I}$ in the case of $\phi_1(\alpha)<0$, i.e. condition (1a) is obtained.

Taking
\begin{equation}\label{cr334e}
\phi_2(\alpha):=\mathcal{T}_1(-\xi_1(\alpha);\alpha)
\end{equation}
for $\alpha\in U$, we can verify that $\phi_2(\alpha)$ is smooth and satisfies (\ref{w234dsvsf}) from Lemma~\ref{dafsferwerc}, (\ref{ew45fsd}), (\ref{489jvdcsd}) and (\ref{cjt43}). Then $\varrho_2(\alpha)\ne-\xi_1(\alpha)$ is equivalent to $\phi_2(\alpha)\ne0$ due to the uniqueness of the zeros of $\mathcal{T}_1$ in $(-a-\varepsilon,-a+\varepsilon)$, Thus, in the case of $\phi_1(\alpha)=0$, $\varrho_2(\alpha)\in\mathcal{I}$ if and only if $\phi_2(\alpha)\ne0$, i.e. condition (1b) is obtained.
Besides, since $\varepsilon$ and $U$ can be any small to ensure that $\mathcal{T}_1$ is strictly increasing in $(-a-\varepsilon,a+\varepsilon)$ with respect to $x$, $\varrho_2(\alpha)<-\xi_1(\alpha)$ (resp. $>-\xi_1(\alpha)$) is equivalent to $\phi_2(\alpha)>0$ (resp. $<0$). This, together with Lemma~\ref{sldynamcierwe3} and the $\mathbb{Z}_2$-symmetry of (\ref{Z2system}), concludes that the crossing cycle existing for condition (1b) encloses the regular-cusps $T^-_c$ and $T^+_c$ if $\phi_2(\alpha)>0$, while if $\phi_2(\alpha)<0$, it does not enclose them.

In the case of $\phi_1(\alpha)>0$, we claim that $\varrho_2(\alpha)\in(\sqrt{\phi_1(\alpha)}-\xi_1(\alpha),-a+\varepsilon)\subset\mathcal{I}$ is equivalent that there is a smooth function $\xi_2(z,\alpha)$ with $\xi_2(0,0)=0$ such that
\begin{equation}\label{94fdcds}
\phi_2(\alpha)<\left(-1-\xi_2(\sqrt{\phi_1(\alpha)},\alpha)\right)\sqrt{\phi_1(\alpha)},
\end{equation}
i.e. condition (1c) holds, and $\varrho_2(\alpha)\in(-a-\varepsilon, -\varsigma(\alpha))\subset\mathcal{I}$ is equivalent that there is a smooth function $\xi_3(z,\alpha)$ with $\xi_3(0,0)=0$ such that
\begin{equation}\label{94fdcdsdae}
\phi_2(\alpha)>\left(3+\xi_3(\sqrt{\phi_1(\alpha)},\alpha)\right)\sqrt{\phi_1(\alpha)},
\end{equation}
i.e. condition (1d) holds. In fact,
since $\mathcal{T}_1$ is strictly increasing in $(-a-\varepsilon,a+\varepsilon)$ with respect to $x$, $\varrho_2(\alpha)\in(\sqrt{\phi_1(\alpha)}-\xi_1(\alpha),-a+\varepsilon)$ and $\varrho_2(\alpha)\in(-a-\varepsilon, -\varsigma(\alpha))$ are equivalent to
\begin{equation}\label{cafadsfm}
\mathcal{T}_1\left(\sqrt{\phi_1(\alpha)}-\xi_1(\alpha);\alpha\right)<0
\end{equation}
and
\begin{equation}\label{cdsfm}
\mathcal{T}_1(-\varsigma(\alpha);\alpha)>0
\end{equation}
respectively. According to the Taylor formula with the integral form of the remainder \cite[Theorem 3.9]{JJCA},
\begin{equation}\label{cdsfafadsfm}
\begin{aligned}
\mathcal{T}_1\left(\sqrt{\phi_1(\alpha)}-\xi_1(\alpha);\alpha\right)&=\mathcal{T}_1(-\xi_1(\alpha);\alpha)+
\int^1_0\mathcal{T}_{1x}\left(t\sqrt{\phi_1(\alpha)}-\xi_1(\alpha);\alpha\right)dt\sqrt{\phi_1(\alpha)}\\
&=\mathcal{T}_1(-\xi_1(\alpha);\alpha)+
\left(1+\int^1_0\sigma_x\left(t\sqrt{\phi_1(\alpha)}-\xi_1(\alpha),0;\alpha\right)dt\right)\sqrt{\phi_1(\alpha)}\\
&=\phi_2(\alpha)+\left(1+\xi_2(\sqrt{\phi_1(\alpha)},\alpha)\right)\sqrt{\phi_1(\alpha)},
\end{aligned}
\end{equation}
where
$\xi_2(z,\alpha):=\int^1_0\sigma_x(tz-\xi_1(\alpha),0;\alpha)dt.$
By Lemma~\ref{dafsferwerc} and (\ref{489jvdcsd}), $\xi_2(0,0)=\sigma_x(-a,0;0)=0$ and $\xi_2(z,\alpha)$ is smooth because the integrand is smooth under the assumption on the smoothness of the upper subsystem. Thus $\varrho_2(\alpha)\in(\sqrt{\phi_1(\alpha)}-\xi_1(\alpha),-a+\varepsilon)$ is equivalent to (\ref{94fdcds}) by (\ref{cafadsfm}) and (\ref{cdsfafadsfm}).

Letting $\mathcal{P}(x;\alpha)$ be the transition map from $\{(x,0):-\sqrt{\phi_1(\alpha)}-\xi_1(\alpha)\le x\le\sqrt{\phi_1(\alpha)}-\xi_1(\alpha)\}$ to $\{(x,0):-\varsigma(\alpha)\le x\le-\sqrt{\phi_1(\alpha)}-\xi_1(\alpha)\}$ defined by the backward orbits of the upper subsystem, we get
$$
\begin{aligned}
\mathcal{P}(x;\alpha)=&~-\sqrt{\phi_1(\alpha)}-\xi_1(\alpha)-\left(x+\sqrt{\phi_1(\alpha)}+\xi_1(\alpha)\right)\\
&~+\int^1_0\mathcal{P}_{xx}\left(-\sqrt{\phi_1(\alpha)}-\xi_1(\alpha)+t(x+\sqrt{\phi_1(\alpha)}+\xi_1(\alpha));\alpha\right)
(1-t)dt\left(x+\sqrt{\phi_1(\alpha)}+\xi_1(\alpha)\right)^2
\end{aligned}
$$
by the Taylor formula with the integral form of the remainder and the fact $\mathcal{P}_x(-\sqrt{\phi_1(\alpha)}-\xi(\alpha);\alpha)=-1$ from \cite[p.236]{AFF}. This concludes that
\begin{equation}\label{382cef}
\begin{aligned}
-\varsigma(\alpha)&=\mathcal{P}\left(\sqrt{\phi_1(\alpha)}-\xi_1(\alpha);\alpha\right)\\
&=-3\sqrt{\phi_1(\alpha)}-\xi_1(\alpha)+4\int^1_0\mathcal{P}_{xx}\left(-\sqrt{\phi_1(\alpha)}-\xi_1(\alpha)+2t\sqrt{\phi_1(\alpha)};\alpha\right)
(1-t)dt\left(\sqrt{\phi_1(\alpha)}\right)^2\\
&=-\left(3+\widetilde\xi_3(\sqrt{\phi_1(\alpha)},\alpha)\right)\sqrt{\phi_1(\alpha)}-\xi_1(\alpha),
\end{aligned}
\end{equation}
where
$$\widetilde\xi_3(z,\alpha)=-4z\int^1_0\mathcal{P}_{xx}(-z-\xi_1(\alpha)+2tz;\alpha)(1-t)dt.$$
Clearly, $\widetilde\xi_3(0,0)=0$ and $\widetilde\xi_3(z,\alpha)$ is smooth because the integrand is smooth under the assumption on the smoothness of the upper subsystem. Note that
\begin{equation}\label{ciwrecdsf}
\begin{aligned}
\mathcal{T}_1(-\varsigma(\alpha);\alpha)=&~\sigma\left(\sqrt{\phi_1(\alpha)}-\xi_1(\alpha),0;\alpha\right)-\varsigma(\alpha)\\
=&~\mathcal{T}_1\left(\sqrt{\phi_1(\alpha)}-\xi_1(\alpha),0;\alpha\right)-\sqrt{\phi_1(\alpha)}+\xi_1(\alpha)-\varsigma(\alpha)\\
=&~\phi_2(\alpha)-\left(3+\xi_3(\sqrt{\phi_1(\alpha)},\alpha)\right)\sqrt{\phi_1(\alpha)}
\end{aligned}
\end{equation}
by (\ref{cjt43}), (\ref{cdsfafadsfm}), (\ref{382cef}) and $\sigma(-\varsigma(\alpha),0;\alpha)=\sigma(\sqrt{\phi_1(\alpha)}-\xi_1(\alpha),0;\alpha)$, where $\xi_3(z,\alpha):=\widetilde\xi_3(z,\alpha)-\xi_2(z,\alpha)$.
Thus $\varrho_2(\alpha)\in(-a-\varepsilon, -\varsigma(\alpha))$ is equivalent to (\ref{94fdcdsdae}) by (\ref{cdsfm}) and (\ref{ciwrecdsf}). Consequently, this claim is proved, and the crossing cycle existing for condition (1c) does not enclose any of the regular-folds $T^\pm_{iv}$ and $T^\pm_{v}$, while the crossing cycle existing for condition (1d) encloses all of them. The last four paragraphs yield the second part of statement (1).

Due to $f^+(-a,0;0)>0$, it follows from Lemma~\ref{sldynamcierwe3} that a necessary condition for system (\ref{Z2system}) to have
a critical crossing cycle in $\mathcal{A}$ is $\phi_1(\alpha)\ge0$. In addition, according to the $\mathbb{Z}_2$-symmetry of (\ref{Z2system}),
a critical crossing cycle of system (\ref{Z2system}) in $\mathcal{A}$ must cross $\Sigma$ at the regular-cusps $T^-_c$ and $T^+_c$ if $\phi_1(\alpha)=0$, while if $\phi_1(\alpha)>0$, it crosses $\Sigma$ either at the regular-folds $T^-_v$ and $T^+_v$, or at $(-\varsigma(\alpha),0)$ and $(\varsigma(\alpha),0)$.
Hence, system (\ref{Z2system}) for $\alpha\in U$ has at most one critical crossing cycle in $\mathcal{A}$ and there is a critical crossing cycle if and only if
either $\phi_1(\alpha)=0$ and $\mathcal{T}_1(-\xi_1(\alpha);\alpha)=0$, or $\phi_1(\alpha)>0$ and $\mathcal{T}_1(\sqrt{\phi_1(\alpha)}-\xi_1(\alpha);\alpha)=0$, or $\phi_1(\alpha)>0$ and $\mathcal{T}_1(-\varsigma(\alpha);\alpha)=0$, which conclude respectively the conditions (2a), (2b) and (2c) by (\ref{cr334e}), (\ref{cdsfafadsfm}) and (\ref{ciwrecdsf}). Combining the dynamics on $\Sigma$ obtained in Lemma~\ref{sldynamcierwe3} and the fact $\partial\mathcal{T}_1(x_0;\alpha)/\partial x=1+\mathcal{O}(\|\alpha\|)>0$ for $x_0\in\{-\xi_1(\alpha), \sqrt{\phi_1(\alpha)}-\xi_1(\alpha), -\varsigma(\alpha)\}$, we obtain the stability of the critical crossing cycle as stated in (2a), (2b) and (2c).

Finally, due to $f^+(-a,0;0)>0$, a sliding cycle of system (\ref{Z2system}) in $\mathcal{A}$ must slide on $S^-_2$ and enter into $\Sigma^+$ at $T^-_v$ until it reaches $\Sigma$ again at a point in $S^+_2\cup T^+_{iv}$ by the sliding dynamics obtained in Lemma~\ref{sldynamcierwe3} and the $\mathbb{Z}_2$-symmetry of (\ref{Z2system}).
Hence, system (\ref{Z2system}) for $\alpha\in U$ has at most one sliding cycle in $\mathcal{A}$, which is stable from the stability of its sliding segments, and a necessary and sufficient condition for system (\ref{Z2system}) to have a sliding cycle in $\mathcal{A}$ is
\begin{equation}\label{38c2fjd}
\phi_1(\alpha)>0,\qquad -\sqrt{\phi_1(\alpha)}+\xi_1(\alpha)<\sigma\left(\sqrt{\phi_1(\alpha)}-\xi_1(\alpha),0;\alpha\right)<\varsigma(\alpha).
\end{equation}
In addition, if
\begin{equation}\label{45dssad}
\sigma\left(\sqrt{\phi_1(\alpha)}-\xi_1(\alpha);\alpha\right)-\left(\sqrt{\phi_1(\alpha)}+\xi_1(\alpha)\right)<0~ ({\rm resp}.~ =0, >0),
\end{equation}
the sliding cycle arrives at a point of $S^+_2$ from $\Sigma^+$ (resp. at $T^+_{iv}$ from $\Sigma^+$, at a point of $S^+_2$ from $\Sigma^-$).
Note that (\ref{38c2fjd}) is equivalent to
\begin{equation}\label{4dad5dssad}
\phi_1(\alpha)>0,\qquad \mathcal{T}_1\left(\sqrt{\phi_1(\alpha)}-\xi_1(\alpha);\alpha\right)>0,\qquad\mathcal{T}_1(-\varsigma(\alpha);\alpha)<0,
\end{equation}
and (\ref{45dssad}) is equivalent to
\begin{equation}\label{4sf5dssad}
\mathcal{T}_1\left(\sqrt{\phi_1(\alpha)}-\xi_1(\alpha);\alpha\right)-2\sqrt{\phi_1(\alpha)}<0~~({\rm resp}. =0, >0)
\end{equation}
by (\ref{cjt43}) and (\ref{ciwrecdsf}). Thus, recalling (\ref{cdsfafadsfm}) and (\ref{ciwrecdsf}), we get conditions (3a), (3b) and (3c) from (\ref{4dad5dssad}) and (\ref{4sf5dssad}).
\end{proof}

\begin{proof}[{\bf Proof of Theorem~\ref{codim-2-cusp-bifur}}]
We only prove Theorem~\ref{codim-2-cusp-bifur} for $f^+(-a,0;0)>0$ because of similarity. For any sufficiently small annulus $\mathcal{A}$ of $\Gamma_0$, we take $U$ as the neighborhood of $\alpha=0$ given in Lemma~\ref{eiwu343cd}, $\phi_1(\alpha)$ and $\phi_2(\alpha)$ for $\alpha\in U$ as the functions given in (\ref{489jvdcsd}) and (\ref{w234dsvsf}) respectively.
If $(\zeta_1,\zeta_2,\cdots,\zeta_m)$ and $(\eta_1,\eta_2,\cdots,\eta_m)$ are linearly independent, there are $j_1,j_2\in\{1,2,\cdots,m\}$ such that
\begin{equation}\label{cfw34wefw}
\zeta_{j_1}\eta_{j_2}-\eta_{j_1}\zeta_{j_2}\ne0.
\end{equation}
Let
\begin{equation}\label{cdsn34afr}
\phi_i(\alpha):=\alpha_{j_i},\quad i=3,4,\cdots,m
\end{equation}
such that $\alpha_{j_i}\ne\alpha_{j_1}$, $\alpha_{j_i}\ne\alpha_{j_2}$ for all $i=3,4,\cdots,m$ and $\alpha_{j_k}\ne\alpha_{j_l}$ for $k\ne l$. Then the Jacobian matrix of
\begin{equation}\label{cdsn3465}
\beta=(\beta_1,\beta_2,\cdots,\beta_m)=(\phi_1(\alpha),\phi_2(\alpha),\cdots,\phi_m(\alpha))
\end{equation}
at $\alpha=0$ is nonsingular from (\ref{489jvdcsd}), (\ref{w234dsvsf}), (\ref{cfw34wefw}) and (\ref{cdsn34afr}). Thus (\ref{cdsn3465}) is a diffeomorphism from $U$ to its range $V$, where $U$ can be reduced if necessary.
Finally, taking
\begin{eqnarray*}
\begin{aligned}
GS:&=\Big\{\beta\in V:~\beta_1>0,\beta_2=\left(3+\xi_3(\sqrt{\beta_1},\alpha^{-1}(\beta))\right)\sqrt{\beta_1}\Big\},\\
SS:&=\Big\{\beta\in V:~\beta_1>0,\beta_2=\left(1-\xi_2(\sqrt{\beta_1},\alpha^{-1}(\beta))\right)\sqrt{\beta_1}\Big\},\\
CS:&=\Big\{\beta\in V:~\beta_1>0,\beta_2=\left(-1-\xi_2(\sqrt{\beta_1},\alpha^{-1}(\beta))\right)\sqrt{\beta_1}\Big\},\\
\end{aligned}
\end{eqnarray*}
where $\alpha^{-1}(\beta)$ denotes the inverse of (\ref{cdsn3465}), and combining the dynamics on $\Sigma$ and the information of various cycles given in Lemma~\ref{sldynamcierwe3} and Lemma~\ref{eiwu343cd} respectively, we get that for any $(\beta_3^*,\beta_4^*,\cdots,\beta_m^*)\in V^*$ the bifurcation diagram of system $\left.(\ref{Z2system})\right|_{\mathcal{A}}$ on the hyperplane $(\beta_3,\beta_4,\cdots,\beta_m)=(\beta_3^*,\beta_4^*,\cdots,\beta_m^*)$ is the one shown in Figure~\ref{codim-2bifurdia2wca} when $f^+(-a,0;0)>0$.
\end{proof}

%%%%%%%%%%%%%%%%%%%%%%%%%%%%%%%%%%%%%%%%%%%%%%%%%%%%%%%%%%%%
\section{Proof of Theorem~\ref{codim-2-fold-bifur1}}
\setcounter{equation}{0}
\setcounter{lm}{0}
\setcounter{thm}{0}
\setcounter{rmk}{0}
\setcounter{df}{0}
\setcounter{cor}{0}

We also only prove Theorem~\ref{codim-2-fold-bifur1} for $f^+(-a,0;0)>0$, as the case of $f^+(-a,0;0)<0$ can be treated similarly by employing the methodology of this section. We start by studying the dynamics on $\Sigma$.

\begin{lm}\label{dvn23fs}
Under the assumptions {\bf(H1)}, {\bf(H2)$^\prime$} and $f^+(-a,0;0)>0$, for any sufficiently small $\varepsilon_5>0$ there exists a neighborhood $U_5$ of $\alpha=0$ and two smooth functions
\begin{equation}\label{23vdf24sff}
\vartheta^-_1(\alpha)=a+\sum_{i=1}^m\frac{g^+_{\alpha_i}(-a,0;0)}{g^+_x(-a,0;0)}\alpha_i+\mathcal{O}(\|\alpha\|^2),\qquad
\vartheta^+_1(\alpha)=a-\sum_{i=1}^m\frac{g^+_{\alpha_i}(a,0;0)}{g^+_x(a,0;0)}\alpha_i+\mathcal{O}(\|\alpha\|^2)
\end{equation}
defined in $U_5$ such that $T^\pm_u:=(\pm\vartheta^\pm_1(\alpha),0)\in I^\pm_3:=\{(x,0):|x\mp a|<\varepsilon_5\}$ is a visible fold of the upper subsystem of $(\ref{Z2system})$ with $f^+(\pm\vartheta^\pm_1(\alpha),0;\alpha)>0$ and the following statements hold for $\alpha\in U_5$, letting
\begin{equation}\label{qq485dscs}
\varphi_1(\alpha):=-\vartheta^-_1(\alpha)+\vartheta^+_1(\alpha)=\sum_{i=1}^m\mu_i\alpha_i+\mathcal{O}(\|\alpha\|^2),
\end{equation}
and
$$
\begin{aligned}
C^-_{d}:=&~\left\{(x,0):-a-\varepsilon_5<x<\min\{-\vartheta^+_1(\alpha),-\vartheta^-_1(\alpha)\}\right\},\\
S^-_{3}:=&~\left\{(x,0):\min\{-\vartheta^+_1(\alpha),-\vartheta^-_1(\alpha)\}<x<\max\{-\vartheta^+_1(\alpha),-\vartheta^-_1(\alpha)\}\right\},\\
C^-_{u}:=&~\left\{(x,0):\max\{-\vartheta^+_1(\alpha),-\vartheta^-_1(\alpha)\}<x<-a+\varepsilon_5\right\},\\
C^+_{d}:=&~\left\{(x,0):a-\varepsilon_5<x<\min\{\vartheta^+_1(\alpha),\vartheta^-_1(\alpha)\}\right\},\\
S^+_{3}:=&~\left\{(x,0):\min\{\vartheta^+_1(\alpha),\vartheta^-_1(\alpha)\}<x<\max\{\vartheta^+_1(\alpha),\vartheta^-_1(\alpha)\}\right\},\\
C^+_{u}:=&~\left\{(x,0):\max\{\vartheta^+_1(\alpha),\vartheta^-_1(\alpha)\}<x<a+\varepsilon_5\right\}.
\end{aligned}
$$
\begin{itemize}
\item[{\rm(1)}] If $\varphi_1(\alpha)=0$, $T^-_u$ and $T^-_l:=(-\vartheta^+_1(\alpha),0)$, namely the $\mathbb{Z}_2$-symmetric counterpart
 of $T^+_u$, coincide to be a fold-fold of system $(\ref{Z2system})$, which splits $I^-_3$ into a downward crossing segment $C^-_{d}$ and an upward crossing segment $C^-_{u}$. If $\varphi_1(\alpha)\ne0$, $T^-_u$ and $T^-_l$ are two regular-folds of system $(\ref{Z2system})$ and they split $I^-_3$ into a downward crossing segment $C^-_{d}$, a stable $($resp. an unstable$)$ sliding segment $S^-_{3}$ and an upward crossing segment $C^-_{u}$ for $\varphi_1(\alpha)>0$ $(resp.~<0)$. Moreover, there is a unique pseudo-equilibrium on $S^-_{3}$, which is a
 pseudo-saddle and lies at $E^-_p:=(-\varpi(\alpha),0)$ with
 \begin{equation}\label{943ce23dew3dsf}
 \varpi(\alpha)=a+\sum^m_{i=1}\frac{g^+_{\alpha_i}(-a,0;0)f^+(a,0;0)-g^+_{\alpha_i}(a,0;0)f^+(-a,0;0)}{g^+_x(-a,0;0)f^+(a,0;0)+g^+_x(a,0;0)f^+(-a,0;0)}\alpha_i
    +\mathcal{O}(\|\alpha\|^2).
 \end{equation}
\item[{\rm(2)}] The dynamics exhibited on $I^+_3$ and $I^-_3$ are $\mathbb{Z}_2$-symmetric with respect to $O$. Hence, if $\varphi_1(\alpha)=0$, $T^+_u$ and $T^+_l:=(\vartheta^-_1(\alpha),0)$, namely the $\mathbb{Z}_2$-symmetric counterpart of $T^-_u$, coincide to be a fold-fold of system $(\ref{Z2system})$, which splits $I^+_3$ into a downward crossing segment $C^+_{d}$ and an upward crossing segment $C^+_{u}$. If $\varphi_1(\alpha)\ne0$, $T^+_u$ and $T^+_l$ are two regular-folds of system $(\ref{Z2system})$ and they split $I^+_3$ into a downward crossing segment $C^+_{d}$, a stable $($resp. an unstable$)$ sliding segment $S^+_{3}$ and an upward crossing segment $C^+_{u}$ for $\varphi_1(\alpha)>0$ $(resp.~<0)$. Moreover, there is a unique pseudo-equilibrium on $S^+_{3}$, which is a
 pseudo-saddle and lies at $E^+_p:=(\varpi(\alpha),0)$.
\end{itemize}
\end{lm}

\begin{proof}
For any sufficiently small $\varepsilon_5>0$, by the Implicit Function Theorem there is a neighborhood $U_5$ of $\alpha=0$ and a smooth function $\vartheta^-_1(\alpha)$ (resp. $\vartheta^+_1(\alpha)$) defined in $U_5$ and satisfying (\ref{23vdf24sff}) such that $x=-\vartheta^-_1(\alpha)$ (resp. $x=\vartheta^+_1(\alpha)$) is the unique zero of $g^+(x,0;\alpha)=0$ in $|x+a|<\varepsilon_5$ (resp. $|x-a|<\varepsilon_5$) for $\alpha\in U_5$. Moreover, according to {\bf(H1)}, {\bf(H2)$^\prime$}, $f^+(-a,0;0)>0$ and the sign-preserving property of continuous functions, $U_5$ can be chosen to ensure
\begin{equation}\label{ioru34fds}
f^+(x,0;\alpha)>0, \qquad g^+_{x}(x,0;\alpha)>0\qquad {\rm for}~~(x,0)\in I^-_3\cup I^+_3~~{\rm and}~~\alpha\in U_5,
\end{equation}
So $T^-_u\in I^-_3$ (resp. $T^+_u\in I^+_3$) is a visible fold of the upper subsystem of (\ref{Z2system}) satisfying $f^+(-\vartheta^-_1(\alpha),0;\alpha)>0$ (resp. $f^+(\vartheta^+_1(\alpha),0;\alpha)>0$) for $\alpha\in U_5$, i.e. the first part of Lemma~\ref{dvn23fs} holds.

Due to the $\mathbb{Z}_2$-symmetry, $T^-_l\in I^-_3$ is a visible fold of the lower subsystem of (\ref{Z2system}) satisfying $f^+(\vartheta^+_1(\alpha),0;\alpha)>0$ for $\alpha\in U_5$.
Thus $T^-_u$ and $T^-_l$ coincide to be a fold-fold that splits $I^-_3$ into $C^-_{d}$ and $C^-_{u}$ if $\varphi_1(\alpha)=0$, while if $\varphi_1(\alpha)\ne0$, they are two regular-folds and split $I^-_3$ into $C^-_{d}$, $S^-_{3}$ and $C^-_{u}$. In this case, a direct verification concludes that
$C^-_{d}$ is a downward crossing segment, $C^-_{u}$ is an upward crossing segment and $S^-_{3}$ is a stable $($resp. an unstable$)$ sliding segment for $\varphi_1(\alpha)>0$ (resp. $<0$).

The remaining part is to analyze the sliding dynamics of system (\ref{Z2system}) existing for $\varphi_1(\alpha)\ne0$. From (\ref{generalslidd}), the sliding system of (\ref{Z2system}) is given by
$$
(\dot x,\dot y)=\left(-\frac{f^s(x;\alpha)}{g^+(x,0;\alpha)+g^+(-x,0;\alpha)},0\right)
$$
for $(x,0)\in \overline{S^-_{3}\cup S^+_{3}}$ and $\alpha\in U_5$, where
$
f^s(x;\alpha):=g^+(x,0;\alpha)f^+(-x,0;\alpha)-g^+(-x,0;\alpha)f^+(x,0;\alpha).
$
Due to $g^+(-\vartheta^-_1(\alpha),0;\alpha)=0$, $g^+(\vartheta^+_1(\alpha),0;\alpha)=0$ and (\ref{ioru34fds}), we get
$f^s(-\vartheta^-_1(\alpha);\alpha)>0>f^s(-\vartheta^+_1(\alpha);\alpha)$ for $\varphi_1(\alpha)>0$ and $f^s(-\vartheta^-_1(\alpha);\alpha)<0<f^s(-\vartheta^+_1(\alpha);\alpha)$ for
$\varphi_1(\alpha)<0$. This means that $f^s(x;\alpha)$ must have a zero in the interval $(\min\{-\vartheta^-_1(\alpha),-\vartheta^+_1(\alpha)\},\max\{-\vartheta^-_1(\alpha),-\vartheta^+_1(\alpha)\})$.
Besides, since
$f^s_x(-a;0)=g^+_x(-a,0;0)f^+(a,0;0)+g^+_x(a,0;0)f^+(-a,0;0)>0$ from {\bf(H1)} and {\bf(H2)$^\prime$}, $U_5$ can be chosen such that $f^s_x(x;\alpha)>0$ for
$x\in(\min\{-\vartheta^-_1(\alpha),-\vartheta^+_1(\alpha)\},\max\{-\vartheta^-_1(\alpha),-\vartheta^+_1(\alpha)\})$ and $\alpha\in U_5$. Thus $f^s(x;\alpha)$ exactly has one zero in the interval $(\min\{-\vartheta^-_1(\alpha),-\vartheta^+_1(\alpha)\},\max\{-\vartheta^-_1(\alpha),-\vartheta^+_1(\alpha)\})$, i.e. there is a unique pseudo-equilibrium in $S^-_{3}$. Since $f^s(-a;0)=0$ and $f^s_x(-a;0)>0$, the pseudo-equilibrium lies at $E^-_p$ by the Implicit Function Theorem. Furthermore, recalling the stability of $S^-_{3}$, we get that $E^-_p$ is a pseudo-saddle due to $\dot x|_{(-\vartheta^-_1(\alpha),0)}=f^+(-\vartheta^-_1(\alpha),0;\alpha)>0$ and $\dot x|_{(-\vartheta^+_1(\alpha),0)}=-f^+(\vartheta^+_1(\alpha),0;\alpha)<0$. The proof of statement (1) is completed.
Finally, using the $\mathbb{Z}_2$-symmetry of system (\ref{Z2system}), we get statement (2) from statement (1) directly.
\end{proof}

Considering the map $\sigma^+(x,y;\alpha)$
(resp. $\sigma^-(x,y;\alpha)$) given in Lemma~\ref{cnkrcdwerc} for $\|(x+a,y)\|<\varepsilon_1$ (resp. $\|(x-a,y)\|<\varepsilon_1$) and $\alpha\in U_1$,
we define
\begin{equation}\label{95vvcs}
\mathcal{T}_2(x;\alpha):=\sigma^+(x,0;\alpha)-\sigma^-(-x,0;\alpha)
\end{equation}
for $|x+a|<\varepsilon_1$ and $\alpha\in U_1$. As will be seen later, the zeros of $\mathcal{T}_2$ are closely related to the crossing cycles and critical crossing cycles bifurcating from $\Gamma_0$. Thus we next study the zeros of $\mathcal{T}_2$.

\begin{lm}\label{458vsvsdf}
Consider the map $\mathcal{T}_2(x;\alpha)$ for $|x+a|<\varepsilon_1$ and $\alpha\in U_1$. Under the assumption of Theorem~\ref{codim-2-fold-bifur1}, for any sufficiently small $\varepsilon_6\in(0,\varepsilon_1)$ there exists a neighborhood $U_6\subset U_1$ of $\alpha=0$ and smooth functions
\begin{equation}\label{ew8csdr}
\begin{aligned}
\varphi_2(\alpha)&=\sum_{i=1}^m\frac{2}{\Delta}\kappa_i\alpha_i+\mathcal{O}(\|\alpha\|^2),\\
\vartheta_2(\alpha)&=a+\sum_{i=1}^m\frac{g^+_{\alpha_i}(-a,0;0)\lambda(0)+g^+_{\alpha_i}(a,0;0)}{\Delta}\alpha_i
+\mathcal{O}(\|\alpha\|^2)
\end{aligned}
\end{equation}
defined in $U_6$ such that the following statements hold for $\alpha\in U_6$.
\begin{itemize}
\item[{\rm(1)}] If $\varphi_2(\alpha)=0$, $\mathcal{T}_2(x;\alpha)$ has a unique zero $x=x^\star_0:=-\vartheta_2(\alpha)$ in $\{x:|x+a|<\varepsilon_6\}$,
which is of multiplicity two with $\mathcal{T}_{2xx}(x^\star_0;\alpha)<0$.
\item[{\rm(2)}] If $\varphi_2(\alpha)<0$, $\mathcal{T}_2(x;\alpha)$ has no zeros in $\{x:|x+a|<\varepsilon_6\}$ and, more precisely, $\mathcal{T}_2(x;\alpha)<0$.
\item[{\rm(3)}] If $\varphi_2(\alpha)>0$, $\mathcal{T}_2(x;\alpha)$ has exactly two zeros
$x=x^\star_\pm:=\pm\sqrt{\varphi_2(\alpha)}-\vartheta_2(\alpha)$ in $\{x:|x+a|<\varepsilon_6\}$, which are simple with
    $\mathcal{T}_{2x}(x^\star_\pm;\alpha)\lessgtr0$.
\end{itemize}
\end{lm}

\begin{proof}
Under the assumptions {\bf (H0)}, {\bf (H1)} and ${\bf(H2)^\prime}$, we can apply Lemma~\ref{Transipero} to write
\begin{equation}\label{4fsdcvdf45}
\begin{aligned}
\mathcal{T}_2(x;\alpha)=&\sum_{i=1}^m\frac{\kappa_i^--\kappa_i^+}{g^+(b,c;0)}\alpha_i+\mathcal{O}(\|\alpha\|^2)
+\left(\sum_{i=1}^m\frac{g^+_{\alpha_i}(-a,0;0)\lambda^+(0)+g^+_{\alpha_i}(a,0;0)\lambda^-(0)}{g^+(b,c;0)}\alpha_i+\mathcal{O}(\|\alpha\|^2)\right)(x+a)\\
&+\left(\frac{g^+_x(-a,0;0)\lambda^+(0)-g^+_x(a,0;0)\lambda^-(0)}{2g^+(b,c;0)}+\mathcal{O}(\|\alpha\|)\right)(x+a)^2+\mathcal{O}((x+a)^3)\\
=&\sum_{i=1}^m\frac{-\kappa_i\lambda^-(0)}{g^+(b,c;0)}\alpha_i+\mathcal{O}(\|\alpha\|^2)
+\left(\sum_{i=1}^m\frac{\lambda^-(0)\left(g^+_{\alpha_i}(-a,0;0)\lambda(0)+g^+_{\alpha_i}(a,0;0)\right)}{g^+(b,c;0)}\alpha_i+\mathcal{O}(\|\alpha\|^2)\right)(x+a)\\
&+\left(\frac{\lambda^-(0)\Delta}{2g^+(b,c;0)}+\mathcal{O}(\|\alpha\|)\right)(x+a)^2+\mathcal{O}((x+a)^3)\\
\end{aligned}
\end{equation}
Perform a linear coordinate shift by introducing a new variable $\hat x=J(x;\alpha):=x+a+\hat\upsilon,$
where $\hat\upsilon=\hat\upsilon(\alpha)$ is a priori unknown function that will be determined later. Thus
\begin{equation}\label{ewr45vsdf}
\begin{aligned}
\mathcal{T}_2\circ J^{-1}(\hat x;\alpha)=&-\sum_{i=1}^m\frac{\kappa_i\lambda^-(0)}{g^+(b,c;0)}\alpha_i+\mathcal{O}(\|\alpha\|^2)\\
&-\left(\sum_{i=1}^m\frac{\lambda^-(0)\left(g^+_{\alpha_i}(-a,0;0)\lambda(0)+g^+_{\alpha_i}(a,0;0)\right)}{g^+(b,c;0)}\alpha_i+\mathcal{O}(\|\alpha\|^2)\right)\hat\upsilon
+\mathcal{O}(\hat\upsilon^2)\\
&+\Bigg\{\sum_{i=1}^m\frac{\lambda^-(0)\left(g^+_{\alpha_i}(-a,0;0)\lambda(0)+g^+_{\alpha_i}(a,0;0)\right)}{g^+(b,c;0)}\alpha_i+\mathcal{O}(\|\alpha\|^2)\\
&-\left(\frac{\lambda^-(0)\Delta}{g^+(b,c;0)}+\mathcal{O}(\|\alpha\|)\right)\hat\upsilon+\mathcal{O}(\hat\upsilon^2)\Bigg\}\hat x\\
&+\left(\frac{\lambda^-(0)\Delta}{2g^+(b,c;0)}+\mathcal{O}(\|\alpha\|)+\mathcal{O}(\hat\upsilon)\right)\hat x^2+\mathcal{O}(\hat x^3).
\end{aligned}
\end{equation}
Letting $K(\alpha,\hat\upsilon)$ be the coefficient of $\hat x$ in (\ref{ewr45vsdf}), we get $K(0,0)=0$,
$$
\begin{aligned}
\frac{\partial K(0,0)}{\partial\hat\upsilon}=-\frac{\lambda^-(0)\Delta}{g^+(b,c;0)},\qquad
\frac{\partial K(0,0)}{\partial\alpha_i}=\frac{\lambda^-(0)\left(g^+_{\alpha_i}(-a,0;0)\lambda(0)+g^+_{\alpha_i}(a,0;0)\right)}{g^+(b,c;0)}.
\end{aligned}
$$
Note that $\partial K(0,0)/\partial\hat\upsilon>0$, since $\Delta>0$ from {\bf (H3)} and $g^+(b,c;0)<0$ from Lemma~\ref{Transipero}. Then by the Implicit Function Theorem there is a neighborhood $\widetilde U_6$ of $\alpha=0$, a unique and smooth function
$$
\hat\upsilon(\alpha)=\sum_{i=1}^m\frac{g^+_{\alpha_i}(-a,0;0)\lambda(0)+g^+_{\alpha_i}(a,0;0)}{\Delta}\alpha_i
+\mathcal{O}(\|\alpha\|^2)
$$
defined in $\widetilde U_6$ such that $K(\alpha,\hat\upsilon(\alpha))=0$ for $\alpha\in \widetilde U_6$.

Substituting $\hat\upsilon(\alpha)$ into (\ref{ewr45vsdf}), we obtain
$$
\mathcal{T}_2\circ J^{-1}(\hat x;\alpha)=-\sum_{i=1}^m\frac{\kappa_i\lambda^-(0)}{g^+(b,c;0)}\alpha_i+\mathcal{O}(\|\alpha\|^2)
+\left(\frac{\lambda^-(0)\Delta}{2g^+(b,c;0)}+\mathcal{O}(\|\alpha\|)\right)\hat x^2+\mathcal{O}(\hat x^3).
$$
Due to $\lambda^-(0)\Delta/(2g^+(b,c;0))<0$, regarding $\hat x^2$ as a new variable and using the Implicit Function Theorem again, we can reduce $\widetilde U_6$ such that for $\alpha\in \widetilde U_6$ there is a unique and smooth function $\varphi_2(\alpha)$ satisfying (\ref{ew8csdr})
such that $\mathcal{T}_2\circ J^{-1}(\hat x;\alpha)$ near $\hat x=0$ has a unique zero $\hat x=\hat x^\star_0:=0$ if $\varphi_2(\alpha)=0$, no zeros and, more precisely, $\mathcal{T}_2\circ J^{-1}(\hat x;\alpha)<0$ if $\varphi_2(\alpha)<0$, and exactly two zeros $\hat x=\hat x^\star_\pm:=\pm\sqrt{\varphi_2(\alpha)}$ if $\varphi_2(\alpha)>0$.
Moreover,
$$
\begin{aligned}
\frac{\partial \mathcal{T}_2\circ J^{-1}(\hat x^\star_0;\alpha)}{\partial \hat x}=&~0,\qquad \frac{\partial^2 \mathcal{T}_2\circ J^{-1}(\hat x^\star_0;\alpha)}{\partial \hat x^2}=\frac{\lambda^-(0)\Delta}{g^+(b,c;0)}+\mathcal{O}(\|\alpha\|)<0,\\
\frac{\partial \mathcal{T}_2\circ J^{-1}(\hat x^\star_\pm;\alpha)}{\partial \hat x}=&~\pm\left(\frac{\lambda^-(0)\Delta}{g^+(b,c;0)}+\mathcal{O}(\|\alpha\|)+\mathcal{O}(\sqrt{\varphi_2(\alpha)})\right)\sqrt{\varphi_2(\alpha)}\lessgtr0.
\end{aligned}
$$
Finally, for any sufficiently small $\varepsilon_6\in(0,\varepsilon_1)$ the above analysis allows us to choose a suitable $U_6\subset \widetilde U_6$ such that this lemma holds, taking $\vartheta_2(\alpha):=a+\hat\upsilon(\alpha)$.
\end{proof}

\begin{lm}\label{rut45scd}
Consider the zeros of $\mathcal{T}_2(x;\alpha)$ given in Lemma~\ref{458vsvsdf}, namely $x=x^\star_0$ existing for $\varphi_2(\alpha)=0$ and $x=x^\star_\pm$ existing for $\varphi_2(\alpha)>0$. Let
$$
\begin{aligned}
\mathcal{I}_1:=&~(\max\{-\vartheta^-_1(\alpha),-\vartheta^+_1(\alpha)\},-a+\varepsilon_7),\\ \mathcal{I}_2:=&~(\min\{-\vartheta^-_1(\alpha),-\vartheta^+_1(\alpha)\},\max\{-\vartheta^-_1(\alpha),-\vartheta^+_1(\alpha)\}),\\
\mathcal{I}_3:=&~(-a-\varepsilon_7,\min\{-\vartheta^-_1(\alpha),-\vartheta^+_1(\alpha)\}),\\
\end{aligned}
$$
where $\varepsilon_7:=\min\{\varepsilon_5,\varepsilon_6\}$. Under the assumption of Theorem~\ref{codim-2-fold-bifur1} and $f^+(-a,0;0)>0$, if $(\mu_1,\mu_2,\cdots,\mu_m)\ne0$, there exists a neighborhood $U_7\subset U_5\cap U_6$ of $\alpha=0$ and two smooth functions
\begin{equation}\label{4389vj}
\vartheta_3(\alpha)=\left(1+\frac{g^+_x(a,0;0)}{\Delta}\right)^2+\mathcal{O}(\|\alpha\|),\qquad  \vartheta_4(\alpha)=\left(\frac{g^+_x(a,0;0)}{\Delta}\right)^2+\mathcal{O}(\|\alpha\|)
\end{equation}
defined in $U_7$ such that for $\alpha\in U_7$,
$$
\begin{aligned}
&x^\star_0
\left\{
\begin{aligned}
&\in \mathcal{I}_1~\quad\quad &&if~\varphi_1(\alpha)>0,\varphi_2(\alpha)=0,\\
&=-\vartheta^-_1(\alpha)=-\vartheta^+_1(\alpha)~\quad\quad &&if~\varphi_1(\alpha)=0,\varphi_2(\alpha)=0,\\
&\in \mathcal{I}_3~\quad\quad &&if~\varphi_1(\alpha)<0,\varphi_2(\alpha)=0;
\end{aligned}
\right.\\
&x^\star_+\left\{
\begin{aligned}
&\in \mathcal{I}_1~~\qquad\qquad\qquad &&if~\varphi_1(\alpha)\ge0,\varphi_2(\alpha)>0~or~\varphi_1(\alpha)<0,\varphi_2(\alpha)>\vartheta_3(\alpha)\varphi_1^2(\alpha),\\
&=-\vartheta^+_1(\alpha)~~\qquad\qquad\qquad &&if~\varphi_1(\alpha)<0,\varphi_2(\alpha)=\vartheta_3(\alpha)\varphi_1^2(\alpha),\\
&\in \mathcal{I}_2~~\qquad\qquad\qquad  &&if~\varphi_1(\alpha)<0,\vartheta_4(\alpha)\varphi_1^2(\alpha)<\varphi_2(\alpha)<\vartheta_3(\alpha)\varphi_1^2(\alpha),\\
&=-\vartheta^-_1(\alpha)~~\qquad\qquad\qquad &&if~\varphi_1(\alpha)<0,\varphi_2(\alpha)=\vartheta_4(\alpha)\varphi_1^2(\alpha),\\
&\in \mathcal{I}_3~~\qquad\qquad\qquad  &&if~\varphi_1(\alpha)<0,0<\varphi_2(\alpha)<\vartheta_4(\alpha)\varphi_1^2(\alpha);
\end{aligned}
\right.\\
&x^\star_-\left\{
\begin{aligned}
&\in \mathcal{I}_1~~\qquad\qquad\qquad &&if~\varphi_1(\alpha)>0,0<\varphi_2(\alpha)<\vartheta_4(\alpha)\varphi_1^2(\alpha),\\
&=-\vartheta^-_1(\alpha)~~\qquad\qquad\qquad &&if~\varphi_1(\alpha)>0,\varphi_2(\alpha)=\vartheta_4(\alpha)\varphi_1^2(\alpha),\\
&\in \mathcal{I}_2~~\qquad\qquad\qquad  &&if~\varphi_1(\alpha)>0,\vartheta_4(\alpha)\varphi_1^2(\alpha)<\varphi_2(\alpha)<\vartheta_3(\alpha)\varphi_1^2(\alpha),\\
&=-\vartheta^+_1(\alpha)~~\qquad\qquad\qquad &&if~\varphi_1(\alpha)>0,\varphi_2(\alpha)=\vartheta_3(\alpha)\varphi_1^2(\alpha),\\
&\in \mathcal{I}_3~~\qquad\qquad\qquad  &&if~\varphi_1(\alpha)>0,\varphi_2(\alpha)>\vartheta_3(\alpha)\varphi_1^2(\alpha)~or~\varphi_1(\alpha)\le0,\varphi_2(\alpha)>0.
\end{aligned}
\right.
\end{aligned}
$$
\end{lm}

\begin{proof}
If $\varphi_1(\alpha)=0$, then
$$
\frac{\partial\mathcal{T}_2(-\vartheta^-_1(\alpha);\alpha)}{\partial x}=\sigma^+_x(-\vartheta^-_1(\alpha),0;\alpha)+\sigma^-_x(\vartheta^-_1(\alpha),0;\alpha)
=\sigma^+_x(-\vartheta^-_1(\alpha),0;\alpha)+\sigma^-_x(\vartheta^+_1(\alpha),0;\alpha)
=0
$$
for $\alpha\in U_5\cap U_6$ by (\ref{urvnjfhp}), (\ref{qq485dscs}), (\ref{95vvcs})  and the fact that $T^-_u=(-\vartheta^-_1(\alpha),0)$ and $T^+_u=(\vartheta^+_1(\alpha),0)$ are folds of the upper subsystem of $(\ref{Z2system})$ (see Lemma~\ref{dvn23fs}). Besides, from the proof of Lemma~\ref{458vsvsdf},
$$\frac{\partial\mathcal{T}_2(-\vartheta_2(\alpha);\alpha)}{\partial x}=
\frac{\partial\mathcal{T}_2(-a-\hat\upsilon(\alpha);\alpha)}{\partial x}=\frac{\partial\mathcal{T}_2\circ J^{-1}(0;\alpha)}{\partial\hat x}=0$$
for $\alpha\in U_6$, and $x=-\vartheta_2(\alpha)$ is the unique zero of $\partial\mathcal{T}_2(x;\alpha)/\partial x$ satisfying $\vartheta_2(0)=a$.
Thus
$$\vartheta_2(\alpha)=\vartheta^-_1(\alpha)=\vartheta^+_1(\alpha)\qquad {\rm if}~~\varphi_1(\alpha)=0,$$
for $\alpha\in U_5\cap U_6$.
Since we are assuming that $(\mu_1,\mu_2,\cdots,\mu_m)\ne0$, this, together with (\ref{4scsrve4}), (\ref{23vdf24sff}), (\ref{qq485dscs}), (\ref{ew8csdr}), {\bf (H3)} and Lemma~\ref{eiowu45cd}, implies that there is a neighborhood $U_7\subset U_5\cap U_6$ and two smooth functions
\begin{equation}\label{48vafosadcd}
\hat\vartheta_3(\alpha)=\frac{g^+_x(-a,0;0)\lambda(0)}{\Delta}+\mathcal{O}(\|\alpha\|),\qquad
\hat\vartheta_4(\alpha)=\frac{g^+_x(a,0;0)}{\Delta}+\mathcal{O}(\|\alpha\|)
\end{equation}
defined in $U_7$ such that
\begin{equation}\label{48vosadcd}
\begin{aligned}
-\vartheta_2(\alpha)+\vartheta^+_1(\alpha)&=-\sum_{i=1}^m\left(\frac{g^+_{\alpha_i}(a,0;0)}{g^+_x(a,0;0)}+
\frac{g^+_{\alpha_i}(-a,0;0)\lambda(0)+g^+_{\alpha_i}(a,0;0)}{\Delta}\right)\alpha_i+\mathcal{O}(\|\alpha\|^2)\\
&=-\sum_{i=1}^m\frac{g^+_{\alpha_i}(-a,0;0)g^+_x(a,0;0)+g^+_{\alpha_i}(a,0;0)g^+_x(-a,0;0)}
{g^+_x(a,0;0)\Delta}\lambda(0)\alpha_i+\mathcal{O}(\|\alpha\|^2)\\
&=\frac{g^+_x(-a,0;0)\lambda(0)}{\Delta}\varphi_1(\alpha)+\mathcal{O}(\|\alpha\|^2)\\
&=\hat\vartheta_3(\alpha)\varphi_1(\alpha)
\end{aligned}
\end{equation}
and
\begin{equation}\label{w5h4dsc}
\begin{aligned}
-\vartheta_2(\alpha)+\vartheta^-_1(\alpha)&=\sum_{i=1}^m\left(\frac{g^+_{\alpha_i}(-a,0;0)}{g^+_x(-a,0;0)}-
\frac{g^+_{\alpha_i}(-a,0;0)\lambda(0)+g^+_{\alpha_i}(a,0;0)}{\Delta}\right)\alpha_i+\mathcal{O}(\|\alpha\|^2)\\
&=-\sum_{i=1}^m\frac{g^+_{\alpha_i}(-a,0;0)g^+_x(a,0;0)+g^+_{\alpha_i}(a,0;0)g^+_x(-a,0;0)}
{g^+_x(-a,0;0)\Delta}\alpha_i+\mathcal{O}(\|\alpha\|^2)\\
&=\frac{g^+_x(a,0;0)}{\Delta}\varphi_1(\alpha)+\mathcal{O}(\|\alpha\|^2)\\
&=\hat\vartheta_4(\alpha)\varphi_1(\alpha)
\end{aligned}
\end{equation}
for $\alpha\in U_7$.

Consider the zero $x=x^\star_0=-\vartheta_2(\alpha)$ existing for $\varphi_2(\alpha)=0$. Note that
\begin{equation}\label{490kvsd}
g^+_x(-a,0;0)>0,\qquad g^+_x(a,0;0)>0.
\end{equation}
from {\bf (H1)}, {\bf(H2)}$^\prime$ and $f^+(-a,0;0)>0$. Then it follows directly from (\ref{48vosadcd}) and (\ref{w5h4dsc})
that the location of $x^\star_0$ is the one as stated in this lemma, due to $\Delta>0$, $\lambda(0)>0$ and (\ref{490kvsd}).

Consider the zero $x=x^\star_+=\sqrt{\varphi_2(\alpha)}-\vartheta_2(\alpha)$ existing for $\varphi_2(\alpha)>0$.
By (\ref{48vosadcd}) and (\ref{w5h4dsc}),
\begin{equation}\label{9859dscsdc}
x^\star_++\vartheta^+_1(\alpha)=\sqrt{\varphi_2(\alpha)}-\vartheta_2(\alpha)+\vartheta^+_1(\alpha)=\sqrt{\varphi_2(\alpha)}+\hat\vartheta_3(\alpha)\varphi_1(\alpha)
\end{equation}
and
\begin{equation}\label{9asf859dscsdc}
x^\star_++\vartheta^-_1(\alpha)=\sqrt{\varphi_2(\alpha)}-\vartheta_2(\alpha)+\vartheta^-_1(\alpha)=\sqrt{\varphi_2(\alpha)}
+\hat\vartheta_4(\alpha)\varphi_1(\alpha).
\end{equation}
Take
$
\vartheta_3(\alpha):=\hat\vartheta_3^2(\alpha), \vartheta_4(\alpha):=\hat\vartheta_4^2(\alpha)
$
for $\alpha\in U_7$. Clearly, they satisfy (\ref{4389vj}) by {\bf (H3)} and (\ref{48vafosadcd}).
It follows from (\ref{48vafosadcd}), (\ref{490kvsd}), (\ref{9859dscsdc}), $\Delta>0$ and $\lambda(0)>0$ that $x^\star_++\vartheta^+_1(\alpha)>0$ if $\varphi_1(\alpha)\ge0$, while if $\varphi_1(\alpha)<0$, $x^\star_++\vartheta^+_1(\alpha)>0$ (resp. $=0, <0$) for $\varphi_2(\alpha)-\vartheta_3(\alpha)\varphi_1^2(\alpha)>0$ (resp. $=0, <0$). Moreover, it follows from (\ref{48vafosadcd}), (\ref{490kvsd}), (\ref{9asf859dscsdc}), $\Delta>0$ and $\lambda(0)>0$ that $x^\star_++\vartheta^-_1(\alpha)>0$ if $\varphi_1(\alpha)\ge0$, while if $\varphi_1(\alpha)<0$, $x^\star_++\vartheta^-_1(\alpha)>0$ (resp. $=0, <0$) for $\varphi_2(\alpha)-\vartheta_4(\alpha)\varphi_1^2(\alpha)>0$ (resp. $=0, <0$). As a result, these arguments yield the location of $x^\star_+$ as stated in this lemma, due to $\vartheta_3(\alpha)>\vartheta_4(\alpha)$ from $\Delta>0$, $\lambda(0)>0$ and (\ref{490kvsd}).

Consider the zero $x=x^\star_-=-\sqrt{\varphi_2(\alpha)}-\vartheta_2(\alpha)$ existing for $\varphi_2(\alpha)>0$.
By (\ref{48vosadcd}) and (\ref{w5h4dsc}),
\begin{equation}\label{98ad59dscsdc}
x^\star_-+\vartheta^+_1(\alpha)=-\sqrt{\varphi_2(\alpha)}-\vartheta_2(\alpha)+\vartheta^+_1(\alpha)=-\sqrt{\varphi_2(\alpha)}
+\hat\vartheta_3(\alpha)\varphi_1(\alpha)
\end{equation}
and
\begin{equation}\label{98af59dscsdc}
x^\star_-+\vartheta^-_1(\alpha)=-\sqrt{\varphi_2(\alpha)}-\vartheta_2(\alpha)+\vartheta^-_1(\alpha)=-\sqrt{\varphi_2(\alpha)}
+\hat\vartheta_4(\alpha)\varphi_1(\alpha).
\end{equation}
It follows from (\ref{48vafosadcd}), (\ref{490kvsd}), (\ref{98ad59dscsdc}), $\Delta>0$ and $\lambda(0)>0$ that $x^\star_-+\vartheta^+_1(\alpha)<0$ if $\varphi_1(\alpha)\le0$, while if $\varphi_1(\alpha)>0$,  $x^\star_-+\vartheta^+_1(\alpha)>0$ (resp. $=0, <0$) for $\varphi_2(\alpha)-\vartheta_3(\alpha)\varphi_1^2(\alpha)<0$ (resp. $=0, >0$). Moreover, it follows from (\ref{48vafosadcd}), (\ref{490kvsd}), (\ref{98af59dscsdc}), $\Delta>0$ and $\lambda(0)>0$ that $x^\star_-+\vartheta^-_1(\alpha)<0$ if $\varphi_1(\alpha)\le0$, while if $\varphi_1(\alpha)>0$, $x^\star_-+\vartheta^-_1(\alpha)>0$ (resp. $=0, <0$) for $\varphi_2(\alpha)-\vartheta_4(\alpha)\varphi_1^2(\alpha)<0$ (resp. $=0, >0$). As a result, these arguments yield the location of $x^\star_-$ as stated in this lemma, due to $\vartheta_3(\alpha)>\vartheta_4(\alpha)$ again.
\end{proof}

In addition to crossing cycles and critical crossing cycles, it is possible for system (\ref{Z2system}) to bifurcate sliding cycles and various connections from $\Gamma_0$. To identify these dynamical behaviors, we need the following lemma.

\begin{lm}\label{3829798jcdf}
Consider the maps $\sigma^\pm(x,y;\alpha)$ for $\|(x\pm a,y)\|<\varepsilon_1$ and $\alpha\in U_7$ given in Lemma~\ref{cnkrcdwerc}, $\vartheta^\pm_1(\alpha)$ and $\varpi(\alpha)$ for $\alpha\in U_7$ given in Lemma~\ref{dvn23fs}. Under the assumption of Theorem~\ref{codim-2-fold-bifur1} and $f^+(-a,0;0)>0$, if $(\mu_1,\mu_2,\cdots,\mu_m)\ne0$ and $(\kappa_1,\kappa_2,\cdots,\kappa_m)\ne0$, there exists a neighborhood $U_8\subset U_7$ of $\alpha=0$ and smooth functions
\begin{equation}\label{4i57vjf}
\begin{aligned}
\vartheta_5(\alpha)&=\left(\frac{g^+_x(a,0;0)}{\Delta}\right)^2+\frac{g^+_x(a,0;0)}{\Delta}
+\mathcal{O}(\|\alpha\|), \\
\vartheta_6(\alpha)&=\left(\frac{g^+_x(a,0;0)}{\Delta}\right)^2+\frac{(1-(1+r)^2)g^+_x(a,0;0)}{\Delta}
+\mathcal{O}(\|\alpha\|), \\
\vartheta_7(\alpha)
&=\left(1+\frac{g^+_x(a,0;0)}{\Delta}\right)^2+\frac{(r^2-1)\lambda(0)g^+_x(-a,0;0)}{\Delta}+\mathcal{O}(\|\alpha\|)
\end{aligned}
\end{equation}
defined in $U_8$ such that the following statements hold for $\alpha\in U_8$, where
$$r:=-\frac{f^+(-a,0;0)g^+_x(a,0;0)}{g^+_x(-a,0;0)f^+(a,0;0)+g^+_x(a,0;0)f^+(-a,0;0)}\in(-1,0).$$
\begin{itemize}
\item[{\rm(1)}] $\sigma^+(-\vartheta^-_1(\alpha),0;\alpha)-\sigma^-(\vartheta^+_1(\alpha),0;\alpha)>0$ $($resp. $=0, <0)$ if $\varphi_2(\alpha)-\vartheta_5(\alpha)\varphi_1^2(\alpha)>0$ $($resp. $=0, <0)$.
\item[{\rm(2)}] $\sigma^+(-\vartheta^-_1(\alpha),0;\alpha)-\sigma^-(\varpi(\alpha),0;\alpha)>0$ $($resp. $=0, <0)$ if $\varphi_2(\alpha)-\vartheta_6(\alpha)\varphi_1^2(\alpha)>0$ $($resp. $=0, <0)$.
\item[{\rm(3)}] $\sigma^+(-\varpi(\alpha),0;\alpha)-\sigma^-(\vartheta^+_1(\alpha),0;\alpha)>0$ $($resp. $=0, <0)$ if $\varphi_2(\alpha)-\vartheta_7(\alpha)\varphi_1^2(\alpha)>0$ $($resp. $=0, <0)$.
\end{itemize}
\end{lm}

\begin{proof}
For $\alpha\in U_7$ we have known from Lemma~\ref{rut45scd} that $\mathcal{T}_2(-\vartheta^+_1(\alpha);\alpha)=0$ if $\varphi_2(\alpha)-\vartheta_3(\alpha)\varphi_1^2(\alpha)=0$, $\mathcal{T}_2(-\vartheta^-_1(\alpha);\alpha)=0$ if $\varphi_2(\alpha)-\vartheta_4(\alpha)\varphi_1^2(\alpha)=0$, and $\mathcal{T}_2(-\vartheta^+_1(\alpha);\alpha)=\mathcal{T}_2(-\vartheta^-_1(\alpha);\alpha)$ if $\varphi_1(\alpha)=0$. Besides, it follows that $\nabla(\varphi_2(\alpha)-\vartheta_3(\alpha)\varphi_1^2(\alpha))\ne0$ and $\nabla(\varphi_2(\alpha)-\vartheta_4(\alpha)\varphi_1^2(\alpha))\ne0$ from (\ref{qq485dscs}), (\ref{ew8csdr}), (\ref{4389vj}) and $(\kappa_1,\kappa_2,\cdots,\kappa_m)\ne0$. Thus by Lemma~\ref{eiowu45cd} there is a neighborhood $U_{81}\subset U_7$ of $\alpha=0$ and a smooth function $\omega_1(\alpha)$ defined in $U_{81}$ such that
\begin{equation}\label{493t9csd}
\mathcal{T}_2(-\vartheta^+_1(\alpha);\alpha)=\omega_1(\alpha)(\varphi_2(\alpha)-\vartheta_3(\alpha)\varphi_1^2(\alpha)),\quad
\mathcal{T}_2(-\vartheta^-_1(\alpha);\alpha)=\omega_1(\alpha)(\varphi_2(\alpha)-\vartheta_4(\alpha)\varphi_1^2(\alpha))
\end{equation}
for $\alpha\in U_{81}$. Substituting $\vartheta^\pm_1(\alpha)$ given in (\ref{23vdf24sff}) into (\ref{4fsdcvdf45}) and using (\ref{ew8csdr}),
we get
\begin{equation}\label{4faa93t9csd}
\omega_1(\alpha)=-\frac{\lambda^-(0)\Delta}{2g^+(b,c;0)}+\mathcal{O}(\|\alpha\|)
\end{equation}
from (\ref{493t9csd}).

By (\ref{urvnjfhp}), (\ref{qq485dscs}) and the fact that $T^-_u=(-\vartheta^-_1(\alpha),0)$ and $T^+_u=(\vartheta^+_1(\alpha),0)$ are folds of the upper subsystem of $(\ref{Z2system})$ (see  Lemma~\ref{dvn23fs}), we obtain
$\sigma^+_x(-\vartheta^+_1(\alpha),0;\alpha)=\sigma^-_x(\vartheta^-_1(\alpha),0;\alpha)=0$ if $\varphi_1(\alpha)=0.$
Besides, it follows that $\nabla\varphi_1(\alpha)\ne0$ from (\ref{qq485dscs}) and $(\mu_1,\mu_2,\cdots,\mu_m)\ne0$. Thus by Lemma~\ref{eiowu45cd} there is a neighborhood $U_{82}\subset U_7$ of $\alpha=0$ and smooth functions $\omega_2(\alpha)$ and $\omega_3(\alpha)$ defined in $U_{82}$ such that
\begin{equation}\label{4afa93t9csd}
\sigma^+_{x}(-\vartheta^+_1(\alpha),0;\alpha)=\omega_2(\alpha)\varphi_1(\alpha),\qquad \sigma^-_{x}(\vartheta^-_1(\alpha),0;\alpha)=\omega_3(\alpha)\varphi_1(\alpha)
\end{equation}
for $\alpha\in U_{82}$. Note that
$$
\begin{aligned}
\sigma^+_x(x,0;\alpha)&=\sum_{i=1}^m\sigma^+_{x\alpha_i}(-a,0;0)\alpha_i+\mathcal{O}(\|\alpha\|^2)
+\left(\sigma^+_{xx}(-a,0;0)+\mathcal{O}(\|\alpha\|)\right)(x+a)+\mathcal{O}((x+a)^2)\\
&=\sum_{i=1}^m\frac{g^+_{\alpha_i}(-a,0;0)\lambda^+(0)}{g^+(b,c;0)}\alpha_i+\mathcal{O}(\|\alpha\|^2)
+\left(\frac{g^+_x(-a,0;0)\lambda^+(0)}{g^+(b,c;0)}+\mathcal{O}(\|\alpha\|)\right)(x+a)+\mathcal{O}((x+a)^2)\\
\end{aligned}
$$
and
$$
\begin{aligned}
\sigma^-_x(x,0;\alpha)&=\sum_{i=1}^m\sigma^-_{x\alpha_i}(a,0;0)\alpha_i+\mathcal{O}(\|\alpha\|^2)
+\left(\sigma^-_{xx}(a,0;0)+\mathcal{O}(\|\alpha\|)\right)(x-a)+\mathcal{O}((x-a)^2)\\
&=\sum_{i=1}^m\frac{g^+_{\alpha_i}(a,0;0)\lambda^-(0)}{g^+(b,c;0)}\alpha_i+\mathcal{O}(\|\alpha\|^2)
+\left(\frac{g^+_x(a,0;0)\lambda^-(0)}{g^+(b,c;0)}+\mathcal{O}(\|\alpha\|)\right)(x-a)+\mathcal{O}((x-a)^2)
\end{aligned}
$$
by Lemma~\ref{Transipero}.
Substituting the expansion of $\vartheta^+_1(\alpha)$ (resp. $\vartheta^-_1(\alpha)$) given in (\ref{23vdf24sff}) into the expansion of $\sigma^+_x(x,0;\alpha)$ (resp. $\sigma^-_x(x,0;\alpha)$) and using (\ref{4scsrve4}) and (\ref{qq485dscs}), we get
\begin{equation}\label{493t9csdadd}
\omega_2(\alpha)=-\frac{\lambda^+(0)g^+_x(-a,0;0)}{g^+(b,c;0)}+\mathcal{O}(\|\alpha\|),\qquad \omega_3(\alpha)=-\frac{\lambda^-(0)g^+_x(a,0;0)}{g^+(b,c;0)}+\mathcal{O}(\|\alpha\|)
\end{equation}
from (\ref{4afa93t9csd}).

By (\ref{95vvcs}), (\ref{493t9csd}) and (\ref{4afa93t9csd}),
\begin{equation}\label{4839dssdc}
\begin{aligned}
&~\sigma^+(-\vartheta^-_1(\alpha),0;\alpha)-\sigma^-(\vartheta^+_1(\alpha),0;\alpha)\\
=&~\sigma^+(-\vartheta^-_1(\alpha),0;\alpha)-\sigma^-(\vartheta^-_1(\alpha)+\varphi_1(\alpha),0;\alpha)\\
=&~\mathcal{T}_2(-\vartheta^-_1(\alpha),0;\alpha)-\sigma^-_x(\vartheta^-_1(\alpha),0;\alpha)\varphi_1(\alpha)
-\frac{1}{2}\sigma^-_{xx}(\vartheta^-_1(\alpha),0;\alpha)\varphi_1^2(\alpha)+\mathcal{O}(\varphi^3_1(\alpha))\\
=&~\omega_1(\alpha)(\varphi_2(\alpha)-\vartheta_4(\alpha)\varphi_1^2(\alpha))-\omega_3(\alpha)\varphi_1^2(\alpha)
-\frac{1}{2}\sigma^-_{xx}(\vartheta^-_1(\alpha),0;\alpha)\varphi_1^2(\alpha)+\mathcal{O}(\varphi^3_1(\alpha))\\
=&~\omega_1(\alpha)(\varphi_2(\alpha)-\vartheta_5(\alpha)\varphi_1^2(\alpha))
\end{aligned}
\end{equation}
for $\alpha\in U_{81}\cap U_{82}$, where
$$\vartheta_5(\alpha):=\vartheta_4(\alpha)+
\frac{\omega_3(\alpha)+\sigma^-_{xx}(\vartheta^-_1(\alpha),0;\alpha)/2+\mathcal{O}(\varphi_1(\alpha))}{\omega_1(\alpha)}$$
and it can be written as the expansion in (\ref{4i57vjf}) by (\ref{qq485dscs}), (\ref{4389vj}), (\ref{4faa93t9csd}), (\ref{493t9csdadd}) and the fact
\begin{equation}\label{32vsdvsd}
\sigma^-_{xx}(\vartheta^-_1(\alpha),0;\alpha)=\sigma^-_{xx}(a,0;0)+\mathcal{O}(\|\alpha\|)=\frac{g^+_x(a,0;0)\lambda^-(0)}{g^+(b,c;0)}+\mathcal{O}(\|\alpha\|)\\
\end{equation}
from Lemma~\ref{Transipero}.

Consider the function $\varpi(\alpha)$ for $\alpha\in U_7$. Note that
$\varpi(\alpha)=\vartheta^-_1(\alpha)=\vartheta^+_1(\alpha)$ if $\varphi_1(\alpha)=0,$
and $\nabla\varphi_1(\alpha)\ne0$ from (\ref{qq485dscs}) and $(\mu_1,\mu_2,\cdots,\mu_m)\ne0$.
Then by Lemma~\ref{eiowu45cd} there is a neighborhood $U_{83}\subset U_7$ of $\alpha=0$ and a smooth function $\omega_4(\alpha)$ defined in $U_{83}$ such that
\begin{equation}\label{3827rcsdvs}
-\varpi(\alpha)+\vartheta^-_1(\alpha)=\omega_4(\alpha)\varphi_1(\alpha),\qquad -\varpi(\alpha)+\vartheta^+_1(\alpha)=(1+\omega_4(\alpha))\varphi_1(\alpha)
\end{equation}
for $\alpha\in U_{83}$. Moreover, it follows from (\ref{4scsrve4}), (\ref{23vdf24sff}), (\ref{qq485dscs}) and (\ref{943ce23dew3dsf}) that
\begin{equation}\label{38dad27rcsdvs}
\omega_4(\alpha)=r+\mathcal{O}(\|\alpha\|).
\end{equation}
By (\ref{95vvcs}), (\ref{493t9csd}), (\ref{4afa93t9csd}) and (\ref{3827rcsdvs}),
\begin{equation}\label{woioks3453}
\begin{aligned}
\sigma^+(-\vartheta^-_1(\alpha),0;\alpha)-\sigma^-(\varpi(\alpha),0;\alpha)
=&~\sigma^+(-\vartheta^-_1(\alpha),0;\alpha)-\sigma^-(\vartheta^-_1(\alpha)-\omega_4(\alpha)\varphi_1(\alpha),0;\alpha)\\
=&~\mathcal{T}_2(-\vartheta^-_1(\alpha),0;\alpha)+\sigma^-_x(\vartheta^-_1(\alpha),0;\alpha)\omega_4(\alpha)\varphi_1(\alpha)\\
&-\frac{1}{2}\sigma^-_{xx}(\vartheta^-_1(\alpha),0;\alpha)\omega_4^2(\alpha)\varphi_1^2(\alpha)+\mathcal{O}(\varphi^3_1(\alpha))\\
=&~\omega_1(\alpha)(\varphi_2(\alpha)-\vartheta_4(\alpha)\varphi_1^2(\alpha))+\omega_3(\alpha)\omega_4(\alpha)\varphi_1^2(\alpha)\\
&-\frac{1}{2}\sigma^-_{xx}(\vartheta^-_1(\alpha),0;\alpha)\omega_4^2(\alpha)\varphi_1^2(\alpha)+\mathcal{O}(\varphi^3_1(\alpha))\\
=&~\omega_1(\alpha)(\varphi_2(\alpha)-\vartheta_6(\alpha)\varphi_1^2(\alpha))
\end{aligned}
\end{equation}
and
\begin{equation}\label{woioks345afafa}
\begin{aligned}
\sigma^+(-\varpi(\alpha),0;\alpha)-\sigma^-(\vartheta^+_1(\alpha),0;\alpha)
=&~\sigma^+(-\vartheta^+_1(\alpha)+(1+\omega_4(\alpha))\varphi_1(\alpha),0;\alpha)-\sigma^-(\vartheta^+_1(\alpha),0;\alpha)\\
=&~\mathcal{T}_2(-\vartheta^+_1(\alpha);\alpha)+\sigma^+_{x}(-\vartheta^+_1(\alpha),0;\alpha)(1+\omega_4(\alpha))\varphi_1(\alpha)\\
&+\frac{1}{2}\sigma^+_{xx}(-\vartheta^+_1(\alpha),0;\alpha)(1+\omega_4(\alpha))^2\varphi_1(\alpha)^2+\mathcal{O}(\varphi_1^3(\alpha))\\
=&~\omega_1(\alpha)(\varphi_2(\alpha)-\vartheta_3(\alpha)\varphi_1^2(\alpha))+\omega_2(\alpha)(1+\omega_4(\alpha))\varphi_1^2(\alpha)\\
&+\frac{1}{2}\sigma^+_{xx}(-\vartheta^+_1(\alpha),0;\alpha)(1+\omega_4(\alpha))^2\varphi_1(\alpha)^2+\mathcal{O}(\varphi_1^3(\alpha))\\
=&~\omega_1(\alpha)(\varphi_2(\alpha)-\vartheta_7(\alpha)\varphi_1^2(\alpha))
\end{aligned}
\end{equation}
for $\alpha\in U_{83}$, where
$$
\begin{aligned}
\vartheta_6(\alpha):=&~\vartheta_4(\alpha)
-\frac{\omega_3(\alpha)\omega_4(\alpha)-\sigma^-_{xx}(\vartheta^-_1(\alpha),0;\alpha)\omega_4^2(\alpha)/2+\mathcal{O}(\varphi_1(\alpha))}{\omega_1(\alpha)},\\
\vartheta_7(\alpha):=&~\vartheta_3(\alpha)
-\frac{\omega_2(\alpha)(1+\omega_4(\alpha))+\sigma^+_{xx}(-\vartheta^+_1(\alpha),0;\alpha)(1+\omega_4(\alpha))^2/2+\mathcal{O}(\varphi_1(\alpha))}{\omega_1(\alpha)}.
\end{aligned}
$$
Note that
$$
\sigma^+_{xx}(-\vartheta^+_1(\alpha),0;\alpha)=\sigma^+_{xx}(-a,0;0)+\mathcal{O}(\|\alpha\|)=\frac{g^+_x(-a,0;0)\lambda^+(0)}{g^+(b,c;0)}+\mathcal{O}(\|\alpha\|)
$$
by Lemma~\ref{Transipero}. This, together with (\ref{qq485dscs}), (\ref{4389vj}), (\ref{4faa93t9csd}), (\ref{493t9csdadd}), (\ref{32vsdvsd}) and (\ref{38dad27rcsdvs}), yields the expansions of $\vartheta_6(\alpha)$ and $\vartheta_7(\alpha)$ given in (\ref{4i57vjf}).

Due to $\lambda^-(0)>0, \Delta>0$ and $g^+(b,c;0)<0$, we can take $U_8\subset U_{81}\cap U_{82}\cap U_{83}$ such that $\omega_1(\alpha)>0$ for $\alpha\in U_8$. Consequently, this lemma holds from (\ref{4839dssdc}), (\ref{woioks3453}) and (\ref{woioks345afafa}).
\end{proof}

Based on Lemmas~\ref{dvn23fs}, \ref{458vsvsdf}, \ref{rut45scd} and \ref{3829798jcdf}, we now give the result on the crossing cycles, critical crossing cycles, sliding cycles and various connections bifurcating from $\Gamma_0$ in the following lemma.
\begin{lm}\label{afa945fvvdf}
Under the assumption of Theorem~\ref{codim-2-fold-bifur1} and $f^+(-a,0;0)>0$, for any sufficiently small annulus $\mathcal{A}$ of $\Gamma_0$ there exists a neighborhood $U\subset U_8$ of $\alpha=0$ such that the following statements hold for $\alpha\in U$.
\begin{itemize}
\item[{\rm(1)}] System $(\ref{Z2system})$ has at most two crossing cycles in $\mathcal{A}$, which do not enclose $T^\pm_u$ and $T^\pm_l$. Moreover, there are exactly two crossing cycles if and only if $\varphi_1(\alpha)>0$ and $0<\varphi_2(\alpha)<\vartheta_4(\alpha)\varphi_1^2(\alpha)$, where the inner $($resp. outer$)$ one is hyperbolic and unstable $($resp. stable$)$. There is exactly one crossing cycles if and only if one of the following conditions:
    \begin{itemize}
    \item[{\rm(1a)}] $\varphi_1(\alpha)>0$ and $\varphi_2(\alpha)=0$, for which the crossing cycle is of multiplicity two and internally unstable;
    \item[{\rm(1b)}] $\varphi_1(\alpha)>0$ and $\varphi_2(\alpha)\ge\vartheta_4(\alpha)\varphi_1^2(\alpha)$ or $\varphi_1(\alpha)=0$ and $\varphi_2(\alpha)>0$ or $\varphi_1(\alpha)<0$ and $\varphi_2(\alpha)>\vartheta_3(\alpha)\varphi_1^2(\alpha)$, for which the crossing cycle is hyperbolic and unstable.
    \end{itemize}
\item[{\rm(2)}] The forward orbit of the upper subsystem of $(\ref{Z2system})$ starting at $T^-_u$ evolves into $\Sigma^+\cap\mathcal{A}$ and cannot return to $\Sigma$ if $\varphi_1(\alpha)\ge0$ and $\varphi_2>\vartheta_5(\alpha)\varphi_1^2(\alpha)$, while if $\varphi_1(\alpha)\ge0$ and $\varphi_2\le\vartheta_5(\alpha)\varphi_1^2(\alpha)$, it returns to $\Sigma$ again at
\begin{itemize}
\item[{\rm(2a)}] $T^+_u$ for $\varphi_1(\alpha)\ge0$ and $\varphi_2(\alpha)=\vartheta_5(\alpha)\varphi_1^2(\alpha)$, giving rise to a tangent-tangent connection. In particular, this orbit together with its $\mathbb{Z}_2$-symmetric counterpart constitutes an internally unstable critical crossing cycle for $\varphi_1(\alpha)=0$ and $\varphi_2(\alpha)=0$.
\item[{\rm(2b)}] a point of $C^+_d$ for $\varphi_1(\alpha)=0$ and $\varphi_2(\alpha)<0$.
\item[{\rm(2c)}] a point of $\{(x,0): \varpi(\alpha)<x<\vartheta^+_1(\alpha)\}\subset S^+_3$ for $\varphi_1(\alpha)>0$ and $\vartheta_6(\alpha)\varphi_1^2(\alpha)<\varphi_2(\alpha)<\vartheta_5(\alpha)\varphi_1^2(\alpha)$, giving rise to a tangent-tangent connection with sliding.
\item[{\rm(2d)}] $E^+_p$ for $\varphi_1(\alpha)>0$ and $\varphi_2(\alpha)=\vartheta_6(\alpha)\varphi_1^2(\alpha)$, giving rise to a tangent-equilibrium connection.
\item[{\rm(2e)}] a point of $\{(x,0): \vartheta^-_1(\alpha)<x<\varpi(\alpha)\}\subset S^+_3$ for $\varphi_1(\alpha)>0$ and $\vartheta_4(\alpha)\varphi_1^2(\alpha)<\varphi_2(\alpha)<\vartheta_6(\alpha)\varphi_1^2(\alpha)$, giving rise to a tangent-tangent connection with sliding. Moreover, this orbit together with its $\mathbb{Z}_2$-symmetric counterpart constitutes a stable sliding cycle.
\item[{\rm(2f)}] $T^+_l$ for $\varphi_1(\alpha)>0$ and $\varphi_2(\alpha)=\vartheta_4(\alpha)\varphi_1^2(\alpha)$, giving rise to a tangent-tangent connection. Moreover, this orbit together with its $\mathbb{Z}_2$-symmetric counterpart constitutes an internally stable critical crossing cycle.
\item[{\rm(2g)}] a point of $C^+_d$ for $\varphi_1(\alpha)>0$ and $\varphi_2(\alpha)<\vartheta_4(\alpha)\varphi_1^2(\alpha)$.
\end{itemize}
\item[{\rm(3)}] The backward orbit of the upper subsystem of $(\ref{Z2system})$ starting at $T^+_u$ evolves into $\Sigma^+\cap\mathcal{A}$ and cannot return to $\Sigma$ if $\varphi_1(\alpha)<0$ and $\varphi_2(\alpha)<\vartheta_5(\alpha)\varphi_1^2(\alpha)$, while if $\varphi_1(\alpha)<0$ and $\varphi_2(\alpha)\ge\vartheta_5(\alpha)\varphi_1^2(\alpha)$, it returns to $\Sigma$ again at
\begin{itemize}
\item[{\rm(3a)}] $T^-_u$ for $\varphi_1(\alpha)<0$ and $\varphi_2(\alpha)=\vartheta_5(\alpha)\varphi_1^2(\alpha)$, giving rise to a tangent-tangent connection.
\item[{\rm(3b)}] a point of $\{(x,0): -\vartheta^-_1(\alpha)<x<-\varpi(\alpha)\}\subset S^-_3$ for $\varphi_1(\alpha)<0$ and $\vartheta_5(\alpha)\varphi_1^2(\alpha)<\varphi_2(\alpha)<\vartheta_7(\alpha)\varphi_1^2(\alpha)$, giving rise to a tangent-tangent connection with sliding.
\item[{\rm(3c)}] $E^-_p$ for $\varphi_1(\alpha)<0$ and $\varphi_2(\alpha)=\vartheta_7(\alpha)\varphi_1^2(\alpha)$, giving rise to a tangent-equilibrium connection.
\item[{\rm(3d)}] a point of $\{(x,0): -\varpi(\alpha)<x<-\vartheta^+_1(\alpha)\}\subset S^-_3$ for $\varphi_1(\alpha)<0$ and $\vartheta_7(\alpha)\varphi_1^2(\alpha)<\varphi_2(\alpha)<\vartheta_3(\alpha)\varphi_1^2(\alpha)$, giving rise to a tangent-tangent connection with sliding. Moreover, this orbit together with its $\mathbb{Z}_2$-symmetric counterpart constitutes an unstable sliding cycle.
\item[{\rm(3e)}] $T^-_l$ for $\varphi_1(\alpha)<0$ and $\varphi_2(\alpha)=\vartheta_3(\alpha)\varphi_1^2(\alpha)$, giving rise to a tangent-tangent connection. Moreover, this orbit together with its $\mathbb{Z}_2$-symmetric counterpart constitutes an internally unstable critical crossing cycle.
\item[{\rm(3f)}] a point of $C^-_u$ for $\varphi_1(\alpha)<0$ and $\varphi_2(\alpha)>\vartheta_3(\alpha)\varphi_1^2(\alpha)$.
\end{itemize}
\end{itemize}
\end{lm}

\begin{proof}
For any sufficiently small annulus $\mathcal{A}$ of $\Gamma_0$, letting $\varepsilon>0$ be the constant such that $\mathcal{A}\cap\Sigma=\{(x,0):|x+a|<\varepsilon\}\cup\{(x,0):|x-a|<\varepsilon\}$, we can ensure $\varepsilon<\min\{\varepsilon_1,\varepsilon_7\}$, where $\varepsilon_1$ and $\varepsilon_7$ are given in Lemma~\ref{cnkrcdwerc} and Lemma~\ref{rut45scd} respectively. Then,
according to the dynamics on $\Sigma$ obtained in Lemma~\ref{dvn23fs} and the definition of the map $\mathcal{T}_2(x;\alpha)$ in (\ref{95vvcs}), there is a neighborhood $U\subset U_8$ of $\alpha=0$ such that for $\alpha\in U$ the crossing cycles of system (\ref{Z2system}) in $\mathcal{A}$ are in one-to-one correspondence with the zeros of $\mathcal{T}_2(x;\alpha)$ in $\{x:\max\{-\vartheta^-_1(\alpha),-\vartheta^+_1(\alpha)\}<x<-a+\varepsilon\}\subset\mathcal{I}_1$. Moreover, the multiplicity and stability of a crossing cycle are the same as the multiplicity and stability of the corresponding zero of $\mathcal{T}_2(x;\alpha)$. As a result, we can get statement (1) from Lemma~\ref{458vsvsdf} and Lemma \ref{rut45scd} directly.

Since {\bf (H1)}, {\bf(H2)}$^\prime$, {\bf(H3)}, $f^+(-a,0;0)>0$ and $\lambda(0)>0$, we get
\begin{equation}\label{3892fcsddsgsdfds}
\vartheta_4(\alpha)<\vartheta_6(\alpha)<\vartheta_5(\alpha)<\vartheta_7(\alpha)<\vartheta_3(\alpha)
\end{equation}
for $\alpha\in U$.
Therefore, it follows from Lemma~\ref{rut45scd} and Lemma~\ref{3829798jcdf} that
$$
\left\{
\begin{aligned}
&\sigma^+(-\vartheta^-_1(\alpha),0;\alpha)>\sigma^-(\vartheta^+_1(\alpha),0;\alpha)  &&~~~~{\rm if}~ \varphi_1(\alpha)=0,\varphi_2(\alpha)>0,\\
&\sigma^+(-\vartheta^-_1(\alpha),0;\alpha)=\sigma^-(\vartheta^+_1(\alpha),0;\alpha)  &&~~~~{\rm if}~ \varphi_1(\alpha)=0,\varphi_2(\alpha)=0,\\
&\sigma^+(-\vartheta^-_1(\alpha),0;\alpha)<\sigma^-(\vartheta^+_1(\alpha),0;\alpha)  &&~~~~{\rm if}~ \varphi_1(\alpha)=0,\varphi_2(\alpha)<0,\\
\end{aligned}
\right.
$$

$$
\left\{
\begin{aligned}
&\sigma^+(-\vartheta^-_1(\alpha),0;\alpha)>\sigma^-(\vartheta^+_1(\alpha),0;\alpha)  &&~~~~{\rm if}~ \varphi_1(\alpha)>0,\varphi_2(\alpha)>\vartheta_5(\alpha)\varphi_1^2(\alpha),\\
&\sigma^+(-\vartheta^-_1(\alpha),0;\alpha)=\sigma^-(\vartheta^+_1(\alpha),0;\alpha)  &&~~~~{\rm if}~ \varphi_1(\alpha)>0,\varphi_2(\alpha)=\vartheta_5(\alpha)\varphi_1^2(\alpha),\\
&\sigma^+(-\vartheta^-_1(\alpha),0;\alpha)\in(\sigma^-(\varpi(\alpha),0;\alpha),\sigma^-(\vartheta^+_1(\alpha),0;\alpha)) &&~~~~ {\rm if}~\varphi_1(\alpha)>0,\\ &~~~~~&&~~~~\vartheta_6(\alpha)\varphi_1^2(\alpha)<\varphi_2(\alpha)<\vartheta_5(\alpha)\varphi_1^2(\alpha),\\
&\sigma^+(-\vartheta^-_1(\alpha),0;\alpha)=\sigma^-(\varpi(\alpha),0;\alpha)\quad &&~~~~ {\rm if}~\varphi_1(\alpha)>0,\varphi_2(\alpha)=\vartheta_6(\alpha)\varphi_1^2(\alpha),\\
&\sigma^+(-\vartheta^-_1(\alpha),0;\alpha)\in(\sigma^-(\vartheta^-_1(\alpha),0;\alpha),\sigma^-(\varpi(\alpha),0;\alpha))\quad &&~~~~ {\rm if}~\varphi_1(\alpha)>0,\\
&~~~~~~~~~~~~~~~~~~~~~~~~~~~~~~~~~~~&&~~~~\vartheta_4(\alpha)\varphi_1^2(\alpha)<\varphi_2(\alpha)<\vartheta_6(\alpha)\varphi_1^2(\alpha),\\
&\sigma^+(-\vartheta^-_1(\alpha),0;\alpha)=\sigma^-(\vartheta^-_1(\alpha),0;\alpha)\quad &&~~~~ {\rm if}~\varphi_1(\alpha)>0,\varphi_2(\alpha)=\vartheta_4(\alpha)\varphi_1^2(\alpha),\\
&\sigma^+(-\vartheta^-_1(\alpha),0;\alpha)<\sigma^-(\vartheta^-_1(\alpha),0;\alpha)\quad &&~~~~ {\rm if}~\varphi_1(\alpha)>0,\varphi_2(\alpha)<\vartheta_4(\alpha)\varphi_1^2(\alpha),\\
\end{aligned}
\right.
$$
and
$$
\left\{
\begin{aligned}
&\sigma^-(\vartheta^+_1(\alpha),0;\alpha)<\sigma^+(-\vartheta^+_1(\alpha),0;\alpha)\qquad &&~~ {\rm if}~\varphi_1(\alpha)<0,\varphi_2(\alpha)>\vartheta_3(\alpha)\varphi_1^2(\alpha),\\
&\sigma^-(\vartheta^+_1(\alpha),0;\alpha)=\sigma^+(-\vartheta^+_1(\alpha),0;\alpha)\quad &&~~ {\rm if}~\varphi_1(\alpha)<0,\varphi_2(\alpha)=\vartheta_3(\alpha)\varphi_1^2(\alpha),\\
&\sigma^-(\vartheta^+_1(\alpha),0;\alpha)\in(\sigma^+(-\vartheta^+_1(\alpha),0;\alpha),\sigma^+(-\varpi(\alpha),0;\alpha))\quad &&~~ {\rm if}~\varphi_1(\alpha)<0,\\
&~~~~~~~~~~~~~~~~~~~~~~~~~~~~~~~~~~~&&~~\vartheta_7(\alpha)\varphi_1^2(\alpha)<\varphi_2(\alpha)<\vartheta_3(\alpha)\varphi_1^2(\alpha),\\
&\sigma^-(\vartheta^+_1(\alpha),0;\alpha)=\sigma^+(-\varpi(\alpha),0;\alpha)\quad &&~~ {\rm if}~\varphi_1(\alpha)<0,\varphi_2(\alpha)=\vartheta_7(\alpha)\varphi_1^2(\alpha),\\
&\sigma^-(\vartheta^+_1(\alpha),0;\alpha)\in(\sigma^+(-\varpi(\alpha),0;\alpha),\sigma^+(-\vartheta^-_1(\alpha),0;\alpha)) &&~~ {\rm if}~\varphi_1(\alpha)<0,\\ &~~~~~&&~~\vartheta_5(\alpha)\varphi_1^2(\alpha)<\varphi_2(\alpha)<\vartheta_7(\alpha)\varphi_1^2(\alpha),\\
&\sigma^-(\vartheta^+_1(\alpha),0;\alpha)=\sigma^+(-\vartheta^-_1(\alpha),0;\alpha)  &&~~{\rm if}~ \varphi_1(\alpha)<0,\varphi_2(\alpha)=\vartheta_5(\alpha)\varphi_1^2(\alpha),\\
&\sigma^-(\vartheta^+_1(\alpha),0;\alpha)>\sigma^+(-\vartheta^-_1(\alpha),0;\alpha)  &&~~{\rm if}~ \varphi_1(\alpha)<0,\varphi_2(\alpha)<\vartheta_5(\alpha)\varphi_1^2(\alpha).
\end{aligned}
\right.
$$
These conclude statements (2) and (3), considering the forward (resp. backward) orbit of the upper subsystem of $(\ref{Z2system})$ starting at $T^-_u$ (resp. $T^+_u$) and recalling the definitions of $\sigma^\pm(x,y;\alpha)$ in Lemma~\ref{cnkrcdwerc} and the dynamics on $\Sigma$ stated in Lemma~\ref{dvn23fs}.
\end{proof}

Having these preliminaries, we now can give the proof of Theorem~\ref{codim-2-fold-bifur1}.

\begin{proof}[{\bf Proof of Theorem~\ref{codim-2-fold-bifur1}}]
We only prove Theorem~\ref{codim-2-fold-bifur1} for $f^+(-a,0;0)>0$ because of similarity. For any sufficiently small annulus $\mathcal{A}$ of $\Gamma_0$, we take $U$ as the neighborhood of $\alpha=0$ given in Lemma~\ref{afa945fvvdf}, $\varphi_1(\alpha)$ and $\varphi_2(\alpha)$ for $\alpha\in U$ as the functions given in (\ref{qq485dscs}) and (\ref{ew8csdr}) respectively.
If $(\mu_1,\mu_2,\cdots,\mu_m)$ and $(\kappa_1,\kappa_2,\cdots,\kappa_m)$ are linearly independent, there are $j_1,j_2\in\{1,2,\cdots,m\}$ such that
\begin{equation}\label{cfw34wefw123}
\mu_{j_1}\kappa_{j_2}-\kappa_{j_1}\mu_{j_2}\ne0.
\end{equation}
Let
\begin{equation}\label{cdsn34afr123}
\varphi_i(\alpha):=\alpha_{j_i},\quad i=3,4,\cdots,m
\end{equation}
such that $\alpha_{j_i}\ne\alpha_{j_1}$, $\alpha_{j_i}\ne\alpha_{j_2}$ for all $i=3,4,\cdots,m$ and $\alpha_{j_k}\ne\alpha_{j_l}$ for $k\ne l$. Then the Jacobian matrix of
\begin{equation}\label{cdsn3465123}
\beta=(\beta_1,\beta_2,\cdots\beta_m)=(\varphi_1(\alpha),\varphi_2(\alpha),\cdots, \varphi_m(\alpha))
\end{equation}
is nonsingular at $\alpha=0$ from (\ref{qq485dscs}), (\ref{ew8csdr}), (\ref{cfw34wefw123}) and (\ref{cdsn34afr123}). Thus (\ref{cdsn3465123}) is a diffeomorphism from $U$ to its range $V$, where $U$ can be reduced if necessary.
Finally, taking
\begin{eqnarray*}
\begin{aligned}
CS^+:&=\Big\{\beta\in V:\beta_1>0,\beta_2=\vartheta_4(\alpha^{-1}(\beta))\beta_1^2\Big\},\\
SH^+:&=\Big\{\beta\in V:\beta_1>0,\beta_2=\vartheta_6(\alpha^{-1}(\beta))\beta_1^2\Big\},\\
TC^+:&=\Big\{\beta\in V:\beta_1>0,\beta_2=\vartheta_5(\alpha^{-1}(\beta))\beta_1^2\Big\},\\
CS^-:&=\Big\{\beta\in V:\beta_1<0,\beta_2=\vartheta_3(\alpha^{-1}(\beta))\beta_1^2\Big\},\\
SH^-:&=\Big\{\beta\in V:\beta_1<0,\beta_2=\vartheta_7(\alpha^{-1}(\beta))\beta_1^2\Big\},\\
TC^-:&=\Big\{\beta\in V:\beta_1<0,\beta_2=\vartheta_5(\alpha^{-1}(\beta))\beta_1^2\Big\},
\end{aligned}
\end{eqnarray*}
where
$\alpha^{-1}(\beta)$ denotes the inverse of (\ref{cdsn3465123}), and combining (\ref{3892fcsddsgsdfds}), the dynamics on $\Sigma$ and the information of various cycles and connections given in Lemma~\ref{dvn23fs} and Lemma~\ref{afa945fvvdf} respectively, we get that for any $(\beta_3^*,\beta_4^*,\cdots,\beta_m^*)\in V^*$ the bifurcation diagram of system $\left.(\ref{Z2system})\right|_{\mathcal{A}}$ on the hyperplane $(\beta_3,\beta_4,\cdots,\beta_m)=(\beta_3^*,\beta_4^*,\cdots,\beta_m^*)$ is the one shown in Figure~\ref{codim-2bifurdia2folddia} when $f^+(-a,0;0)>0$.
\end{proof}

\section*{Declaration of competing interest}
The authors declare that they have no known competing financial interests or personal relationships that could have appeared
to influence the work reported in this paper.

\section*{Data availability}
No data was used for the research described in the article.

%%%%%%%%%%%%%%%%%%%%%%%%%%%%%%%%%%%%%%%%%%%%%%%%%%%%%%%%%%%%%%%%%

\end{document}